\documentclass[11pt]{article}
\usepackage[utf8]{inputenc} % (!) 1er (ou 2e) package
\usepackage[T1]{fontenc} % (!) 2e (ou 1er) package
\usepackage{amsmath}
\usepackage{amsthm}
\usepackage{amsfonts}
\usepackage{amssymb}
\usepackage{lmodern} % ordre indifférent
\usepackage{tikz}
\usetikzlibrary{arrows}

\usepackage{pgf,tikz,pgfplots}
\usepackage{mathrsfs}

\usepackage[a4paper]{geometry}
\usepackage{cite}
\usepackage[english]{babel}

\newtheorem{theorem}{Theorem}[section]
\newtheorem{lemma}[theorem]{Lemma}

\newtheorem{definition}[theorem]{Definition}

\newtheorem{fact}[theorem]{Fact}

\newcommand{\N}{{\mathbb N}}
\newcommand{\Hyp}{{\mathbb H}}

\newcommand{\Z}{{\mathbb Z}}
\newcommand{\R}{{\mathbb R}}

\newcommand{\Q}{{\mathbb Q}}
\newcommand{\C}{{\mathbb C}}
\newcommand{\T}{{\mathbb T}}

\begin{document}

\title{A short introduction to translation surfaces, Veech surfaces,  and Teichm\H{u}ller dynamics }
\author{  D. Massart}
\date{\today}
\maketitle
\begin{abstract}
We review the different notions about translation surfaces which are necessary to understand McMullen's classification of $\mathrm{GL}_2^+(\R)$-orbit closures in genus two. In Section \ref{definitions} we recall the different definitions of a translation surface, in increasing order of abstraction, starting with cutting and pasting plane polygons, ending with Abelian differentials. In Section \ref{moduli space} we define the moduli space of translation surfaces and explain  its stratification by the type of zeroes of the Abelian differential, the local coordinates given by the relative periods, its relationship with the moduli space of complex structures and the Teichm\H{u}ller geodesic flow. In Part \ref{gl2} we introduce the $\mathrm{GL}_2^+(\R)$-action, and define the related notions of Veech group, Teichm\H{u}ller disk, and Veech surface. In Section \ref{genus2} we explain how McMullen classifies $\mathrm{GL}_2^+(\R)$-orbit closures in genus $2$: you have orbit closures of dimension $1$ (Veech surfaces, of which a complete list is given), $2$ (Hilbert modular surfaces, of which again a complete list is given), and $3$ (the whole moduli space of complex structures). In the last section we review some recent progress in higher genus.	
\end{abstract}

keywords: translation surface,  Veech surface, Abelian differential, moduli space


\tableofcontents
\part{What is a translation surface? }
\section{Introduction}
Translation surfaces are a great pedagogical tool at almost every level of education. At the elementary, or secondary, level, you are impeded by the lack of embedding into 3-space, but billiards provide plenty of access points into  geometry. At the undergraduate level, they help in topology, to explain quotient spaces, or in differential geometry, to explain atlases and transition maps, or in complex analysis, with meromorphic functions.
At the graduate level, the doors of abstraction burst open, and in a few light steps you get to the enchanted gardens of moduli spaces, Teichm\H{u}ller theory, and elementary algebraic geometry.  The beauty of the subject lies in the interplay between the basic (cutting and pasting little pieces of paper) and the abstract (moduli spaces of Riemann surfaces).

There are already many wonderful introductions to translation surfaces (for instance, \cite{Zorich}, \cite{FM}, \cite{HS}, \cite{Moller}, or \cite{Wright}), and this one is in no way meant as a substitute to any of them, but it could be used as a stepping stone. It is specifically aimed at beginning graduate students. %The layout is essentially that of \cite{HS}.
Mostly, it consists of what, given two (or five) hours and a blackboard,  I would tell a student looking for a thesis subject and eager to know what translation surfaces are about.
 
 I have tried to include (sketches of) proofs of  two types of results: first, those that are too elementary to be included in the aforementionned introductory papers, but can still cause a good deal of head-scratching; and some, not uncommon in the subject, that  are of the ``takes genius to see it, but not that hard once you know it" variety.

I thank Erwan Lanneau for many  useful conversations, the editor of the present volume, and Pablo Montealegre for their careful reading of the manuscript, and Smail Cheboui for letting me use some of the beautiful drawings of \cite{Smail} (the ugly ones are mine).
%%%%%%%%%%%%%%%%%%%%%%%%%%%%%%
\section{Three different definitions of a translation surface}\label{definitions}
%%%%%%%%%%%%%%%%%%%%%%%%%%%%%

%
\subsection{The most hands-on definition: polygons with identifications}\label{def1}
Let $P$ be a polygon in the Euclidean plane, not necessarily convex, maybe not even connected,  but with its sides pairwise parallel and of equal length. Glue each side to a parallel side of equal length. There are several ways of doing this, so let's be a bit more careful.

First, let us make the convention that we orient the sides of a polygon so that the interior of the polygon lies to the left.

Now let us identify the sides in pairs, in such a way that the resulting surface is orientable (see Figure \ref{toto} for a non-orientable example). One way to think of it is to imagine that the upside of the polygon is painted red, while the downside is painted black. The surface is orientable if it has one red side, and one black side. This means that each edge is glued to another edge in such a way that the colors match, so the arrows on the edges do not match. 

Let us assume, furthermore,  that we glue each edge to a parallel edge of equal length. Let $n$ be the number of edges. Thus there exists some permutation $\sigma$  of $\left\lbrace 1,\ldots, n \right\rbrace $ such that $v_{\sigma (i)}= \pm v_i$, for all $i=1,\ldots, n$. Since we must oppose the arrows when gluing, there are two possible cases: either $v_{\sigma (i)}= - v_i$, in which case the two edges are glued by translation (see Figure \ref{toto}, top right), or $v_{\sigma (i)}=  v_i$, in which case the edges are glued with a half-turn (see Figure \ref{toto}, bottom).

Formally, you define an equivalence relation $\sim$ on $P$, and you are considering the quotient space $X=P/\sim$. 
Although the polygon itself may not be connected, we require the resulting quotient space to be connected, just because if it is not, we study its connected components separately. 
What you get is a compact, orientable manifold of dimension two (a surface, for short).
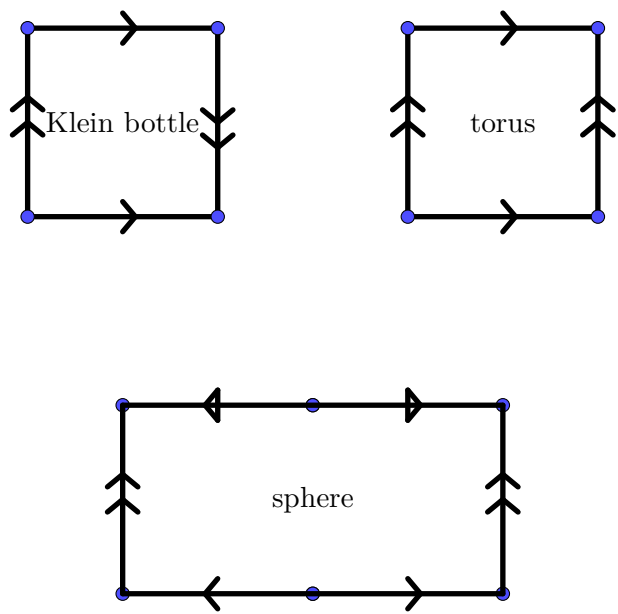
\begin{figure}[h!]
	\begin{center}
\definecolor{ududff}{rgb}{0.30196078431372547,0.30196078431372547,1}
\definecolor{uuuuuu}{rgb}{0.26666666666666666,0.26666666666666666,0.26666666666666666}
\begin{tikzpicture}[line cap=round,line join=round,>=triangle 45,x=2.5cm,y=2.5cm]
\clip(-3.5,-2.3) rectangle (3.5,2);
\draw [line width=2pt] (0,0)-- (1,0);
\draw [line width=2pt] (0.5666666666666672,0) -- (0.5,-0.08);
\draw [line width=2pt] (0.5666666666666672,0) -- (0.5,0.08);
\draw [line width=2pt] (0,1)-- (1,1);
\draw [line width=2pt] (0.5666666666666672,1) -- (0.5,0.92);
\draw [line width=2pt] (0.5666666666666672,1) -- (0.5,1.08);
\draw [line width=2pt] (0,0)-- (0,1);
\draw [line width=2pt] (0,0.6333333333333333) -- (0.08,0.5666666666666667);
\draw [line width=2pt] (0,0.6333333333333333) -- (-0.08,0.5666666666666667);
\draw [line width=2pt] (0,0.5) -- (0.08,0.4333333333333332);
\draw [line width=2pt] (0,0.5) -- (-0.08,0.4333333333333332);
\draw [line width=2pt] (1,0)-- (1,1);
\draw [line width=2pt] (1,0.6333333333333333) -- (1.08,0.5666666666666667);
\draw [line width=2pt] (1,0.6333333333333333) -- (0.92,0.5666666666666667);
\draw [line width=2pt] (1,0.5) -- (1.08,0.4333333333333332);
\draw [line width=2pt] (1,0.5) -- (0.92,0.4333333333333332);
\draw [line width=2pt] (-2,0)-- (-1,0);
\draw [line width=2pt] (-1.4333333333333331,0) -- (-1.5,-0.08);
\draw [line width=2pt] (-1.4333333333333331,0) -- (-1.5,0.08);
\draw [line width=2pt] (-2,1)-- (-1,1);
\draw [line width=2pt] (-1.4333333333333331,1) -- (-1.5,0.92);
\draw [line width=2pt] (-1.4333333333333331,1) -- (-1.5,1.08);
\draw [line width=2pt] (-2,0)-- (-2,1);
\draw [line width=2pt] (-2,0.6333333333333333) -- (-1.92,0.5666666666666667);
\draw [line width=2pt] (-2,0.6333333333333333) -- (-2.08,0.5666666666666667);
\draw [line width=2pt] (-2,0.5) -- (-1.92,0.4333333333333332);
\draw [line width=2pt] (-2,0.5) -- (-2.08,0.4333333333333332);
\draw [line width=2pt] (-1,1)-- (-1,0);
\draw [line width=2pt] (-1,0.3666666666666665) -- (-1.08,0.4333333333333332);
\draw [line width=2pt] (-1,0.3666666666666665) -- (-0.92,0.4333333333333332);
\draw [line width=2pt] (-1,0.5) -- (-1.08,0.5666666666666667);
\draw [line width=2pt] (-1,0.5) -- (-0.92,0.5666666666666667);
\begin{scriptsize}
%\draw [fill=uuuuuu] (0,0) circle (2.5pt);
\draw [fill=ududff] (0,0) circle (2.5pt);
\draw [fill=ududff] (1,0) circle (2.5pt);
\draw [fill=ududff] (1,1) circle (2.5pt);
\draw [fill=ududff] (0,1) circle (2.5pt);
\draw [fill=ududff] (-2,0) circle (2.5pt);
\draw [fill=ududff] (-1,0) circle (2.5pt);
\draw [fill=ududff] (-2,1) circle (2.5pt);
\draw [fill=ududff] (-1,1) circle (2.5pt);

%\draw [fill=uuuuuu] (-0.5,-2) circle (2.5pt);
\draw [fill=ududff] (-0.5,-2) circle (2.5pt);
\draw [fill=ududff] (0.5,-2) circle (2.5pt);
\draw [fill=ududff] (0.5,-1) circle (2.5pt);
\draw [fill=ududff] (-0.5,-1) circle (2.5pt);
\draw [fill=ududff] (-1.5,-2) circle (2.5pt);
\draw [fill=ududff] (-0.5,-2) circle (2.5pt);
\draw [fill=ududff] (-1.5,-1) circle (2.5pt);
\draw [fill=ududff] (-0.5,-1) circle (2.5pt);

\draw [line width=2pt] (-1.5,-1)-- (0.5,-1);
\draw [line width=2pt] (0.0666666666666672,-2) -- (0,-1.92);
\draw [line width=2pt] (0.0666666666666672,-2) -- (0,-2.08);

\draw [line width=2pt] (-1.0666666666666672,-2) -- (-1,-1.92);
\draw [line width=2pt] (-1.0666666666666672,-2) -- (-1,-2.08);

\draw [line width=2pt] (-1.5,-2)-- (0.5,-2);
\draw [line width=2pt] (0.0666666666666672,-1) -- (0,-0.92);
\draw [line width=2pt] (0.0666666666666672,-1) -- (0,-1.08);
\draw [line width=2pt]  (0,-0.92)-- (0,-1.08);

\draw [line width=2pt] (-1.0666666666666672,-1) -- (-1,-0.92);
\draw [line width=2pt] (-1.0666666666666672,-1) -- (-1,-1.08);
\draw [line width=2pt]  (-1,-0.92)-- (-1,-1.08);

\draw [line width=2pt] (-1.5,-1)-- (-1.5,-2);
\draw [line width=2pt] (-1.5,-1.3666666666666665) -- (-1.58,-1.4333333333333332);
\draw [line width=2pt] (-1.5,-1.3666666666666665) -- (-1.42,-1.4333333333333332);
\draw [line width=2pt] (-1.5,-1.5) -- (-1.58,-1.5666666666666667);
\draw [line width=2pt] (-1.5,-1.5) -- (-1.42,-1.5666666666666667);

\draw [line width=2pt] (0.5,-1)-- (0.5,-2);
\draw [line width=2pt] (0.5,-1.3666666666666665) -- (2-1.58,-1.4333333333333332);
\draw [line width=2pt] (0.5,-1.3666666666666665) -- (2-1.42,-1.4333333333333332);
\draw [line width=2pt] (0.5,-1.5) -- (2-1.58,-1.5666666666666667);
\draw [line width=2pt] (0.5,-1.5) -- (2-1.42,-1.5666666666666667);

\end{scriptsize}
\draw[color=black] (0.5,0.5) node {torus};
\draw[color=black] (-0.5,-1.5) node {sphere};
\draw[color=black] (-1.5,0.5) node {Klein bottle};
\end{tikzpicture}
	\end{center}
\caption{ Here the arrows indicate the gluing, not the orientation of the boundary, so arrows must match under the gluing. Note that in the case of the sphere, the edges are not glued by translation, but with a half-turn. }\label{toto}
\end{figure}

Here is why the quotient space is a manifold. 

What we need is to find, for each $x \in X$, a neighborhood $U_x$ of $x$ in $X$, and a chart $\phi_x  : U_x \longrightarrow \C$, such that 
for any $x,y$ in $X$, if $U_x \cap U_y \neq \emptyset$, then $\phi_x \circ \phi_y^{-1}$ preserves whatever structure it is that you want your manifold to come with (if you want a differentiable manifold, you require 
$\phi_x \circ \phi_y^{-1}$  to be differentiable, if you want a complex manifold, you require it to be holomorphic, and so on). 

If $x$ is the image in $X$ of an interior point $\bar{x}$ of $P$, then  $x$ has a neighborhood $U_x$ in $X$ which is the homeomorphic image in $X$ of an open neighborhood $\bar{U}_x$ of $\bar{x}$ in $P$ (hence in $\C$). We take the inverse quotient map $U_x \rightarrow \bar{U}_x$ as the chart $\phi_x$. 

Shorter and somewhat sloppier version: the quotient map, restricted to the interior of $P$, is a homeomorphism. Identify $x$ with its pre-image. Take a neighborhood $U_x$ of $x$ which is contained in the interior of $P$, and take the identity as a chart.

If $x$ is the image in $X$ of an interior point of an edge of $P$, then $x$ has only two pre-images in $P$, and $x$ has a neighborhood in $X$ whose pre-image in $P$ is the reunion of two half-disks of equal radius and parallel diameters. If the two edges are glued by translation, we define a chart by the identity on one of the half-disks, and a translation on the other half-disk.
If the two edges are glued with a half-turn, we define a chart by the identity on one of the half-disks, and a translation composed with a half-turn on the other half-disk.

If $x$ is the image in $X$ of a vertex of $P$, then all pre-images in $P$ of $x$ are vertices  $P_1, \ldots, P_k$ of  $P$, because we glue edges of equal length, and $x$ has a neighborhood $U_x$ in $X$ whose pre-image in $P$  is a union of circular sectors $S_i$, $i=1,\ldots, k$, with radius $r$ and straight boundaries $e_i$ and $f_i$, so that $f_i$ is parallel to $e_{i+1}$, and $f_k$ is parallel to $e_1$. The angle $\theta_i$ of the sector $S_i$ is the angle between two adjacent  (at $P_i$) edges of $P$. Note that since $f_k$ is parallel to $e_1$, the angles $\theta_i$ sum to a multiple of $\pi$, say $p\pi$. Furthermore, if the edges $f_k$ and  $e_1$ are glued by a translation, the angles $\theta_i$ sum to a multiple of $2\pi$. If the edges $f_k$ and  $e_1$ are glued with a half-turn, the angles $\theta_i$ sum to an odd multiple of $\pi$. Denote 
\[
\Theta_i = \sum_{j=1}^{i-1}\theta_j.
\]

We define a chart which takes each $S_i$ to a circular sector with vertex at $0$, by
\[
P_i + \rho \exp i(\theta+\Theta_i) \longmapsto  \rho \exp i\frac{2}{p}(\theta+\Theta_i), \mbox{ for } 0 \leq \rho \leq r, 0 \leq \theta \leq \theta_i .
\]

We say a point in $X$ which is the image of an interior point of $P$ is of type I, a point in $X$ which is the image of an interior point of an edge of $P$ is of type II, and a point in $X$ which is the image of a vertex 
of $P$ is of type III. 

If $x$ and $y$ are both of type I, and   $U_x \cap U_y \neq \emptyset$, then $\phi_x \circ \phi_y^{-1}$ is the identity, which preserves every structure imaginable. 

If $x$ and $y$ are both of type I or II, and   $U_x \cap U_y \neq \emptyset$, then $\phi_x \circ \phi_y^{-1}$ is either  the identity,  or a translation, or a ``half-translation" $z \mapsto -z+c$, all of  which preserve almost every structure imaginable. 

If $x$ is of type III and $y$ is  of type I, II, or III and   $U_x \cap U_y \neq \emptyset$, then $\phi_x \circ \phi_y^{-1}$ is  holomorphic (it is essentially a branch of the $p/2$-th root). If both $x$ and $y$ are of type III, 
we choose $U_x$ and $U_y$ so that $U_x \cap U_y = \emptyset$.

Thus we have defined a complex structure on $X$ (an atlas with holomorphic transition maps), but this does not tell the whole story. Denote $\Sigma= \{ x_1, \ldots, x_k \}$ the set of all images in $X$ of the vertices of $P$, then if all edges are glued by translation, $X \setminus \Sigma$ has  an atlas whose transition maps are translations, whence the name 
\textbf{translation surface\index{translation surface}}. Among the three surfaces in Figure \ref{toto}, only the torus is a translation surface.  

If some pair of edges is glued with a half turn, $X \setminus \Sigma$ has  an atlas whose transition maps take the form $z \mapsto \pm z +c$, in which case $X$ is called a \textbf{half-translation surface\index{half-translation surface}}. The sphere in Figure \ref{toto} is a half-translation surface.  The Klein bottle is not orientable, so it is neither a translation nor a half-translation surface. 

Points of type III in $X$ with a total angle $>2\pi$ will hereafter be called \textbf{singularities}. While being of type III depends on the polygon, being a singularity only depends on the quotient space $X$. %(exemple avec un tore à deux carrés).

The fact that translations and the $z \mapsto -z$ map preserve almost every structure imaginable on the plane (complex, metric, you name it) entails that translation and half-translation surfaces come with a lot of structure, and part of the appeal of this theory is that you can view it from so many different angles.  

\subsection{Main examples of translation surfaces}
Before proceeding further, let us introduce our favorite examples. Probably the first thing that comes to mind after seeing the definition is an even-sided, regular polygon. 

The surface obtained by identifying opposite sides of a square is a torus. 

The surface obtained by identifying opposite sides of a regular hexagon is also a torus, as may be seen by computing the Euler characteristic: one two-cell (the hexagon itself), three edges (one for each pair of opposite edges of the hexagon), and two vertices (if the vertices of the hexagon are cyclically numbered, all even (resp. odd)-numbered vertices are identified into one point), so the  Euler characteristic is zero. 

For $n>1$, the surface obtained by identifying opposite sides of a regular $4n$-gon has genus $n$, with all vertices identified into one point, with a total angle $(4n-2)\pi$.  The surface obtained by identifying opposite sides of a regular $4n+2$-gon has genus $n$, with two points of type III, both with 
angle $2n\pi$.

Another surface of interest is the double $(2n+1)$-gon, which, from the combinatorial viewpoint, is the same as the regular $4n$-gon, so it has  genus $n$, with all vertices identified into one point, with a total angle $(4n-2)\pi$.
See the double pentagon on Figure \ref{doublePentagone}.

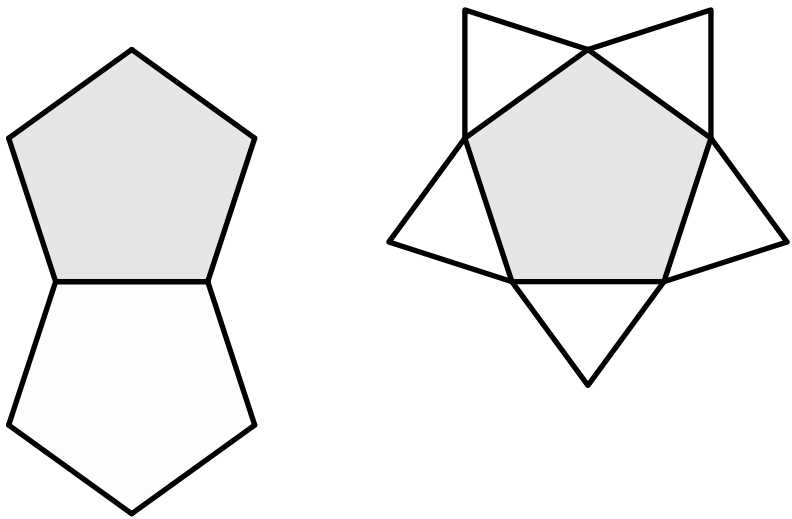
\begin{figure}[h!]
	\begin{center}
\definecolor{uuuuuu}{rgb}{0.26666666666666666,0.26666666666666666,0.26666666666666666}
\begin{tikzpicture}[line cap=round,line join=round,>=triangle 45,x=2cm,y=2cm]
\clip(-2,-2) rectangle (8,2.5);
\fill[line width=2pt,fill=black,fill opacity=0.1] (0,0) -- (1,0) -- (1.3090169943749475,0.9510565162951532) -- (0.5,1.5388417685876266) -- (-0.30901699437494734,0.9510565162951536) -- cycle;
\fill[line width=2pt,fill=black,fill opacity=0.01] (1,0) -- (0,0) -- (-0.30901699437494745,-0.9510565162951532) -- (0.5,-1.5388417685876266) -- (1.3090169943749475,-0.9510565162951536) -- cycle;
\fill[line width=2pt,fill=black,fill opacity=0.1] (3,0) -- (4,0) -- (4.3090169943749475,0.9510565162951532) -- (3.5,1.5388417685876266) -- (2.6909830056250525,0.9510565162951536) -- cycle;
\draw [line width=2pt] (0,0)-- (1,0);
\draw [line width=2pt] (1,0)-- (1.3090169943749475,0.9510565162951532);
\draw [line width=2pt] (1.3090169943749475,0.9510565162951532)-- (0.5,1.5388417685876266);
\draw [line width=2pt] (0.5,1.5388417685876266)-- (-0.30901699437494734,0.9510565162951536);
\draw [line width=2pt] (-0.30901699437494734,0.9510565162951536)-- (0,0);
\draw [line width=2pt] (1,0)-- (0,0);
\draw [line width=2pt] (0,0)-- (-0.30901699437494745,-0.9510565162951532);
\draw [line width=2pt] (-0.30901699437494745,-0.9510565162951532)-- (0.5,-1.5388417685876266);
\draw [line width=2pt] (0.5,-1.5388417685876266)-- (1.3090169943749475,-0.9510565162951536);
\draw [line width=2pt] (1.3090169943749475,-0.9510565162951536)-- (1,0);
\draw [line width=2pt] (3,0)-- (4,0);
\draw [line width=2pt] (4,0)-- (4.3090169943749475,0.9510565162951532);
\draw [line width=2pt] (4.3090169943749475,0.9510565162951532)-- (3.5,1.5388417685876266);
\draw [line width=2pt] (3.5,1.5388417685876266)-- (2.6909830056250525,0.9510565162951536);
\draw [line width=2pt] (2.6909830056250525,0.9510565162951536)-- (3,0);
\draw [line width=2pt] (4.3090169943749475,0.9510565162951532)-- (4.309016994374948,1.8017073246471944)-- (3.5,1.5388417685876266)-- (2.690983005625054,1.8017073246471922)-- (2.6909830056250525,0.9510565162951536)-- (2.190983005625053,0.26286555605956646)-- (3,0)-- (3.5,-0.6881909602355866)-- (4,0)-- (4.8090169943749475,0.26286555605956646);
\draw [line width=2pt] (4.8090169943749475,0.26286555605956646)-- (4.3090169943749475,0.9510565162951532);

\end{tikzpicture}
\end{center}
\caption{the double pentagon: different polygons, same surface}\label{doublePentagone}
\end{figure}

Our second main example is a generalization of the square: imagine that instead of just one square, you have a finite collection of same-sized squares, each of which has its  edges labeled ``right", ``left", ``top", ``bottom", and glue each right (resp. left) edge with a left (resp. right) edge, and each top (resp. bottom) edge with a bottom (resp. top) edge. Such surfaces are called \textbf{square-tiled\index{square-tiled surface}}, or sometimes \textbf{origami\index{origami}}. We have already seen the one-squared surface, which is a torus. Two-squared surfaces are also tori (compute the Euler characteristic). The first examples of genus $2$ are three-squared (see Figure \ref{3squared}).
They have only one point of type III, with total angle $6\pi$. A four-squared, genus $2$ surface, with two points of type III, each with total angle $4\pi$, is shown in Figure \ref{4squared}.

Note that according to common usage, square-tiled surfaces are translation surfaces, so the Klein bottle and the sphere in Figure \ref{toto} are not square-tiled surfaces, even though they are actually tiled by squares. 

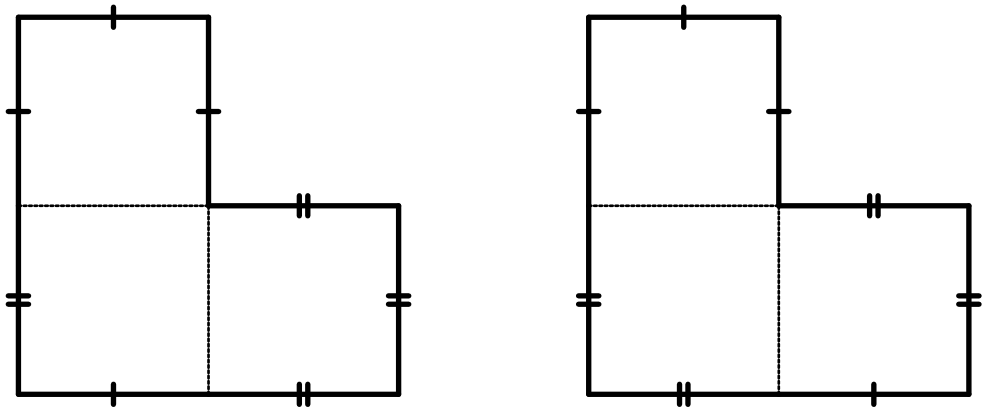
\begin{figure}[h!]
	\begin{center}
\definecolor{ududff}{rgb}{0.30196078431372547,0.30196078431372547,1}
\begin{tikzpicture}[line cap=round,line join=round,>=triangle 45,x=2.5cm,y=2.5cm]
\clip(-0.5,-0.5) rectangle (6,3);
\draw [line width=2pt] (0,0)-- (0,1);
\draw [line width=2pt] (-0.05333333333333336,0.47777777777777775) -- (0.05333333333333336,0.47777777777777775);
\draw [line width=2pt] (-0.05333333333333336,0.5222222222222221) -- (0.05333333333333336,0.5222222222222221);
\draw [line width=2pt] (0,1)-- (0,2);
\draw [line width=2pt] (-0.05333333333333336,1.5) -- (0.05333333333333336,1.5);
\draw [line width=2pt] (2,0)-- (2,1);
\draw [line width=2pt] (1.9466666666666663,0.47777777777777775) -- (2.0533333333333332,0.47777777777777775);
\draw [line width=2pt] (1.9466666666666663,0.5222222222222221) -- (2.0533333333333332,0.5222222222222221);
\draw [line width=2pt] (1,1)-- (1,2);
\draw [line width=2pt] (0.9466666666666665,1.5) -- (1.0533333333333332,1.5);
\draw [line width=2pt] (0,2)-- (1,2);
\draw [line width=2pt] (0.5,2.0533333333333332) -- (0.5,1.946666666666667);
\draw [line width=2pt] (1,1)-- (2,1);
\draw [line width=2pt] (1.4777777777777774,1.053333333333333) -- (1.4777777777777774,0.9466666666666665);
\draw [line width=2pt] (1.5222222222222217,1.053333333333333) -- (1.5222222222222217,0.9466666666666665);
\draw [line width=2pt] (1,0)-- (2,0);
\draw [line width=2pt] (1.4777777777777774,0.053333333333333274) -- (1.4777777777777774,-0.053333333333333274);
\draw [line width=2pt] (1.5222222222222217,0.053333333333333274) -- (1.5222222222222217,-0.053333333333333274);
\draw [line width=2pt] (0,0)-- (1,0);
\draw [line width=2pt] (0.5,0.053333333333333274) -- (0.5,-0.053333333333333274);
\draw [line width=2pt] (3,0)-- (3,1);
\draw [line width=2pt] (2.9466666666666668,0.47777777777777775) -- (3.0533333333333337,0.47777777777777775);
\draw [line width=2pt] (2.9466666666666668,0.5222222222222221) -- (3.0533333333333337,0.5222222222222221);
\draw [line width=2pt] (3,1)-- (3,2);
\draw [line width=2pt] (2.9466666666666668,1.5) -- (3.0533333333333337,1.5);
\draw [line width=2pt] (5,0)-- (5,1);
\draw [line width=2pt] (4.946666666666666,0.47777777777777775) -- (5.053333333333334,0.47777777777777775);
\draw [line width=2pt] (4.946666666666666,0.5222222222222221) -- (5.053333333333334,0.5222222222222221);
\draw [line width=2pt] (4,1)-- (4,2);
\draw [line width=2pt] (3.9466666666666663,1.5) -- (4.053333333333333,1.5);
\draw [line width=2pt] (3,2)-- (4,2);
\draw [line width=2pt] (3.5,2.0533333333333332) -- (3.5,1.946666666666667);
\draw [line width=2pt] (4,1)-- (5,1);
\draw [line width=2pt] (4.477777777777778,1.053333333333333) -- (4.477777777777778,0.9466666666666665);
\draw [line width=2pt] (4.522222222222222,1.053333333333333) -- (4.522222222222222,0.9466666666666665);
\draw [line width=2pt] (4,0)-- (5,0);
\draw [line width=2pt] (4.5,0.053333333333333274) -- (4.5,-0.053333333333333274);
\draw [line width=2pt] (3,0)-- (4,0);
\draw [line width=2pt] (3.477777777777777,0.053333333333333274) -- (3.477777777777777,-0.053333333333333274);
\draw [line width=2pt] (3.5222222222222217,0.053333333333333274) -- (3.5222222222222217,-0.053333333333333274);
\draw [line width=1pt,dash pattern=on 1pt off 1pt] (1,0)-- (1,1);
\draw [line width=1pt,dash pattern=on 1pt off 1pt] (0,1)-- (1,1);
\draw [line width=1pt,dash pattern=on 1pt off 1pt] (4,0)-- (4,1);
\draw [line width=1pt,dash pattern=on 1pt off 1pt] (3,1)-- (4,1);
\end{tikzpicture}
	\end{center}
\caption{Two different three-squared surfaces of genus two, with $St(3)$ on the left}\label{3squared}
\end{figure}

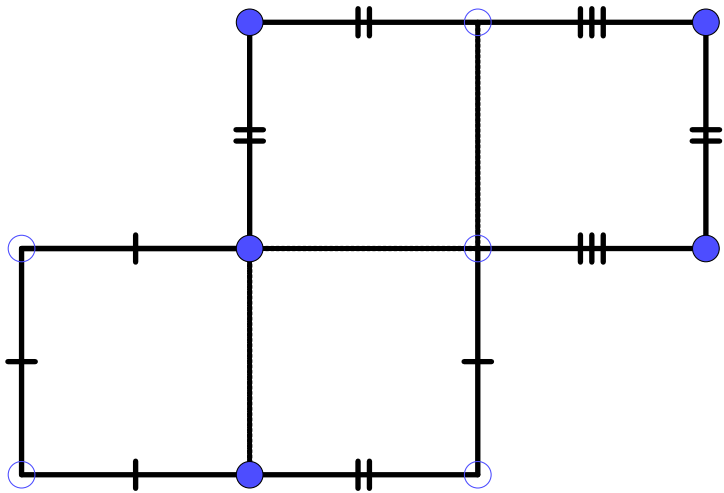
\begin{figure}[h!]
	\begin{center}
\definecolor{ududff}{rgb}{0.30196078431372547,0.30196078431372547,1}
\begin{tikzpicture}[line cap=round,line join=round,>=triangle 45,x=1.5cm,y=1.5cm]
\clip(-2.5,-1) rectangle (6.76,5);
\draw [line width=2pt] (0,0)-- (2,0);
\draw [line width=2pt] (1,0.12) -- (1,-0.12);
\draw [line width=2pt] (4,0)-- (2,0);
\draw [line width=2pt] (3.05,-0.12) -- (3.05,0.12);
\draw [line width=2pt] (2.95,-0.12) -- (2.95,0.12);
\draw [line width=2pt] (0,0)-- (0,2);
\draw [line width=2pt] (-0.12,1) -- (0.12,1);
\draw [line width=2pt] (0,2)-- (2,2);
\draw [line width=2pt] (1,2.12) -- (1,1.88);
\draw [line width=2pt] (2,2)-- (2,4);
\draw [line width=2pt] (1.88,2.95) -- (2.12,2.95);
\draw [line width=2pt] (1.88,3.05) -- (2.12,3.05);
\draw [line width=2pt] (2,4)-- (4,4);
\draw [line width=2pt] (2.95,4.12) -- (2.95,3.88);
\draw [line width=2pt] (3.05,4.12) -- (3.05,3.88);
\draw [line width=2pt] (4,4)-- (6,4);
\draw [line width=2pt] (4.9,4.12) -- (4.9,3.88);
\draw [line width=2pt] (5,4.12) -- (5,3.88);
\draw [line width=2pt] (5.1,4.12) -- (5.1,3.88);
\draw [line width=2pt] (6,4)-- (6,2);
\draw [line width=2pt] (6.12,3.05) -- (5.88,3.05);
\draw [line width=2pt] (6.12,2.95) -- (5.88,2.95);
\draw [line width=2pt] (6,2)-- (4,2);
\draw [line width=2pt] (5.1,1.88) -- (5.1,2.12);
\draw [line width=2pt] (5,1.88) -- (5,2.12);
\draw [line width=2pt] (4.9,1.88) -- (4.9,2.12);
\draw [line width=2pt] (4,2)-- (4,0);
\draw [line width=2pt] (4.12,1) -- (3.88,1);
\draw [line width=2pt,dash pattern=on 1pt off 1pt] (4,2)-- (4,4);
\draw [line width=2pt,dash pattern=on 1pt off 1pt] (2,2)-- (4,2);
\draw [line width=2pt,dash pattern=on 1pt off 1pt] (2,2)-- (2,0);
\draw [line width=2pt] (0,2)-- (2,2);
%\begin{scriptsize}
\draw [color=ududff] (0,0) circle (5pt);
\draw [fill=ududff] (2,0) circle (5pt);
\draw [color=ududff] (4,0) circle (5pt);
\draw [color=ududff] (4,2) circle (5pt);
\draw [fill=ududff] (6,2) circle (5pt);
\draw [fill=ududff] (6,4) circle (5pt);
\draw [color=ududff] (4,4) circle (5pt);
\draw [fill=ududff] (2,4) circle (5pt);
\draw [fill=ududff] (2,2) circle (5pt);
\draw [color=ududff] (0,2) circle (5pt);
%\end{scriptsize}
\end{tikzpicture}
	\end{center}
\caption{$St(4)$, a  four-squared surface of genus $2$, with two singularities}\label{4squared}
\end{figure}
Unlike regular polygons, which are somewhat few and far between, square-tiled surfaces are swarming all over the place. In fact, any translation surface may be approximated, in a sense that will be made precise later, by square-tiled surfaces, just because any line may be approximated uniformly by a stair-shaped broken line made of horizontal and vertical segments. 

A useful feature of square-tiled surfaces is that they are covers of the square torus, ramified over one point. Of course there is nothing special about the square here, each regular polygon comes with its own family of ramified covers. 

It is usually not a good idea to try to vizualize the surface $X$ in space, if only because, having non-positive curvature, it does not embed isometrically in $\R^3$. Computing the Euler characteristic is the way to go. It is interesting, however, when it lets you vizualize a decomposition of the surface into flat cylinders. Figures \ref{St3, stage 1}, \ref{St3, stage 2}, and \ref{St3, stage 3} show a topological embedding into $\R^3$ of the surface on the left of Figure \ref{3squared}, called $St(3)$ after \cite{S}. Figures \ref{St4, stage 1 }, 
\ref{St4, stage 2 }, and \ref{St4, stage 3} show a topological embedding into $\R^3$ of the surface of Figure \ref{4squared}, called $St(4)$ after \cite{S}.

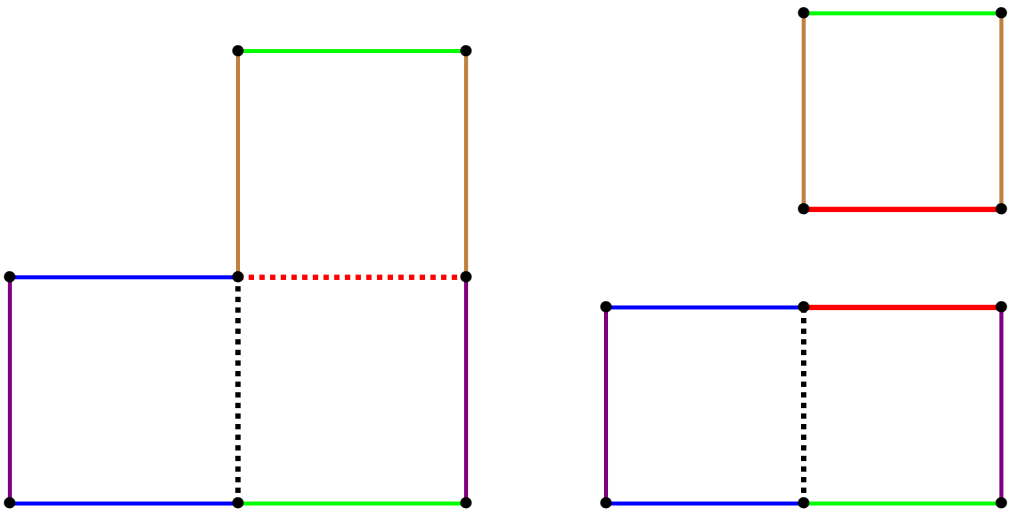
\begin{figure}[h!]
	\begin{center}
		\begin{minipage}[b]{0.5\linewidth}
			\begin{tikzpicture}[scale=1.5]
			
			\draw[color=blue,line width=1.5pt] (0,0) -- (2,0) ;
			\draw[color=blue,line width=1.5pt] (0,2) -- (2,2);
			\draw[color=green,line width=1.5pt] (2,0) -- (4,0);
			\draw[color=green,line width=1.5pt] (2,4) -- (4,4);
			\draw[color=brown,line width=1.5pt] (2,2) -- (2,4);
			\draw[color=brown,line width=1.5pt] (4,2) -- (4,4);
			\draw[color=violet,line width=1.5pt] (0,0) -- (0,2);
			\draw[color=violet,line width=1.5pt] (4,0) -- (4,2);
			\draw[dotted,line width=2pt,red] (2,2) -- (4,2);
			\draw[dotted,line width=2pt] (2,0) -- (2,2);
			\draw (0,0) node {$\bullet$};
			\draw (2,0) node {$\bullet$};
			\draw (4,0) node {$\bullet$};
			\draw (2,2) node {$\bullet$};
			\draw (2,4) node {$\bullet$};
			\draw (0,2) node {$\bullet$};
			\draw (4,2) node {$\bullet$};
			\draw (4,4) node {$\bullet$};
			\end{tikzpicture}
		\end{minipage}\hfil
		\begin{minipage}[b]{0.4\linewidth}   
			\begin{tikzpicture}[scale=1.3]
			\draw[color=blue,line width=1.5pt] (0,0) -- (2,0) ;
			\draw[color=blue,line width=1.5pt] (0,2) -- (2,2);
			\draw[color=green,line width=1.5pt] (2,0) -- (4,0);
			\draw[color=green,line width=1.5pt] (2,5) -- (4,5);
			\draw[color=violet,line width=1.5pt] (0,0) -- (0,2);
			\draw[color=violet,line width=1.5pt] (4,0) -- (4,2);
			\draw[color=brown,line width=1.5pt] (2,3) -- (2,5);
			\draw[color=brown,line width=1.5pt] (4,3) -- (4,5);
			\draw[line width=2pt,red] (2,3) -- (4,3);
			\draw[line width=2pt,red] (2,2) -- (4,2);
			\draw[dotted,line width=2pt] (2,0) -- (2,2);
			\draw (0,0) node {$\bullet$};
			\draw (2,0) node {$\bullet$};
			\draw (4,0) node {$\bullet$};
			\draw (2,2) node {$\bullet$};
			\draw (2,5) node {$\bullet$};
			\draw (0,2) node {$\bullet$};
			\draw (4,2) node {$\bullet$};
			\draw (4,5) node {$\bullet$};
			\draw (2,3) node {$\bullet$};
			\draw (4,3) node {$\bullet$};
			\end{tikzpicture}
		\end{minipage}
	\end{center}
\caption{Identifications for the surface $St(3)$, stage 1}\label{St3, stage 1}
\end{figure}

\begin{figure}[h!]
	\begin{center}
		\begin{tikzpicture}
		
		\draw [line width=1pt, violet] (0,0) ellipse (0.6cm and 1.5cm);
		%\draw [line width=0.8pt] (4,0) ellipse (0.6cm and 1.5cm);
		\draw (4,1.5) [dotted]arc [start angle=90, end angle=270,
		x radius=0.6cm, y radius=1.5cm];
		\draw (4,1.5) [line width=0.8pt]arc [start angle=90, end angle=-90,x radius=0.6cm, y radius=1.5cm];
		\draw [line width=0.8pt,blue] (0,1.5)-- (4,1.5);
		\draw [line width=0.8pt] (0,-1.5)-- (4,-1.5);
		\draw [line width=0.8pt] (4.15,-1.5)-- (9,-1.5);
		\draw [line width=1pt, red](4.15,1.5) .. controls (6,2) and (7,2).. (9,1.5);
		\draw [line width=1pt,green](4.15,1.5) .. controls (6,1) and (7,1).. (9,1.5);
		\draw (9,1.5) [dotted,violet, line width=1.2pt]arc [start angle=90, end angle=270,
		x radius=0.6cm, y radius=1.5cm];
		\draw (9,1.5) [line width=1.2pt, violet]arc [start angle=90, end angle=-90,x radius=0.6cm, y radius=1.5cm];

		\draw [line width=0.8pt, green] (6.5,6) ellipse (1.5cm and 0.6cm);
		\draw (8,3) [dotted, line width=1pt, red]arc [start angle=0, end angle=180, x radius=01.5cm, y radius=0.6cm];
		\draw (8,3) [line width=1pt, red]arc [start angle=0, end angle=-180, x radius=1.5cm, y radius=0.6cm];
		\draw [line width=1.2pt] (5,3)-- (5,6);
		\draw [line width=1.2pt, brown] (8,3)-- (8,6);
		\draw (0,1.45) node {$\bullet$};
		\draw (4.2,1.45) node {$\bullet$};
		\draw (8.8,1.45) node {$\bullet$};
		\draw (8,3) node {$\bullet$};
		\draw (8,6) node {$\bullet$};
		\end{tikzpicture}
	\end{center}
\caption{Identifications for the surface $St(3)$, stage 2}\label{St3, stage 2}
\end{figure}
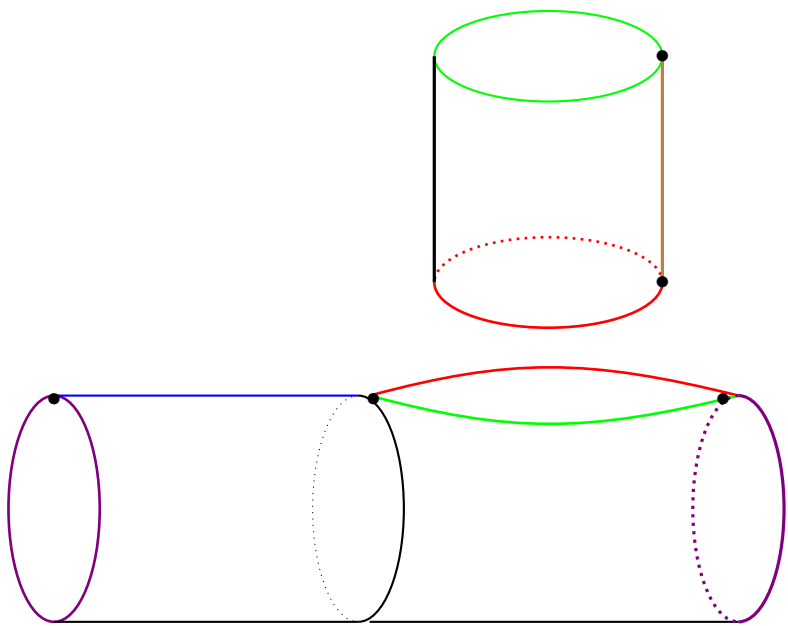

\begin{figure}[h!]
	\begin{center}
		\begin{minipage}[b]{0.5\linewidth}
			\begin{tikzpicture}[scale=0.9]
			%draw plot coordinates{(0,0)(1,1) (2,1.5) (3,1.3)(4,1)};
			
			\draw (0,0) [line width=0.7pt, blue] arc [start angle=40, end angle=-220,x radius=2.5cm, y radius=3.5cm];
			\draw [line width=1pt, green] (-0.9,-0.15) ellipse (1cm and 0.3cm);
			\draw [line width=1pt, red] (-2.9,-0.15) ellipse (1cm and 0.3cm);
			\draw [line width=1pt](-1.8,-1.5) .. controls +(-0.5,-0.7) and +(-0.5,0.7).. (-1.7,-4.3);
			\draw [line width=1pt](-1.9,-1.7) .. controls +(0.5,-0.5) and +(0.5,0.7).. (-1.9,-4);
			%\draw [line width=1pt](-4,-1.85) .. controls +(1,0.7) and +(-1,0.5).. (0.1,-1.85);
			
			\draw [line width=1pt, green] (-0.9,1) ellipse (1cm and 0.3cm);
			\draw [line width=1pt, red] (-2.9,1) ellipse (1cm and 0.3cm);
			\draw (0.1,1) [line width=0.7pt] arc [start angle=-30, end angle=210,x radius=2.3cm, y radius=3.5cm];
			
			\draw [line width=0.8pt, brown](-1.9,1) .. controls +(0.5,0.7) and +(0.5,0).. (-1.9,4);
			\draw [line width=0.8pt, brown](-1.9,1) .. controls +(-0.5,0.7) and +(-0.5,0).. (-1.9,4);
			\draw (-1.9,-0.2) [dotted, line width=1.2pt, violet] arc [start angle=90, end angle=270,	x radius=0.3cm, y radius=0.75cm];
			\draw (-1.9,-0.2) [line width=0.8pt, violet]arc [start angle=90, end angle=-90,x radius=0.3cm, y radius=0.75cm];
			\draw (-1.9,-0.15) node {$\bullet$};
			\draw (-1.9,1) node {$\bullet$};
			\end{tikzpicture}
		\end{minipage}\hfil
		\begin{minipage}[b]{0.25\linewidth}   
			
			\begin{tikzpicture}[scale=0.9]
			%draw plot coordinates{(0,0)(1,1) (2,1.5) (3,1.3)(4,1)};
			
			\draw (0,0) [line width=0.7pt, blue] arc [start angle=40, end angle=-220,x radius=2.5cm, y radius=3.5cm];
			\draw [line width=1pt, green] (-0.9,-0.17) ellipse (1cm and 0.3cm);
			\draw [line width=1pt, red] (-2.9,-0.15) ellipse (1cm and 0.3cm);
			\draw [line width=1pt](-1.8,-1.5) .. controls +(-0.5,-0.7) and +(-0.5,0.7).. (-1.7,-4.3);
			\draw [line width=1pt](-1.9,-1.7) .. controls +(0.5,-0.5) and +(0.5,0.7).. (-1.9,-4);
			%\draw [line width=1pt](-4,-1.85) .. controls +(1,0.7) and +(-1,0.5).. (0.1,-1.85);

			\draw (0.1,-0.1) [line width=0.7pt] arc [start angle=-30, end angle=210,x radius=2.3cm, y radius=3.5cm];
			
			\draw [line width=0.8pt, brown](-1.9,-0.1) .. controls +(0.5,0.7) and +(0.5,0).. (-1.9,3);
			\draw [line width=0.8pt, brown](-1.9,-0.1) .. controls +(-0.5,0.7) and +(-0.5,0).. (-1.9,3);
			\draw (-1.9,-0.2) [dotted, line width=1.2pt, violet] arc [start angle=90, end angle=270,	x radius=0.3cm, y radius=0.75cm];
			\draw (-1.9,-0.2) [line width=0.8pt, violet]arc [start angle=90, end angle=-90,x radius=0.3cm, y radius=0.75cm];
			\draw (-1.9,-0.15) node {$\bullet$};
			
			\end{tikzpicture}
		\end{minipage}
	\end{center}
\caption{Identifications for the surface $St(3)$, stage 3}\label{St3, stage 3}
\end{figure}
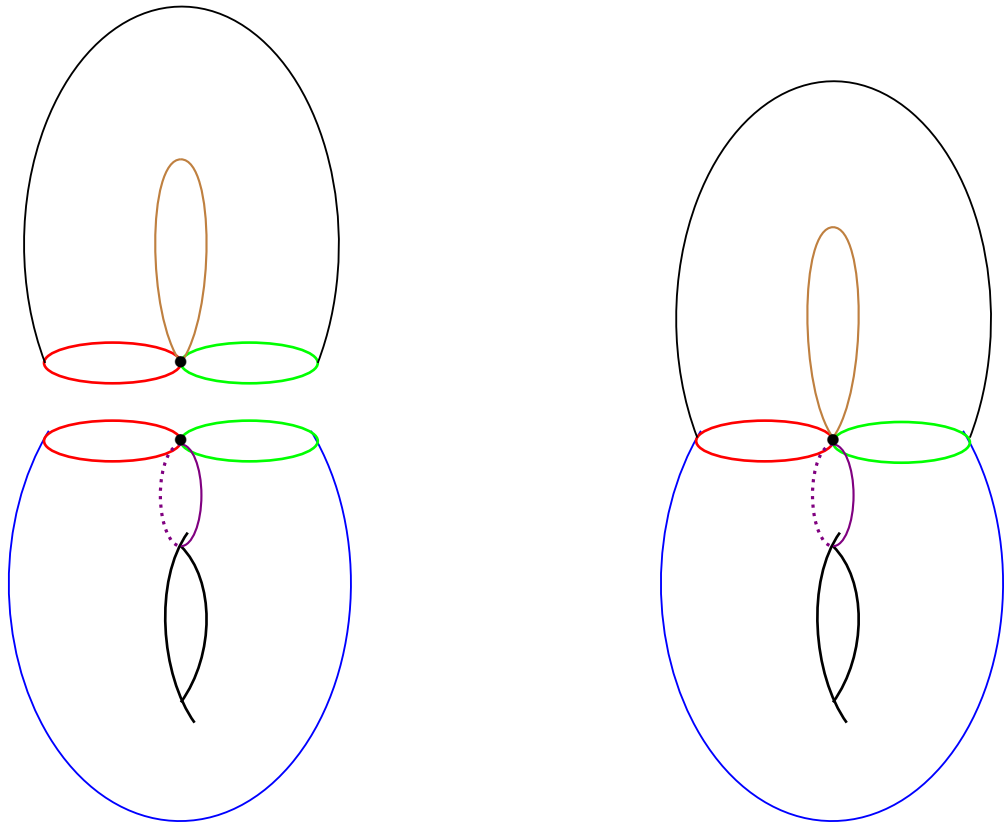

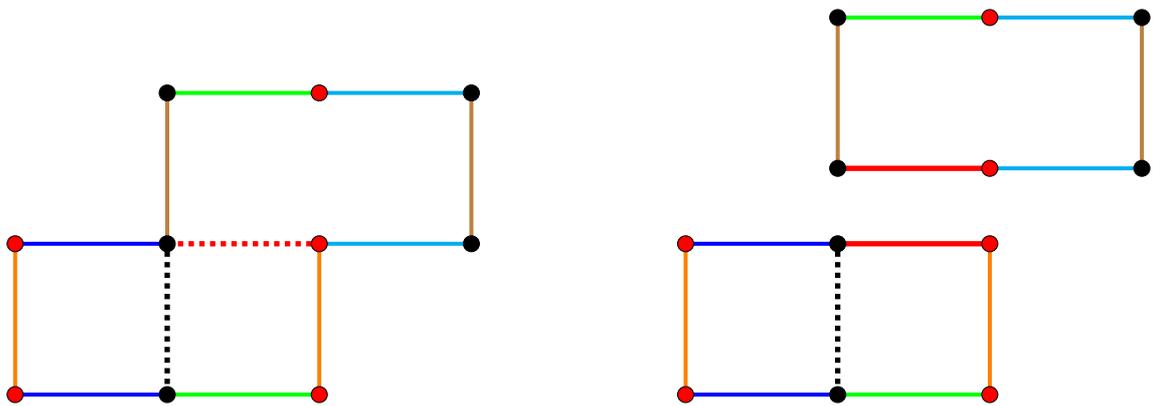
\begin{figure}[h!]
	\begin{center}
		\begin{minipage}[b]{0.6\linewidth}
			\begin{tikzpicture}%[scale=1.2]
			
			\draw[color=blue,line width=1.5pt] (0,0) -- (2,0) ;
			\draw[color=blue,line width=1.5pt] (0,2) -- (2,2);
			\draw[color=green,line width=1.5pt] (2,0) -- (4,0);
			\draw[color=green,line width=1.5pt] (2,4) -- (4,4);
			\draw[color=brown,line width=1.5pt] (2,2) -- (2,4);
			\draw[color=brown,line width=1.5pt] (6,2) -- (6,4);
			\draw[color=orange,line width=1.5pt] (0,0) -- (0,2);
			\draw[color=orange,line width=1.5pt] (4,0) -- (4,2);
			\draw[color=cyan,line width=1.5pt] (4,2) -- (6,2);
			\draw[color=cyan,line width=1.5pt] (4,4) -- (6,4);
			\draw[dotted,line width=2pt,red] (2,2) -- (4,2);
			\draw[dotted,line width=2pt] (2,0) -- (2,2);
			
			\draw [fill=red!100] (0,0) circle (3pt);
			\draw [fill=red!100] (0,2) circle (3pt);
			\draw [fill=red!100] (4,0) circle (3pt);
			\draw [fill=red!100] (4,2) circle (3pt);
			\draw [fill=red!100] (4,4) circle (3pt);
			
			\draw [fill] (2,0) circle (3pt);
			\draw [fill] (2,2) circle (3pt);
			\draw [fill] (2,4) circle (3pt);
			\draw [fill] (6,2) circle (3pt);
			\draw [fill] (6,4) circle (3pt);
			
			\end{tikzpicture}
		\end{minipage}\hfil
		\begin{minipage}[b]{0.4\linewidth}   
			
			\begin{tikzpicture}%[scale=1.2]
			
			\draw[color=blue,line width=1.5pt] (0,0) -- (2,0) ;
			\draw[color=blue,line width=1.5pt] (0,2) -- (2,2);
			\draw[color=green,line width=1.5pt] (2,0) -- (4,0);
			\draw[color=green,line width=1.5pt] (2,5) -- (4,5);
			\draw[color=orange,line width=1.5pt] (0,0) -- (0,2);
			\draw[color=orange,line width=1.5pt] (4,0) -- (4,2);
			\draw[color=brown,line width=1.5pt] (2,3) -- (2,5);
			\draw[color=brown,line width=1.5pt] (6,3) -- (6,5);
			\draw[color=cyan,line width=1.5pt] (4,3) -- (6,3);
			\draw[color=cyan,line width=1.5pt] (4,5) -- (6,5);
			\draw[line width=2pt,red] (2,3) -- (4,3);
			\draw[line width=2pt,red] (2,2) -- (4,2);
			\draw[dotted,line width=2pt] (2,0) -- (2,2);
			\draw [fill=red!100] (0,0) circle (3pt);
			\draw [fill=red!100] (0,2) circle (3pt);
			\draw [fill=red!100] (4,0) circle (3pt);
			\draw [fill=red!100] (4,2) circle (3pt);
			\draw [fill=red!100] (4,5) circle (3pt);
			\draw [fill=red!100] (4,3) circle (3pt);
			
			\draw [fill] (2,0) circle (3pt);
			\draw [fill] (2,2) circle (3pt);
			\draw [fill] (2,5) circle (3pt);
			\draw [fill] (6,3) circle (3pt);
			\draw [fill] (6,5) circle (3pt);
			\draw [fill] (2,3) circle (3pt);
			\end{tikzpicture}
		\end{minipage}
	\end{center}
\caption{Identifications for the surface $St(4)$, stage 1}
\label{St4, stage 1 }
\end{figure}

\begin{figure}[h!]
	\begin{center}
		\begin{tikzpicture}
		
		\draw [line width=1pt, orange] (0,0) ellipse (0.6cm and 1.5cm);
		%\draw [line width=0.8pt] (4,0) ellipse (0.6cm and 1.5cm);
		\draw (4,1.5) [dotted]arc [start angle=90, end angle=270,
		x radius=0.6cm, y radius=1.5cm];
		\draw (4,1.5) [line width=0.8pt]arc [start angle=90, end angle=-90,x radius=0.6cm, y radius=1.5cm];
		\draw [line width=0.8pt,blue,line width=1.5pt] (0,1.5)-- (4,1.5);
		\draw [line width=0.8pt] (0,-1.5)-- (4,-1.5);
		\draw [line width=0.8pt] (4.15,-1.5)-- (9,-1.5);
		\draw [line width=1pt, red](4.15,1.5) .. controls +(1.85,0.5) and +(-2,0.5).. (9,1.5);
		\draw [line width=1pt,green](4.15,1.5) .. controls +(1.85,-0.5) and +(-2,-0.5).. (9,1.5);
		\draw (9,1.5) [dotted,orange, line width=1.2pt]arc [start angle=90, end angle=270,
		x radius=0.6cm, y radius=1.5cm];
		\draw (9,1.5) [line width=1.2pt, orange] arc [start angle=90, end angle=-90,x radius=0.6cm, y radius=1.5cm];
		\draw [fill=red!100] (0,1.5) circle (3pt);
		\draw [fill=red!100] (9,1.5) circle (3pt);
		\draw [fill] (4.15,1.5) circle (3pt);

		\draw [line width=1pt, brown] (13,4) ellipse (0.6cm and 1.5cm);
		%\draw [line width=0.8pt] (9,0) ellipse (0.6cm and 1.5cm);
		\draw (8,5.5) arc [start angle=90, end angle=270,
		x radius=0.6cm, y radius=1.5cm];
		\draw (8,5.5) [dotted,line width=0.8pt]arc [start angle=90, end angle=-90,x radius=0.6cm, y radius=1.5cm];
		\draw [line width=0.8pt] (4,5.5)-- (8,5.5);
		\draw [line width=0.8pt] (8,5.5)-- (13,5.5);
		\draw [line width=1.5pt,cyan] (8.15,2.5)-- (13,2.5);
		\draw [line width=1.5pt, red,dotted](4,2.5) .. controls +(1.5,0.3) and +(-1.5,0.3).. (8,2.5);
		\draw [line width=1pt,green](4,2.5) .. controls +(1.5,-0.2) and +(-1.5,-0.2).. (8,2.5);
		\draw (4,5.5) [dotted,brown, line width=1.2pt]arc [start angle=90, end angle=-90,x radius=0.6cm, y radius=1.5cm];
		\draw (4,5.5) [line width=1.2pt, brown] arc [start angle=90, end angle=270,x radius=0.6cm, y radius=1.5cm];
		
		\draw [fill] (4,2.5) circle (3pt);
		\draw [fill] (13,2.5) circle (3pt);
		\draw [fill=red!100] (8.1,2.5) circle (3pt);
		
		\end{tikzpicture}
	\end{center}
\caption{Identifications for the surface $St(4)$, stage 2}
\label{St4, stage 2 }
\end{figure}
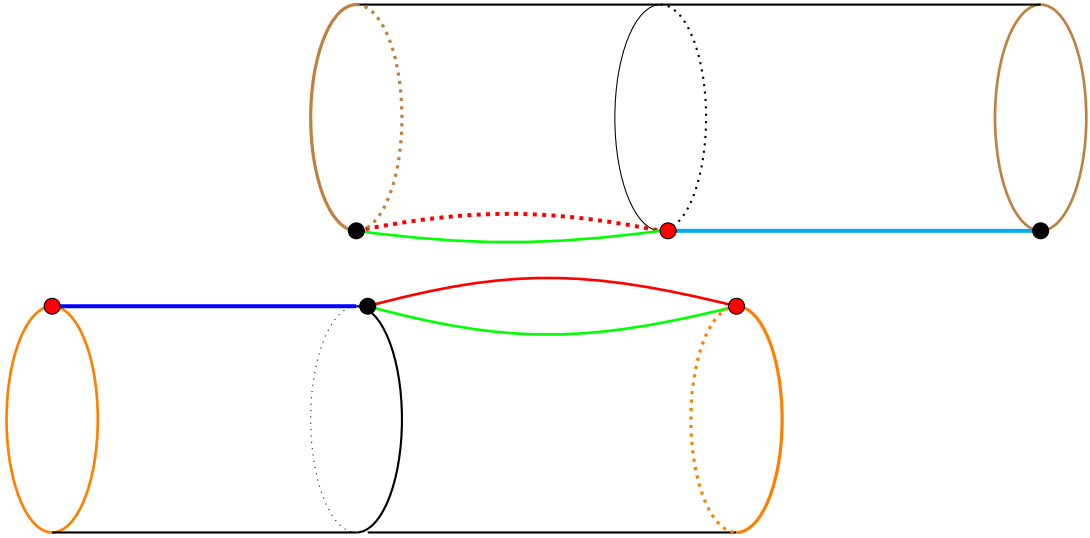

\begin{figure}[h!]
	\begin{center}
		\begin{minipage}[b]{0.5\linewidth}
			\begin{tikzpicture}%[scale=1.2]
			\draw (0,0) [blue,line width=1pt] arc [start angle=40, end angle=-220,x radius=2.53cm, y radius=3.5cm];
			\draw [line width=1pt, red](-3.9,0) .. controls +(1,0.7) and +(-1.5,0.7).. (0,0);
			\draw [line width=1pt,green](-3.9,0) .. controls +(1,-0.7) and +(-1.5,-0.7).. (0,0);
			%\draw [line width=1pt](-1.8,-1.5) .. controls +(-0.5,-0.7) and +(-0.5,0.7).. (-1.7,-4.3);
			%\draw [line width=1pt](-1.9,-1.7) .. controls +(0.5,-0.5) and +(0.5,0.7).. (-1.9,-4);
			\draw [line width=1pt, orange](0,0) .. controls +(0.7,-1) and +(1,-0.8).. (-1.9,-1.72);
			\draw [line width=1pt,dotted, orange](0,0) .. controls +(-1,0.5) and +(-0.5,1).. (-1.9,-1.7);
			\draw [line width=1pt](-1.8,-1.5) .. controls +(-0.5,-0.7) and +(-0.5,0.7).. (-1.7,-4.3);
			\draw [line width=1pt](-1.9,-1.7) .. controls +(0.5,-0.5) and +(0.5,0.7).. (-1.9,-4);
			
			\draw (0,2) [cyan,line width=1pt] arc [start angle=-40, end angle=220,x radius=2.53cm, y radius=3.5cm];
			\draw [line width=1.5pt,dotted, red](-3.9,2) .. controls +(1,0.7) and +(-1.5,0.7).. (0,2);
			\draw [line width=1pt,green](-3.9,2) .. controls +(1,-0.7) and +(-1.5,-0.7).. (0,2);
			
			\draw [line width=1pt](-1.8,3.1) .. controls +(-0.5,0.7) and +(-0.5,-0.7).. (-1.7,6.2);
			\draw [line width=1pt](-1.9,3.3) .. controls +(0.5,0.5) and
			+(0.5,-0.7).. (-1.8,6);
			
			\draw [line width=1.3pt,dotted, brown](-3.9,2) .. controls +(-0.2,0.3) and +(-1,0.3).. (-1.9,3.35);
			\draw [line width=1.5pt, brown](-3.9,2) .. controls +(0.5,-0.3) and +(-0.4,0.3).. (-1.9,3.35);
			\draw [fill] (-3.9,2) circle (2pt);
			\draw [fill=red!100] (0,2) circle (2pt);
			\draw [fill] (-3.9,0) circle (2pt);
			\draw [fill=red!100] (0,0) circle (2pt);
			\end{tikzpicture}
		\end{minipage}\hfil
		\begin{minipage}[b]{0.25\linewidth}   
			
			\begin{tikzpicture}%[scale=1.2]
			\draw (0,0) [blue,line width=1pt] arc [start angle=40, end angle=-220,x radius=2.53cm, y radius=3.5cm];
			\draw [line width=1pt,dotted, red](-3.9,0) .. controls +(1,0.7) and +(-1.5,0.7).. (0,0);
			\draw [line width=1pt,green](-3.9,0) .. controls +(1,-0.7) and +(-1.5,-0.7).. (0,0);
			%\draw [line width=1pt](-1.8,-1.5) .. controls +(-0.5,-0.7) and +(-0.5,0.7).. (-1.7,-4.3);
			%\draw [line width=1pt](-1.9,-1.7) .. controls +(0.5,-0.5) and +(0.5,0.7).. (-1.9,-4);
			%\draw [line width=1pt, orange](0,0) .. controls +(0.7,-1) and +(1.3,-0.8).. (-1.9,-1.7);
			%\draw [line width=1.5pt,dotted, orange](0,0) .. controls +(-1,0.5) and +(-0.5,1.5).. (-1.9,-1.7);
			\draw [line width=1pt, orange](0,0) .. controls +(0.7,-1) and +(1,-0.8).. (-1.9,-1.72);
			\draw [line width=1pt,dotted, orange](0,0) .. controls +(-1,0.5) and +(-0.5,1).. (-1.9,-1.7);
			\draw [line width=1pt](-1.8,-1.5) .. controls +(-0.5,-0.7) and +(-0.5,0.7).. (-1.7,-4.3);
			\draw [line width=1pt](-1.9,-1.7) .. controls +(0.5,-0.5) and +(0.5,0.7).. (-1.9,-4);
			%\draw [line width=1.3pt,dotted, brown](0,0) .. controls +(0.3,0.3) and +(1,0.3).. (-1.9,1.4);
			%\draw [line width=1.5pt, brown](0,0) .. controls +(-0.5,-0.3) and +(0.3,0.3).. (-1.9,1.4);
			\draw (0,0) [cyan,line width=1pt] arc [start angle=-40, end angle=220,x radius=2.53cm, y radius=3.5cm];
			%\draw [line width=1pt,dotted, red](-3.9,2) .. controls +(1,0.7) and +(-1.5,0.7).. (0,2);
			%\draw [line width=1pt,green](-3.9,2) .. controls +(1,-0.7) and +(-1.5,-0.7).. (0,2);
			\draw [line width=1pt](-1.8,1.1) .. controls +(-0.5,0.7) and +(-0.5,-0.7).. (-1.7,4.2);
			\draw [line width=1pt](-1.9,1.3) .. controls +(0.5,0.5) and +(0.5,-0.7).. (-1.8,4);
			
			\draw [line width=1.3pt,dotted, brown](-3.9,0) .. controls +(-0.2,0.3) and +(-1,0.3).. (-1.9,3.35-2);
			\draw [line width=1.5pt, brown](-3.9,0) .. controls +(0.5,-0.3) and +(-0.4,0.3).. (-1.9,3.35-2);
			\draw [fill] (-3.9,0) circle (2pt);
			\draw [fill=red!100] (0,0) circle (2pt);
			
			\end{tikzpicture}
		\end{minipage}
	\end{center}
\caption{Identifications for the surface $St(4)$, stage 3}
\label{St4, stage 3}
\end{figure}
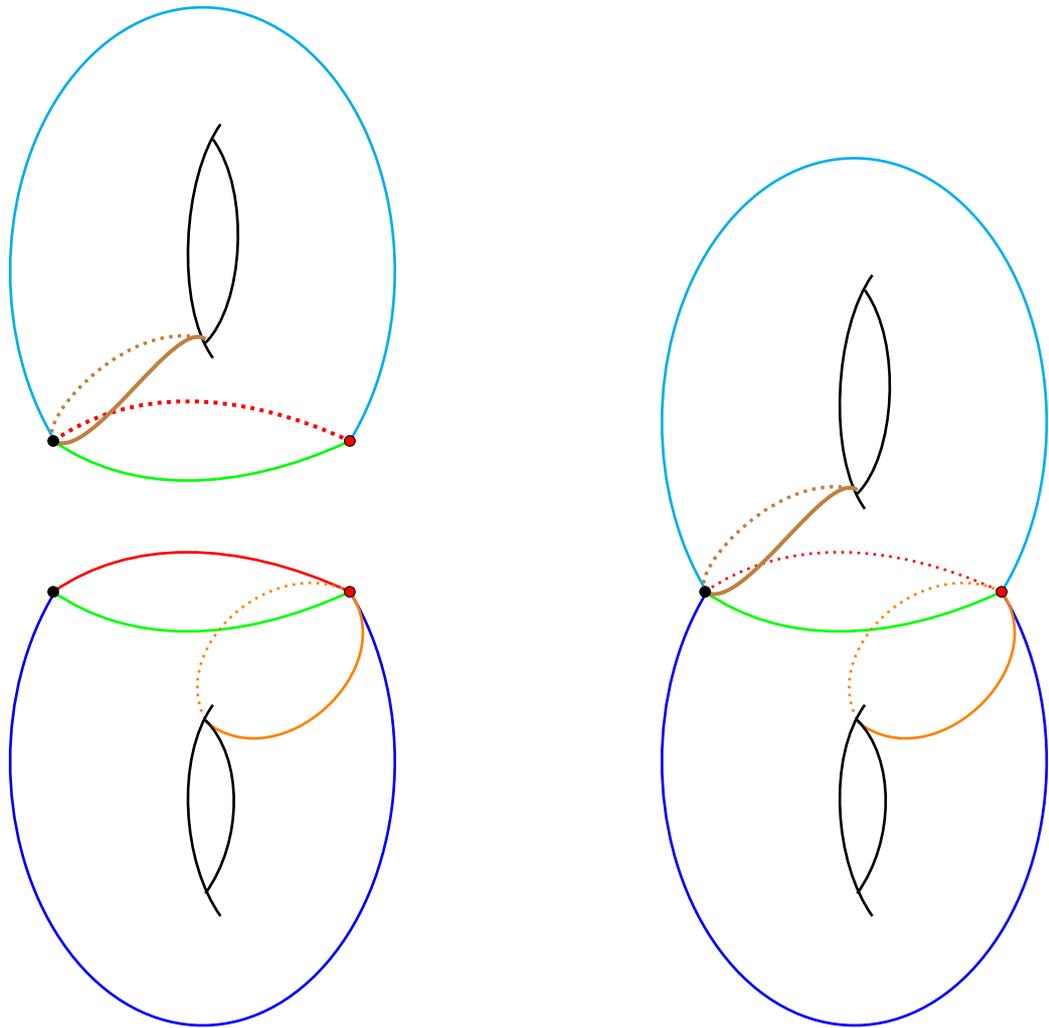
\subsection{Definition through an atlas}
What if we took the property we have just proved as a definition? Let us say that a translation surface is a compact manifold $X$ of dimension two, such that there exists a finite subset 
$\Sigma= \{ x_1, \ldots, x_k \}$, called the singular set of $X$, and an atlas of $X$, such that the transition map between any two charts whose domains do not meet $\Sigma$ is a translation, and the transition map between a chart whose domain meets the singular set, and a chart whose domain does not, is $z \mapsto z^k$ for some $k \in \N$.
 Is that an equivalent definition? Meaning, from this data, can we extract a polygon with identifications, so that when we perform the identifications, we get back the translation atlas? 

Well, among the many structures translations preserve, there  is the Euclidean metric. Therefore we can equip our surface $X \setminus \Sigma$ with a Riemannian metric which is locally Euclidean, 
so its geodesics are locally straight lines. The Riemannian metric may not extend to the whole of $X$, but the distance function does. In particular we can draw geodesics between any two singularities 
(elements of $\Sigma$). 

Now, draw geodesics between singularities, as many  (but finitely many) of them as you like, as long as the connected components of the complement in $X$ of those geodesics are simply connected. 
Being simply connected, they may be developed to the plane, to polygons. Then, apply the construction of Subsection \ref{def1} to this polygon  
(remember, I never said the polygon from which $X$ is glued should be connected).

The atlas we obtain from gluing back may not be the same atlas we started with, but they share a common maximal atlas. So, with the usual polite fiction of a manifold as a maximal atlas, the new definition is equivalent to the first one.

Although we shall try to ignore it as much as possible, we have to mention the following annoying
\begin{fact}
	Given a homeomorphism $f$ of the surface $X$, and an atlas $(U_i,\phi_i)_i$ on $X$, the charts $\phi_i$ may be pre-composed with $f$, thus giving a new atlas $(U_i,\phi_i \circ f)_i$, which in general is not compatible with the first one.
\end{fact}
This fact is annoying because  it only  makes a difference if you tag each point in $X$ with a label and are interested in tracking each individual label. The usual way around it is through another polite fiction: 
decide two maximal atlases are equivalent if they may be deduced from each other by pre-composition with a homeomorphism (sometimes it is convenient to restrict to homeomorphisms which are isotopic to the identity map). Then define your structure as this most unfathomable object: an equivalence class of maximal atlases. 
\subsection{Definition of a translation surface as a holomorphic differential}
We have seen that a translation surface (as in our first definition) has a complex structure, which is essentially a protractor (a way to measure angles between tangent vectors at any point). But it actually has much more: a graduated ruler (since we can measure distances), and a compass, whose needle points North, wherever you are, except at the singular set. This is because our polygon is a subset of the Euclidean plane, so we can choose any direction we want as the North, and since all identifications are  translations (except at the singular set), this goes down to the quotient space $X$. At a vertex of $P$, the needle of the compass gets a little dizzy, but not in the same way as an actual compass at the magnetic pole of the Earth would: while the latter has infinitely many directions to choose from, our compass only has finitely many. If the angle around the singular point $x_i$ is $2k \pi$, then there are, so to say, $k$ different northes at $x_i$. 

Formally, the package (protractor, ruler, compass)  is called a \textbf{holomorphic\index{holomorphic differential}\index{differential!holomorphic}, or Abelian\index{Abelian differential}\index{differential!Abelian},  differential}, and usually denoted $\omega$. If you have a complex manifold $X$, a holomorphic differential is just, for every $x \in X$, the choice of a complex linear map 
$l(x)$ from the tangent space to $X$ at $x$, to $\C$, with the condition that $l(x)$ depends homomorphically on $x$. Our $X$'s have complex dimension one, and the only complex linear maps from $\C$ to itself are the maps $z \mapsto \lambda z$, with $\lambda \in \C$. When we think of the identity as a holomorphic differential on $\C$, we denote it $dz$, so we will usually think of a holomorphic differential $\omega$ 
on $X$ as $f(z)dz$, with $f$ holomorphic, in  local charts, with the usual compatibility requirement that if two chart domains intersect, and $\omega$ reads $f(z)dz$ in one chart, and $g(z)dz$ in the other, 
and $T(z)$ is the transition map, then $f(T(z))T'(z)=g(z)$ ($f$ and $g$ should be thought of as derivatives, since $f(z)dz$ is a differential, so we just apply the chain rule for derivatives). 

If you would like to see it put another way: if $x \in X$, and $v$ is a tangent vector to $X$ at $x$, and we have two different charts $\phi$ and $\psi$ at $x$, such that $\omega$ reads $f(z)dz$ in the chart $\phi$, and $g(z)dz$ in the chart $\psi$, then $\omega(x).v$ reads $f(\phi(x))\phi'(x).v$ in the chart $\phi$, and $g(\psi(x))\psi'(x).v$ in the chart $\psi$. Since the evaluation of $\omega$ does not depend on the chart, we have $f(\phi(x))\phi'(x)=g(\psi(x))\psi'(x)$. Now let $T$ be the  transition map $\phi \circ \psi^{-1}$, then, setting $z=\psi(x)$, we have $f(T(z))\phi'(\psi^{-1}(z))=g(z)\psi'(\psi^{-1}(z))$, which is $f(T(z))T'(z)=g(z)$.

Not every compact manifold has a non-zero holomorphic differential, for instance, if you try and extend $dz$ to the sphere $\C P^1$, 
you get a double pole at infinity, so the needle of your compass will act tipsy at infinity, like  at the North Pole. 
Translation surfaces are different from $\C P^1$ : the holomorphic differential $dz$ in the plane goes down to a holomorphic differential on the quotient space $X$, 
albeit with zeroes. The order of the zeroes is given by the angle at the singularities, as follows from the chain rule: if $g$ is a chart around a singular point, and $f$ is a regular chart, 
then the transition map is $z \mapsto z^k$, so $k f(z^k)z^{k-1}=g(z)$, in particular $g$ has a zero of order $k-1$ at the singular point. 

Assume $\omega$ reads $g(z)dz$ in some chart $(U, \phi)$ at $x_0 \in X$, with $\phi(x_0)=0$ and $g(0) \neq 0$, so we say that $x_0$ is a regular point of $\omega$. Then we may take a primitive of $\omega$ as a new chart at $x_0$: 
\[
\begin{array}{rcl}
\psi :  V & \longrightarrow & \C \\
x & \longmapsto & \int_{x_0}^{x} \omega 
\end{array}
 \]
 where $V$ is some neighborhood of $x_0$, with $V \subset U$. 
The map $\psi$ is a local biholomorphism because $\omega$ does not vanish at $x_0$. 
Let $G$ be the local primitive of $g$, defined in $V$, such that $G(0)=0$.
Note that $\psi(x) = G( \phi(x))$, so the transition map $\psi \circ \phi^{-1}$ between $\phi$ and $\psi$ is $G$. Thus, by the chain rule, assuming $\omega$ reads $f(z)dz$ in the chart $\psi$, we have 
\[
f\left( \psi \circ \phi^{-1} (z) \right) G'(z) = g(z)
  \]
  so $f$ is constant $=1$, that is, $\omega$ reads $dz$ in the chart $\psi$.
  
  Now assume we have two different charts, at different, regular, points $x$ and $y$, whose domains intersect, and assume that in both charts $\omega$ reads $dz$. Let $\phi$ be the transition map between the two charts. Then,  by the chain rule, we have $\phi'=1$, that is, $\phi$ is a translation. Therefore, from a holomorphic differential, we recover a translation atlas,
  thus proving the equivalence of our three definitions.  
  
\subsection{Definition of a half-translation surface as a quadratic differential} 
  Now assume that instead of a translation surface, we have a half-translation surface. What could play the role of a holomorphic differential in that case? Well, the $z \mapsto -z$ map does not preserve the linear form $dz$, but it does preserve the quadratic form $dz^2$. 
  
  If you have a complex manifold $X$, a \textbf{holomorphic quadratic differential\index{quadratic differential}\index{differential!quadratic}} is
  to a holomorphic differential what a quadratic form is to a linear form, so it  is,   for every $x \in X$, the choice of a complex-valued  quadratic form 
  $l(x)$ from the tangent space to $X$ at $x$, to $\C$, with the condition that $l(x)$ depends holomorphically on $x$. Our $X$'s have complex dimension one, and the only complex-valued quadratic forms from $\C$ to itself are the maps $z \mapsto \lambda z^2$, with $\lambda \in \C$. When we think of the square map as a  quadratic differential on $\C$, we denote it $dz^2$, so we will usually think of a holomorphic quadratic differential $\omega$ 
  on $X$ as $f(z)dz^2$, with $f$ holomorphic, in  local charts, with the usual compatibility requirement that if two chart domains intersect, and $\omega$ reads $f(z)dz^2$ in one chart, and $g(z)dz^2$ in the other, 
  and $\phi(z)$ is the transition map, then $f(\phi(z))(\phi'(z))^2=g(z)$ ($f$ and $g$ should be thought of as squares of derivatives). Actually we should be allowing simple poles for $f$ as well, for instance the half-translation sphere in Figure \ref{toto} has four simple poles,  but we shall not dwell on that.  
  
  Note that given a holomorphic differential $\omega$, written $f(z)dz$ in local coordinates, the formula $f(z)^2dz^2$ yields a quadratic holomorphic differential, so the set of holomorphic differentials (modulo identification of a differential with its polar opposite) identifies  with a subset of the set of quadratic holomorphic differentials. It is a proper subset, however, because not every holomorphic function is a square.

  While not every quadratic differential is a square, every quadratic differential becomes a square in a suitable double cover of $X$,  by the classical Riemann surface construction of the square root function. For instance,  the half-translation sphere in Figure \ref{toto} is covered, twofold, by a 4-square torus, in fact the involution of the covering is the hyperelliptic involution of the torus.  It is often convenient to deal with translation surfaces only, taking double covers when need be.

\subsection{Dynamical system point of view}

So far we have seen a translation surface as a geometric, or complex-analytic, object, but it is interesting to view it as a dynamical system. 

We have seen that the North is well-defined on a translation surface $X$, outside the singularities; so we have a local flow $\phi_t$ on $X$,  usually called the \textbf{vertical flow\index{vertical flow}\index{flow!vertical}}, defined by ``walk North for $t$ miles". The flow is complete (i.e well-defined for any $t \in \R$) if we remove the finite, hence negligible,  union of orbits which end in a singularity (which we call \textbf{singular orbits\index{singular orbit}\index{orbit!singular}}). 

Of course we can multiply our Abelian differential by $e^{i\theta}$ for any $\theta$, thus turning the vertical by $\theta$, so in fact we have a family of flows indexed by $\theta \in \left[ 0,2\pi \right] $.

We  ask the usual dynamical questions: are there periodic orbits? if so, can we enumerate them? are there  invariant measures not supported on periodic orbits? If so, how many of them? 

In the case of the square torus, the questions are readily answered: by taking the first return map on a closed transversal, we see that the dynamical properties of the  flow in a given direction $\theta$ are those of a rotation on the circle. That is, when $\theta/\pi \in \Q$, every orbit is periodic, with the same period. 

When $\theta/\pi \in \R \setminus \Q$, every orbit is equidistributed, that is, the proportion of its time it spends in a given interval is the proportion of the circle this interval occupies. In particular the flow in the  direction $\theta$ is uniquely ergodic, that is, it supports a unique invariant measure, up to a scaling factor. 

In the case of square-tiled surfaces, the answer is equally satisfying.
Recall that any square-tiled surface $X$  is a  ramified cover of the square torus, so the flow in the direction $\theta$ on $X$ projects to the flow in the direction $\theta$ on the square torus. Therefore, if $\theta/\pi \in \Q$, every orbit of the  flow in the direction $\theta$ on $X$ projects to a closed orbit in the square torus, so it must be periodic itself. The period need not be the same for all orbits, though, for instance the vertical flow on  $X=St(3)$ has orbits of period $1$ and $2$ (see Figure \ref{3squared}). 
In fact $X=St(3)$, minus the set of singular orbits, is the reunion of two cylinders, one made with orbits of length $2$, and the other made of orbits of length $1$.

The answer when $\theta/\pi \in \R \setminus \Q$ is a bit trickier, because measures cannot be projected, they may only be lifted, and it is unclear why every invariant measure in $X$ should be the lift of an invariant measure in the torus, but as in the torus case, orbits are equidistributed.

This dichotomy between directions which are \textbf{completely periodic\index{completely periodic}}, meaning that the surface decomposes into cylinders of periodic orbits, and directions which are uniquely ergodic, is called the \textbf{Veech dichotomy\index{Veech dichotomy}}.
In the next sections we are going to see a  powerful criterion for unique ergodicity, and apply it to find a larger class of surfaces which satisfy the Veech dichotomy.

In general it is hard to prove that a given surface (say, the double pentagon) satisfies the Veech dichotomy. On the other hand, it is easy to find a surface which does not satisfy the Veech dichotomy (see Figure \ref{pasVeechDichotomy}).

In \cite{MT} an example is given of a translation surface with a direction $\theta$ such that the flow in the direction $\theta$ is minimal (every orbit is dense) but not uniquely ergodic, that is, it supports several distinct invariant measures.  
\begin{figure}[h!]
	\begin{center}
		\definecolor{zzttqq}{rgb}{0.6,0.2,0}
\begin{tikzpicture}[line cap=round,line join=round,>=triangle 45,x=1cm,y=1cm, scale=1.5]
\clip(-1,-1) rectangle (5,5);
\fill[line width=2pt,color=zzttqq,fill=zzttqq,fill opacity=0.10000000149011612] (2,2) -- (4,2) -- (4,0) -- (2,0) -- cycle;
\draw [line width=1pt] (0,0)-- (2,0);
\draw [line width=1pt] (2,0)-- (4,0);
\draw [line width=1pt] (4,0)-- (4,2);
\draw [line width=1pt] (4,2)-- (2,2);
\draw [line width=1pt] (2,2)-- (2.66666,4);
\draw [line width=1pt] (2.66666,4)-- (0.66666,4);
\draw [line width=1pt] (0.66666,4)-- (0,2);
\draw [line width=1pt] (0,2)-- (0,0);
\draw [line width=1pt,color=zzttqq] (2,2)-- (4,2);
\draw [line width=1pt,color=zzttqq] (4,2)-- (4,0);
\draw [line width=1pt,color=zzttqq] (4,0)-- (2,0);
\draw [line width=1pt,color=zzttqq] (2,0)-- (2,2);
%\draw [line width=1.2pt,dash pattern=on 1pt off 1pt] (0.66666,4)-- (0.66666,0);
\draw [line width=1pt,dash pattern=on 1pt off 2pt] (0.66666,4)-- (0.66666,0);
\draw [line width=1pt] (0,-0.5)-- (0.67,-0.5);
\draw [line width=1pt] (0,-0.5)-- (0.1,-0.4);
\draw [line width=1pt] (0,-0.5)-- (0.1,-0.6);
\draw [line width=1pt] (0.67,-0.5)-- (0.57,-0.4);
\draw [line width=1pt] (0.67,-0.5)-- (0.57,-0.6);

\draw[color=black] (0.35,-0.3) node {$x$};

\end{tikzpicture}
	\caption{Take the surface $St(3)$ and shear the top edge to the right by $x$, keeping the same gluing pattern.The shaded square projects to the surface as a cylinder of vertical periodic geodesics of length $1$, while the rest projects to a cylinder where every vertical geodesic is dense, if $x$ is irrational.} \label{pasVeechDichotomy}
\end{center}
\end{figure}
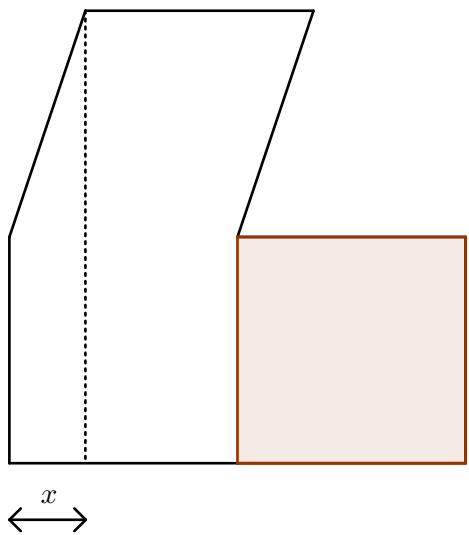
%%%%%%%%%%%%
\section{Moduli space}\label{moduli space}
%%%%%%%%%%%%%

Now that we know what a translation structure on a surface is, we would like to know what it means for two translation structures to be the same, or almost the same. Sameness is easy to define, if hard to visualize,  through atlases: we say two translation structures (as atlases) are equivalent if they share a common (equivalence class of) maximal translation atlas. Observe that a necessary condition for sameness is that the number of singular points, and the angle around each singular point, are the same.

It is a bit harder to see sameness from polygons, because two polygons which look very different may yield the same surface after identifications, see Figure \ref{doublePentagone}.
The correct definition of sameness is that two polygons yield the same translation surface if one may be cut and pasted into the other, provided the pieces are only re-arranged by translation 
(no rotation or flip allowed). 

It is not always easy to tell when two polygons do \textit{not} yield the same surface, either. The two polygons  of Figure \ref{3squared} are really different surfaces, because the one on the right is foliated, in the vertical direction, by closed geodesics of length $3$, while the one on the left is not. 
%%%%%%%%%%%%%%%%%%%%%%%
\subsection{Strata}\label{subsection_strata}
The previous discussion suggests lumping together all translation surfaces with the same number of singularities, and equal angles around each singularities. The set of all such translation surfaces is called a \textbf{stratum\index{stratum}} (term coined by Veech  \cite{Veech90}, to be explained later). It is usually denoted $\mathcal{H}(k_1, \ldots, k_n)$, where $2\pi (k_i +1)$ is the angle around the $i$-th singular point. For instance $\mathcal{H}(2)$ is the stratum of surfaces with only one singular point, of angle $6\pi$, while $\mathcal{H}(1,1)$ is the stratum of surfaces with two singular points, each with angle $4\pi$. 

Note that if two translation surfaces lie in the same stratum, they must have the same genus. To see this, let us assume, to begin with,  that $X$ is obtained by identifying the sides of a connected $2N$-gon, with the vertices identifying into the $n$ singularities of an Abelian differential $\omega$.

Then the Euler characteristic $\chi (X)$  of $X$ is $1-N+n$.

Let us say the angle around each singularity $x_i$ is $2(k_i +1)\pi$, then the sum of the interior angles of the polygon equals the sum of the angles around the singularities, that is, 
\[ 
\sum_{i=1}^{n} 2(k_i +1)\pi = (2N-2)\pi,
 \]
whence 
\[
\sum_{i=1}^{n} k_i = N-1-n = -\chi (X) = 2 \mbox{genus}(X)-2.
\]

Now, assume that $X$ is obtained by identifying the sides of a  $2N$-gon with $p$ connected components, each with $N_i$ vertices, for $i=1,\ldots, p$. Then
\[ 
\chi (X)= p - \frac{1}{2}\sum_{i=1}^{p}N_i + n,
 \] 
and the sum of the interior angles of the polygons is 
\[ 
\sum_{i=1}^{p}(N_i-2) \pi = \sum_{i=1}^{n} 2(k_i +1)\pi
 \]

whence
\[ 
\sum_{i=1}^{n} k_i = \frac{1}{2}\sum_{i=1}^{p}N_i -p -n = -\chi (X).
 \]

For instance, the stratum $\mathcal{H}(0)$ consists of all flat tori. In genus $2$, we have the strata $\mathcal{H}(2)$ or  $\mathcal{H}(1,1)$, and those are the only strata of genus $2$, because the only way the number $2$ can be partitionned is as $1+1$ or $2+0$. The surfaces in Figures \ref{doublePentagone} and \ref{3squared} lie in 
$\mathcal{H}(2)$, while the surface in Figure \ref{4squared} lies in 
$\mathcal{H}(1,1)$.

The regular $4n$-gon, and the double $(2n+1)$-gon, lie in the stratum $\mathcal{H}(2n-2)$. The regular $(4n+2)$-gon lies in the stratum $\mathcal{H}(n-1, n-1)$. On the other hand, square-tiled surfaces are dense in every stratum.

 Now we want to define what it means  for translation surfaces to be close, that is, we are looking for a topology on the set of translation surfaces. 
In fact we can do even better: we may find local coordinates on a given stratum, thus making the stratum (almost) a manifold.

 \subsection{Period coordinates}
Fix a basis for the $\Z$-module $H_1(X, \Sigma, \Z)$. Such a basis may be chosen as (the relative homology classes of)  $\alpha_1, \ldots, \alpha_{2g}, c_1, \ldots, c_{n-1}$, where $g$ is the genus of $X$,  $\alpha_1, \ldots, \alpha_{2g}$ are simple closed curves (based at $x_1$) which generate the absolute homology 
 $H_1(X,  \Z)$, and for each $i=1, \ldots, n-1$, $c_i$ is a simple arc joining $x_1$ to $x_{i+1}$.

 \begin{definition}
 	The \textbf{period coordinates\index{period coordinate}\index{period!coordinate}} of the holomorphic differential $\omega$, with respect to the basis $\alpha_1, \ldots, \alpha_{2g}, c_1, \ldots, c_{n-1}$, are the $2g+n-1$ complex numbers
 	\[ 
 	\int_{\alpha_1} \omega, \ldots, \int_{\alpha_{2g}} \omega, 
 	\int_{c_1} \omega, \ldots, \int_{c_{n-1}} \omega.
 	 \]
 \end{definition}
 The complex numbers $\int_{\alpha_1} \omega, \ldots, \int_{\alpha_{2g}}\omega$ are called the \textbf{absolute periods\index{absolute period}\index{period!absolute}} of $\omega$, because the homology classes 
 $\left[\alpha_1\right] , \ldots, \left[\alpha_{2g}\right]$ live in the absolute homology $H_1 (X, \Z)$, while the complex numbers $\int_{c_{n-}1} \omega, \ldots, \int_{c_{n-1}}\omega$ are called the 
 \textbf{relative periods\index{relative period}\index{period!relative}} of $\omega$, because the homology classes 
 $\left[ c_1 \right] , \ldots, \left[ c_{n-1} \right]$ live in the relative homology 
 $H_1(X, \Sigma, \Z)$. 
 
 Why this is actually a set of local coordinates on the stratum of $X$ is a theorem of \cite{Veech86} (see also \cite{FM}, section 2.3, and also \cite{ZK}, \citen{Veech90}). Let us briefly explain the idea. Assume two holomorphic differential $\omega_1$ and $\omega_2$
lie in the same stratum and have the same period coordinates.  Assume that 
 $\omega_1$ and $\omega_2$ are close enough, in some sense, that they can be considered as smooth (not necessarily holomorphic) differential forms on the same surface $X$ (seen as a differentiable manifold), with the same homology basis $\alpha_1, \ldots, \alpha_{2g}, c_1, \ldots, c_{n-1}$. Each holomorphic differential $\omega_1$ and $\omega_2$ gives you a Euclidean metric on $X$, for which geodesic representatives of $\alpha_1, \ldots, \alpha_{2g}, c_1, \ldots, c_{n-1}$ can be found. Cut $X$ open along these geodesics, in each of the two Euclidean metrics,  you get a pair of $2(2g+2n-1)$-gons whose sides are represented by the period coordinates of the holomorphic differentials. Since $\omega_1$ and $\omega_2$
 have the same period coordinates, the two polygons are isometric, by a translation, so $\omega_1$ and $\omega_2$ give the same translation structure.
 
 \textbf{Example:} assume $(X,\omega)$ is a square-tiled surface. Then the homology basis may be chosen so that the curves $\alpha_1, \ldots, \alpha_{2g}, c_1, \ldots, c_{n-1}$ lie in the sides of the squares.  Thus the period coordinates lie in $\Z\left[ i\right] $.  Conversely, if all period coordinates lie in $\Z\left[ i\right] $, then $X$ may be cut into a plane polygon, all of whose vertices lie in $\Z\left[ i\right] $, so $X$ is actually a square-tiled surface.

\subsection{Dimension of the strata} 
 Thus we know the (complex) dimension of the stratum $\mathcal{H}(k_1, \ldots, k_n)$ is $2g+n-1$. The largest stratum is the one with $2g-2$ singularities of index $1$, its dimension is $4g-3$. The smallest stratum is the one with only one singularity of index $2g-2$, its dimension is $2g$.
 The strata are not disconnected from each other, each stratum but the largest one lies in the closure of one or several of the larger ones: for instance, if we have a sequence $(X, \omega_m)$ of translation surfaces in the stratum 
$\mathcal{H}(k_1, \ldots, k_n)$, and  $\int_{c_1} \omega_m \longrightarrow 0$  when $m \longrightarrow \infty$, then the limit surface lies in 
 $\mathcal{H}(k_1 +k_2,k_3, \ldots, k_n)$, because the singularites $x_1$ and $x_2$ merge in the limit. 
  This is the reason they are called strata, because  they are nested in each other's closure, like matriochka. The union of all strata, with the topology given by the period coordinates, is called the 
  \textbf{moduli space of translation surfaces of genus $g$\index{moduli space of translation surfaces}\index{moduli space!translation surfaces}}, denoted $\mathcal{H}_g$. It contains the largest (or principal) stratum, denoted $\mathcal{H}(1^{2g-2})$, as a dense open subset, so its dimension is $4g-3$.
  
  While the global topology of $\mathcal{H}_g$ is mysterious, we have a simple compactness criterion (\textbf{Mumford's compactness theorem}, \cite{Mumford}) for subsets of $\mathcal{H}_g$: a sequence $X_n$ of elements of $\mathcal{H}_g$, of uniformly bounded area, goes to infinity if and only if there exist closed geodesics in $X_n$ whose lengths go to zero.

 \subsection{Quadratic differentials}
 The theory of translation surfaces becomes much more interesting when you think about it in connection with Teichm\H{u}ller theory.
 %%%
 
 The set of complex structures on surfaces of  genus  $g$ is a well-studied object, going  back to Riemann (see \cite{Riemann}, reprinted in \cite{Riemann_complet}, pp.88-144, or, for the faint of heart, \cite{Hubbard}). It is worth noting that Riemann actually got there when studying Abelian differentials.
 
 It is called the \textbf{moduli space of complex structures of genus $g$\index{moduli space of complex structures}\index{moduli space!complex structures}}, and denoted $\mathcal{M}_g$. It has a topology, but actually it has much more structure than that,
 in fact it is, except at some special points which correspond to very symetrical surfaces, a complex manifold of dimension $3g-3$, provided $g>1$.
 The genus one case is special and will be discussed in more detail in Section \ref{moduli space of the torus}.  The cotangent space to moduli space at a given $X$ is the vector space of holomorphic quadratic differentials (see \cite{Hubbard}, Theorem 6.6.1). 
 
Recall that we have just computed the dimension of $\mathcal{H}_g$, which we found to be $4g-3$, while  the cotangent bundle of $\mathcal{M}_g$ has dimension $6g-6$, so the squares of Abelian differentials have codimension $2g-3$ in the space of quadratic differentials.
 
 \subsubsection{Dimension of strata of quadratic differentials}
 Following \cite{LanneauHDR}, 
 we denote $\mathcal{Q}_g$ the space of 
 all quadratic differentials of genus $g$, and 
 $\mathcal{Q}(k_1,\ldots, k_n)$ 
 the set of quadratic differentials with $n$ zeroes, of respective orders $k_1,\ldots, k_n$, which are not squares.

 Computations identical to those of Subsection \ref{subsection_strata} show that $\mathcal{Q}_g$ stratifies as the union of $\mathcal{Q}(k_1,\ldots, k_n)$, with $k_1+\ldots+k_n = 4g-4$.

 It is interesting to compute the dimension of the strata of non-square quadratic differentials, because there is one notable difference with the Abelian case. 
  In the case of the Abelian stratum $\mathcal{H}(k_1,\ldots, k_n) \subset \mathcal{H}_g$, we cut our surface into a $2(2g+n-1)$-gon, and basically said that sides being pairwise equal,  you need $2g+n-1$ complex parameters to determine an element of $\mathcal{H}(k_1,\ldots, k_n)$. Let us  try to apply the same argument to quadratic differentials.
  
  Let us take an element of $\mathcal{Q}(k_1,\ldots, k_n)$, and cut  it 
  into a polygon, whose vertices correspond to the singularities of the quadratic differential. Let us assume for simplicity that the polygon is connected. Then the sum  of the interior angles of the polygon equals the sum of the angles around the singularities, which is 
  $\sum_{i=1}^{n}(k_i +2)\pi$. Thus the polygon has 
  $\sum_{i=1}^{n}(k_i+2)  +2= 4g-4 +2n +2$ edges. Since the edges are pairwise equal, this gives us $2g+n-1$ complex parameters. But then, the largest stratum, with $n= 4g-4$ and $k_i=1, i = 1,\ldots, 4g-4$, would have dimension $6g-5$, instead of $6g-6$ as befits the cotangent bundle of $\mathcal{M}_g$.
  
 To understand this, we must take a closer look at how we identify the sides of a polygon. Recall that we made  the convention that we orient the sides of a polygon so that the interior of the polygon lies to the left. With this convention, a necessary and sufficient condition for plane vectors $v_1, \ldots, v_n$ to be the oriented edges of a polygon is that $v_1+ \ldots+ v_n=0$. 
 
 Next, recall that we identify the sides in pairs, in such a way that the resulting surface is orientable, so each edge is glued to another edge in such a way that the arrows on the edges do not match. 
 
 Recall, furthermore,  that we glue each edge to a parallel edge of equal length, so  there exists some permutation $\sigma$  of $\left\lbrace 1,\ldots, n \right\rbrace $ such that $v_{\sigma (i)}= \pm v_i$, for all $i=1,\ldots, n$. Since we must oppose the arrows when gluing, there are two possible cases: either $v_{\sigma (i)}= - v_i$, in which case the two edges are glued by translation, or $v_{\sigma (i)}=  v_i$, in which case the edges are glued with a half-turn. 
 
 In the former case, we may choose $v_i$ freely without being constrained by the relation 
 $v_1+ \ldots+ v_n=0$, while in the latter case, it imposes a condition on $v_i$:  
 \[
 -2v_i = v_1+ \ldots+ \hat{v_i}+\ldots+ \hat{v}_{\sigma (i)}+\ldots+ v_n
 \]
where the hats mean that  $v_{\sigma (i)}$ and $ v_i$ are omitted in the sum. 

So, when some pair of sides is glued with a half-turn, this pair of sides is determined by the others, so we lose one complex parameter. Now, saying that the quadratic differential is not a square is precisely saying that some pair of sides is glued with a half-turn, since we have seen that when every side is glued by translation, we get an Abelian differential. Therefore,  the dimension of the stratum $\mathcal{Q}(k_1,\ldots, k_n)$ is $2g+n-2$.
 
 \subsection{Another look at the dimension of $\mathcal{H}_g$}
 Here is another way to understand the fact that $\dim \mathcal{H}_g=4g-3$.
 A translation structure, or holomorphic differential, consists of the following data: a complex structure on a surface $X$, and a complex-valued differential form, holomorphic for the given complex structure. As we have seen, a complex structure is determined by $3g-3$ complex parameters. 
 
 Besides, the vector space of  holomorphic (for a given complex structure)  differentials has complex dimension $g$. This is  because, given a complex structure $X$, by the Hodge theorem, the complex vector space $H^1(X, \C)$, which has dimension $2g$, is the direct sum of two isomorphic summands, the subspace of holomorphic differentials, and the subspace of anti-holomorphic differentials.
 Thus a translation structure is determined by $3g-3+g=4g-3$ complex parameters, and $ \mathcal{H}_g$ may be viewed as a real vector bundle over 
 $ \mathcal{M}_g$, with fiber $H^1(X,\R)$.

\section{The Teichm\H{u}ller geodesic flow} 
 
\subsection{ Teichm\H{u}ller's theorem}\label{TeichThm}
 Take two complex structures $X_1$ and $X_2$ of genus $g$. Then Teichm\H{u}ller's theorem says there exists 
 a holomorphic quadratic differential $q$ on $X_1$,  and a number $t \in \R$, such that, assuming for simplicity that $q$ is the square of a holomorphic differential $\omega$,  the  1-differential form with real part $e^t \mbox{Re } \omega$ and imaginary part $e^{-t} \mbox{Im } \omega$ is holomorphic with respect to the complex structure  $X_2$. 
 
 Then one may define, at least locally, a distance function by  $d(X_1, X_2):= |t|$, and this distance comes with geodesics (shortest paths): if $s \in \left[ 0,t\right] $, and $X_s$ is the complex structure which makes the complex differential  with real part $e^s \mbox{Re } \omega$ and imaginary part $e^{-s} \mbox{Im } \omega$  holomorphic, then $d(X_1, X_s)= |s|$, so the path $s \mapsto X_s$ is a geodesic. 
 
 Teichm\H{u}ller's distance has an infinitesimal expression, like  a Riemannian metric, and more precisely, it is a Finsler metric  (a Finsler metric is to a Riemannian metric what a norm is to a Euclidean norm). This was discovered by Teichm\H{u}ller,  see \cite{Teichmueller}, p.26, translated in \cite{Teichmueller_traduit}, or just \cite{Hubbard}, Theorem 6.6.5. 
 
 A quadratic differential $\omega$ induces a Riemannian (flat except at the singularities) metric on $X$, in particular it comes with a volume form, so it makes sense to evaluate the total volume of $X$ with respect to $q$, denoted $\mbox{Vol}(X, q)$. The map 
 $(X, q) \mapsto \mbox{Vol}(X, q)$, restricted to the tangent space to $\mathcal{M}_g$ at $X$, is a norm: it is 1-homogeneous, positive except at $\omega=0$, and satisfies the triangle inequality. The first two properties are immediate, the last one may warrant a short proof. 
 
 First, observe that if the quadratic differential $q$ is $f(z)dz^2$ in some chart, then the volume form of $q$, seen as a flat metric, is 
 $\frac{i}{2}\left| f(z)\right| dz \wedge d\bar{z}$, because 
 $dz \wedge d\bar{z}= -2i dx \wedge dy$. This expression is coordinate-invariant, because if $q$ is $g(z)dz^2$ in some other chart, and $T$ is the transition map, then $f(z)=g(T(z))T'(z)^2$, so 
 \[ 
 \left| f(z)\right| dz \wedge d\bar{z} =
 \left| g(T(z))\right|   T'(z)\overline{T'(z)} dz \wedge d\bar{z}=
 \left| g(T(z))\right|   (T'(z) dz)\wedge (\overline{T'(z)} d\bar{z})
  \]
 so the total volume of $q$ may be calculated in local coordinates.
 Now,  assume we have two quadratic differentials $q_1=f_1(z)dz^2$ and $q_2=f_2(z)dz^2$, so $q_1+q_2= (f_1+f_2)dz^2$, the triangle inequality in $\C$ yields $\left| f_1+f_2\right| \leq \left| f_1\right|+\left| f_2\right|$, and, integrating over $X$, we get 
 $\mbox{Vol}(X, q_1+q_2) \leq \mbox{Vol}(X, q_1)+\mbox{Vol}(X, q_2)$.

% Teichm\H{u}ller's norm is not a Euclidean norm, however, as soon as the genus $g$ is $>1$ (this is a subtle fact; the norm is $C^1$ but not $C^2$, hence it is not Euclidean. See \cite{Hubbard}, Section 7.4).
 
 This norm endows $\mathcal{M}_g$ with a metric: if we have a $C^1$ path $X(t)$ in $\mathcal{M}_g$, its derivative $\dot{X}(t)$ is a quadratic differential, holomorphic at $X(t)$, and we just say that the velocity of $X$ at time $t$ is $\mbox{Vol}(X, \dot{X}(t))$. This is not a Riemannian metric, because the norm is not Euclidean, the specific name is Finsler metric. In any case it induces a geodesic flow on the cotangent bundle of $\mathcal{M}_g$: start at a complex structure $X$, in the direction given by a quadratic differential  $\omega$, and go to $X_t$ such that the  complex quadratic differential  with real part $e^t \mbox{Re } \omega$ and imaginary part $e^{-t} \mbox{Im } \omega$ is holomorphic with respect to the complex structure  $X_t$.

Most interesting to us here is the fact that  $\mathcal{H}_g$, which is a submanifold of the cotangent bundle of $\mathcal{M}_g$, is invariant by the geodesic flow.
In fact, we shall see, in Part \ref{gl2}, that $\mathcal{H}_g$ is invariant by a much more interesting action. 

\subsection{Masur's criterion}
In dynamical systems, we have a saying, % \cite{party}...
%\bigskip
 which goes ``sow in the parameter space, reap in the phase space". This is particularly relevant to translation surfaces, because the parameter space $\mathcal{H}_g$ comes with so much structure: a geodesic flow, invariant measures, affine coordinates...
Masur's criterion (possibly the single most useful result in the theory) is a striking example. 
\begin{theorem}[Masur's criterion, \cite{Masur92}]
Given an Abelian differential $\omega$,	the orbit of $\omega$ under the Teichm\H{u}ller geodesic flow is recurrent if and only if the vertical flow of $\omega$ is uniquely ergodic.
\end{theorem}
The idea is very clearly explained in \cite{Monteil}.
We'll see several applications of Masur's criterion, here is one: 
\begin{theorem}[\cite{KMS}]\label{thmKMS}
Given an Abelian differential $\omega$, for almost every $\theta$, the vertical flow of $e^{i\theta}\omega$ is uniquely ergodic.
\end{theorem}
Once we have Masur's criterion, it is easy to get a weak version of Theorem \ref{thmKMS}: \textit{for almost every} Abelian differential $\omega$, for almost every $\theta$, the vertical flow of $e^{i\theta}\omega$ is uniquely ergodic. The proof is basically just Poincaré's Recurrence Theorem. Of course, to apply it we need an invariant measure of full support and finite total volume, so we apply it, not on $\mathcal{H}_g$, but on the subset $\mathcal{H}^1_g \subset \mathcal{H}_g $ of Abelian differentials of area $1$. It is a theorem by Masur and Veech (\cite{Masur}, \cite{Veech82}) that the Lebesgue measure induced  on $\mathcal{H}^1_g $ by the period coordinates has finite total volume. 
%%%%%%%%%%%%%%%%%%%%
\part{Orbits of the $GL^+_2(\R)$-action}\label{gl2}
%%%%%%%%%%%%%%%%%%%
 Every translation surface may be viewed as a polygon in the Euclidean plane with parallel sides of equal length pairwise identified. The group 
$\mathrm{GL}_2^+ (\R )$ acts linearly on polygons, mapping pairs of  parallel sides of equal length to pairs of  parallel sides of equal length, so it acts on translation surfaces. 

Let us take a basis   $\mathcal{B}= (\alpha_1, \ldots, \alpha_{2g}, c_1, \ldots, c_{n-1})$ of $H_1(X, \Sigma, \Z)$. If we cut $X$ into a polygon along the curves $\alpha_1, \ldots, \alpha_{2g}, c_1, \ldots, c_{n-1}$, the period coordinates of $(X,\omega)$ in the basis $\mathcal{B}$ are the sides of the polygon (identifying a complex number and its affix). So, if 
$A \in \mathrm{GL}_2^+ (\R )$, the period coordinates, in the basis $\mathcal{B}$, of the surface $A.(X,\omega)$, are the complex numbers 
	\[ 
A.\int_{\alpha_1} \omega, \ldots, A.\int_{\alpha_{2g}} \omega, 
A.\int_{c_1} \omega, \ldots, A.\int_{c_{n-1}} \omega,
\]
where $A$ acting on a complex numbers means that it acts on its affix. The matrix $A$ induces a homeomorphism (which is actually a diffeomorphism outside the singularities), from the translation surface $(X,\omega)$, to the surface $A.(X,\omega)$.

First, let us observe that strata are invariant under $GL^+_2(\R)$. This is because acting by an element of $GL^+_2(\R)$ on a polygon does not change the order in which sides are identified, which determines the singularities, and the angles around the singularities. 

The $\mathrm{SL}_2 (\R)$-action contains the geodesic flow, in the following sense: 
the geodesic, with respect to the Teichm\H{u}ller metric, which has initial position $X$, and initial velocity $\omega$, is the orbit of $(X, \omega)$
under the diagonal subgroup of $\mathrm{SL}_2 (\R)$,
\[ 
\left\lbrace 
\left( 
\begin{array}{cc}
e^{t} & 0 \\
0 & e^{-t}
\end{array}
\right) 
: t \in \R
\right\rbrace .
\]

It is usually preferred to deal with the $\mathrm{GL}_2^+ (\R)$-action rather  than with the geodesic flow, for the following reason, which is a theorem by Eskin and Mirzakhani \cite{EM}, henceforth referred to as the Magic Wand, after \cite{Zorich_gazette}. While the invariant sets for the geodesic flow may be wild (for instance, Teichm\H{u}ller disks contain geodesic laminations, which are locally the product of a Cantor set with an interval), the closed invariant sets of the $\mathrm{GL}_2^+ (\R)$-action are always nice submanifolds (again, ignoring singularities which occur at very symmetrical surfaces), locally defined by affine  equations in the period coordinates. Any attempt at informally describing the proof, assuming the author could do it, would be as long as this paper. A recurring theme is the analogy with Ratner's theorems. 

Now, by a result by Masur and (independantly) Veech (\cite{Masur, Veech82}), the geodesic flow is ergodic, with respect to a measure of full support in $\mathcal{H}_g$, so almost every orbit is dense; consequently, almost every $\mathrm{GL}_2^+ (\R)$-orbit is dense. In fact,   any stratum supports a full-support, ergodic measure. This means finding interesting (meaning: other than strata closures and $\mathcal{H}_g$ itself), closed invariant subsets won't be easy.
%%%%%%%%%%%%%%%%%%%%%%%%%%%
\section{Veech groups}

It may happen that for some polygon $P$ and some element $A$ of $\mathrm{GL}_2^+ (\R )$, both $P$ and $A.P$, after identifications,  are just the same translation surface $X$. This happens when $A.P$ can be cut into pieces, and the pieces re-arranged, by translations, into $A$.

  The \textbf{Veech group\index{Veech group}\index{group!Veech}} of $(X, \omega)$ is the subgroup of $\mathrm{GL}_2^+ (\R )$  which preserves $X$. Since the Veech group must preserve volumes, it is a subgroup of $\mathrm{SL}_2 (\R )$, and it turns out to be a Fuchsian group (see \cite{HS}, Lemma 2). We sometimes denote it $SL(X, \omega)$ after McMullen.
  
Take a matrix $A$ in $SL(X, \omega)$, then it defines a homeomorphism from $X$ to itself, which we again denote $A$ for simplicity. Let $A_*$ be the linear automorphism of $H_1 (X, \Sigma, \Z)$ induced by $A$. Then the period coordinates of $(X, \omega)$, in a basis $\mathcal{B}$ of $H_1 (X, \Sigma, \Z)$, are exactly the period coordinates of $A.(X, \omega)$ in the basis $(A_*)^{-1}(\mathcal{B})$. Thus, modulo a change of basis, the set of periods is invariant under  the Veech group.   
  
   Note that
  $SL(X, \omega)$ is never cocompact, by the following argument: a translation surface always has a closed geodesic, for topological reasons; up to a rotation, we may assume this geodesic is horizontal. Let $l$ be its length. Applying the Teichm\H{u}ller geodesic flow for a time $t$, the length of the closed geodesic becomes $e^{-t}l$, so the Teichm\H{u}ller geodesic goes to infinity in the moduli space, therefore $\mathrm{SL}_2 (\R )/SL(X, \omega)$ cannot be compact.
  
  The next best thing to being co-compact is to have finite co-volume, Veech groups of finite co-volume are going to play an important part. Note that for a Fuchsian group $\Gamma$ to have finite co-volume, its ends at infinity must be finitely many cusps (see Section 4.2 of \cite{S.Katok}).

   Now let us see some examples. 
 \subsection{The Veech group of the square torus is $\mathrm{SL}_2 (\Z )$}

  Figure \ref{T,S} shows that the matrices 
  \[ 
 T= \begin{pmatrix}
  1 & 1  \\ 
  0 & 1  
  \end{pmatrix}
  \mbox{ and }
  S=\begin{pmatrix}
  0 & -1  \\ 
  1 & 0  
  \end{pmatrix},
   \]
   which generate $\mathrm{SL}_2 (\Z )$, lie in the Veech group of the square torus. Conversely, every element of the Veech group preserve the set of periods, so, in the case of the torus, it must preserve $\Z\left[ i\right] $, hence
    it must lie in $\mathrm{SL}_2(\Z)$.
  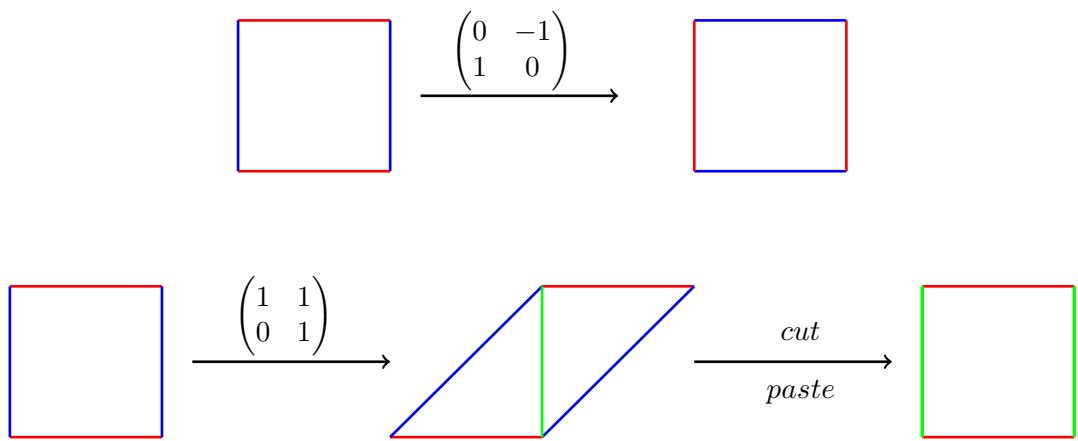
\begin{figure}[h!]
  	\begin{center}
  		\begin{tikzpicture}[scale=2]
  		\draw [red, line width=1pt] (0,0)--(1,0);
  		\draw [red, line width=1pt] (0,1)--(1,1);
  		\draw [blue, line width=1pt] (0,0)--(0,1);
  		\draw [blue, line width=1pt] (1,0)--(1,1);
  		\draw [->, line width=1pt]  (1.2,0.5)--(2.5,0.5);
  		\draw (1.8,0.8) node {$\begin{pmatrix}
  			0 & -1  \\ 
  			1 & 0  
  			\end{pmatrix}$};
  		\draw [blue, line width=1pt] (0+3,0)--(1+3,0);
  		\draw [blue, line width=1pt] (0+3,1)--(1+3,1);
  		\draw [red, line width=1pt] (0+3,0)--(0+3,1);
  		\draw [red, line width=1pt] (1+3,0)--(1+3,1);

  		\end{tikzpicture}
  	\end{center}
  	\,
  	\begin{center}
  		\begin{tikzpicture}[scale=2]
  		\draw [red, line width=1pt] (0,0)--(1,0);
  		\draw [red, line width=1pt] (0,1)--(1,1);
  		\draw [blue, line width=1pt] (0,0)--(0,1);
  		\draw [blue, line width=1pt] (1,0)--(1,1);
  		\draw [->, line width=1pt]  (1.2,0.5)--(2.5,0.5);
  		\draw (1.8,0.8) node {$\begin{pmatrix}
  			1 & 1  \\ 
  			0 & 1  
  			\end{pmatrix}$};
  		\draw [red, line width=1pt] (2.5,0)--(3.5,0);
  		\draw [red, line width=1pt] (3.5,1)--(4.5,1);
  		\draw [blue, line width=1pt] (2.5,0)--(3.5,1);
  		\draw [blue, line width=1pt] (3.5,0)--(4.5,1);
  		\draw [green, line width=1pt] (3.5,0)--(3.5,1);
  		
  		\draw [->, line width=1pt]  (4.5,0.5)--(5.8,0.5);
  		\draw (5.2,0.7) node {$cut$};
  		\draw (5.2,0.3) node {$paste$};
  		
  		\draw [red, line width=1pt] (6,0)--(7,0);
  		\draw [red, line width=1pt] (6,1)--(7,1);
  		\draw [green, line width=1.3pt] (6,0)--(6,1);
  		\draw [green, line width=1.3pt] (7,0)--(7,1);
  		
  		\end{tikzpicture}
  		\caption{Action of $T$ and $S$ on the torus $\T^2$} \label{T,S}
  	\end{center}
  \end{figure}
  \subsection{Veech groups of the regular polygons}
  It is obvious that the rotation by $\pi/n$ is in the Veech group of the regular $2n$-gon. If you view the double $(2n+1)$-gon as a star, as in Figure \ref{doublePentagone}, you see that the rotation by $2\pi/(2n+1)$ is in the Veech group of the double $(2n+1)$-gon. 
  
  Figure \ref{octagonL} explains how to find a parabolic element in the Veech group of the regular $n$-gon: first, the regular $n$-gon may be cut and pasted into a slanted stair-shape. Then the matrix
    \[
  \left( 
  \begin{array}{cc}
  1 & -2\cot \frac{\pi}{n} \\
  0 & 1
  \end{array} 
  \right) \]
  brings the slanted stair-shape to its mirror image with respect to the vertical axis, and the mirror image may be cut and pasted back to a regular $n$-gon.
  
  \begin{figure} [h!]
  	\begin{center}
  		
  		\definecolor{ffqqqq}{rgb}{1,0,0}
  		\begin{tikzpicture}[line cap=round,line join=round,>=triangle 45,x=1cm,y=1cm, scale=0.55]
  		\clip(-4.32,-5) rectangle (22.44,5);
  		\draw [line width=1pt] (3.695518130045147,1.5307337294603591)-- (3.6955181300451465,-1.530733729460359)-- (1.5307337294603591,-3.695518130045147)-- (-1.5307337294603591,-3.695518130045147)-- (-3.695518130045147,-1.5307337294603591)-- (-3.695518130045147,1.5307337294603591)-- (-1.5307337294603591,3.695518130045147)-- (1.530733729460359,3.6955181300451465);
  		\draw [line width=1pt] (1.530733729460359,3.6955181300451465)-- (3.695518130045147,1.5307337294603591);
  		\draw [line width=1pt,color=ffqqqq] (-1.5307337294603591,-3.695518130045147)-- (3.6955181300451465,-1.530733729460359);
  		\draw [line width=1pt,color=ffqqqq] (1.5307337294603591,-3.695518130045147)-- (6.756985588965865,-1.530733729460359);
  		\draw [line width=1pt,color=ffqqqq] (3.6955181300451465,-1.530733729460359)-- (6.756985588965865,-1.530733729460359);
  		\draw [line width=1pt,color=ffqqqq] (1.5307337294603593,-3.695518130045147)-- (8.921769989550652,-3.6955181300451465);
  		\draw [line width=1pt,color=ffqqqq] (6.756985588965866,-1.530733729460359)-- (14.148021849056159,-1.5307337294603587);
  		\draw [line width=1pt,color=ffqqqq] (8.921769989550652,-3.6955181300451465)-- (14.148021849056159,-1.5307337294603587);
  		\draw [line width=1pt,color=ffqqqq] (6.756985588965866,-1.530733729460359)-- (14.14802184905616,1.5307337294603593);
  		\draw [line width=1pt,color=ffqqqq] (14.148021849056157,-1.5307337294603587)-- (21.539058109146453,1.5307337294603596);
  		\draw [line width=1pt,color=ffqqqq] (14.14802184905616,1.5307337294603593)-- (21.539058109146453,1.5307337294603596);
  		\draw [line width=1pt,dash pattern=on 1pt off 1pt] (-3.695518130045147,-1.47)-- (3.6955181300451465,-1.530733729460359);
  		\draw [line width=1pt,dash pattern=on 1pt off 1pt] (3.695518130045147,1.5307337294603591)-- (-3.663125929752753,1.563125929752753);
  		\draw [line width=1pt,dash pattern=on 1pt off 1pt] (-3.695518130045147,-1.47)-- (3.695518130045147,1.5307337294603591);
  		\draw [line width=1pt,dash pattern=on 1pt off 1pt] (-3.663125929752753,1.563125929752753)-- (1.5481259297527528,3.678125929752753);
  		\draw [line width=1pt,dash pattern=on 1pt off 1pt,color=ffqqqq] (14.148021849056162,-1.5307337294603587)-- (14.148021849056157,1.5307337294603587);
  		\draw [line width=1pt,dash pattern=on 1pt off 1pt,color=ffqqqq] (6.756985588965863,-1.5307337294603591)-- (8.921769989550652,-3.6955181300451465);
  		\begin{scriptsize}
  		%\draw[color=black] (2.37,2.625) node {$j$};
  		\draw[color=ffqqqq] (1.38,-3) node {$1$};
  		\draw[color=ffqqqq] (6,-2.685) node {$3$};
  		\draw[color=ffqqqq] (4,-2) node {$2$};
  		%\draw[color=ffqqqq] (5.46,-3.825) node {$m'$};
  		\draw[color=ffqqqq] (9,-2.5) node {$4$};
  		%\draw[color=ffqqqq] (11.82,-2.685) node {$n$};
  		\draw[color=ffqqqq] (12,-0.075) node {$5$};
  		%\draw[color=ffqqqq] (18.33,0.015) node {$p'_{1}$};
  		%\draw[color=ffqqqq] (17.94,1.425) node {$q$};
  		\draw[color=black] (-1,-2.7) node {$4$};
  		\draw[color=black] (0.12,2.385) node {$3$};
  		\draw[color=black] (2,-0.045) node {$5$};
  		\draw[color=black] (-2,-0.045) node {$6$};
  		\draw[color=black] (-1.1,3) node {$2$};
  		\draw[color=ffqqqq] (16,0.435) node {$6$};
  		%\draw[color=ffqqqq] (7.68,-2.505) node {$g_1$};
  		\end{scriptsize}
  		\end{tikzpicture}
  		
  		\caption{The octagon is cut and pasted into a  slanted stair shape} \label{octagonL}
  	\end{center}
  \end{figure}
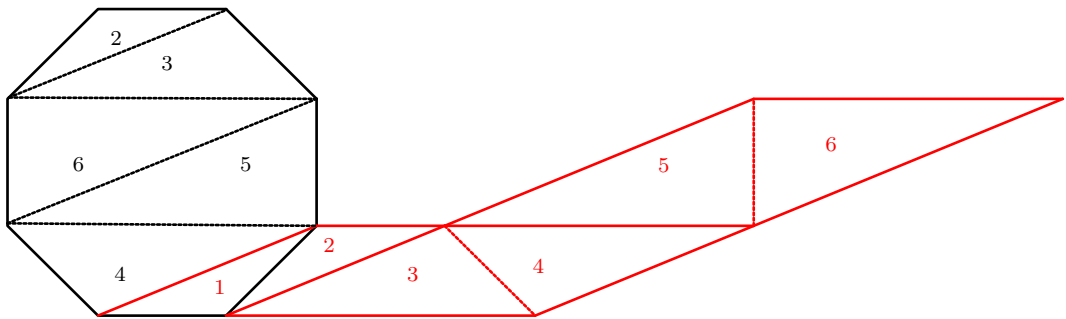

 In fact, Veech proved in \cite{Veech89}
 that the Veech groups of the aforementionned surfaces are generated by the rotation, and by the parabolic element.
 
 \subsection{Veech groups of square-tiled surfaces}
 The Veech group of a square-tiled surface, if we assume that the squares are copies of the unit square in $\R^2$, is a subgroup  of $\mathrm{SL}_2(\Z)$, because any element of the Veech group must preserve the set of periods, hence it must preserve $\Z\left[ i\right] $. In fact it has finite index in  $\mathrm{SL}_2(\Z)$ (see \cite{GJ}).

 Figure \ref{stabilisateur} shows why the parabolic element $T^2$
 % = \begin{pmatrix}
% 1 & 2  \\ 
% 0 & 1  
% \end{pmatrix}$
 lies in the Veech group of the surface $St(3)$.
 Figure \ref{pasVeech} shows why the parabolic element $T$
does not lie in  the Veech group of the surface $St(3)$, in fact it takes $St(3)$ to the surface on the right in Figure \ref{3squared}. The Veech group of $St(3)$ is generated by $T^2$
 and $S$, the rotation of order 4. It has index 3 in $\mathrm{SL}_2(\Z)$, the right cosets are those of $id$, $T$ and its transpose.
 
 In \cite{S1} an algorithm is given to compute the Veech group of any square-tiled surface. It is implemented in the Sagemath package \cite{Delecroix}.
 \begin{figure}[b!]
 	\begin{center} 
 		\begin{tikzpicture}[scale=1.5]		
 		\draw [line width=0.8pt] (0,0)--(2,0)--(2,1)--(1,1)--(1,2)-- (0,2)--(0,1)--(0,0);
 		\draw (0.5,-0.2) node {$a$}; \draw (0.5,2.2) node {$a$};
 		\draw (1.5,-0.2) node {$b$}; \draw (1.5,1.2) node {$b$};
 		\draw (-0.2,0.5) node {$c$}; \draw (2.2,0.5) node {$c$};
 		\draw (-0.2,1.5) node {$d$};\draw (1.2,1.5) node {$d$};
 		\draw (0,0) node {$\bullet$};\draw (1,0) node {$\bullet$};\draw (2,0) node {$\bullet$};\draw (1,1) node {$\bullet$};\draw (1,2) node {$\bullet$};\draw (0,1) node {$\bullet$};\draw (0,2) node {$\bullet$};\draw (2,1) node {$\bullet$};
 		\draw [->, line width=0.6pt]  (1,-0.5)--(1,-2.5);
 		\draw (1.5,-1.5) node {$\begin{pmatrix}
 			1 & 1  \\ 
 			0 & 1  
 			\end{pmatrix}$};
 		
 		\draw [line width=0.8pt] (0,-4)--(2,-4)--(3,-3)--(2,-3)--(3,-2)-- (2,-2)--(1,-3)--(0,-4);
 		\draw (0,-4) node {$\bullet$};\draw (1,-4) node {$\bullet$};
 		\draw (2,-4) node {$\bullet$};
 		%\draw (4,-2) node {$\bullet$};
 		\draw (3,-3) node {$\bullet$};\draw (3,-2) node {$\bullet$};\draw (2,-2) node {$\bullet$};\draw (3,-3) node {$\bullet$}; \draw (1,-3) node {$\bullet$};
 		\draw (1.2,-3.5) node {$\alpha$}; \draw (1.8,-2.5) node {$\beta$};
 		
 		\draw [red] (1,-4)--(1,-3);
 		\draw [blue] (2,-3)--(2,-2);
 		%\draw [blue] (4,-2.5)--(4,-2);
 		\draw (0.5,-0.2-4) node {$a$}; \draw (2.5,2.2-4) node {$a$};
 		\draw (1.5,-0.2-4) node {$b$}; \draw (2.5,1.1-4) node {$b$};
 		\draw (0.3,-3.5) node {$c$}; \draw (2.7,-3.5) node {$c$};
 		\draw (2.7,-2.5) node {$d$}; \draw (1.3,-2.5) node {$d$};

 		\draw  (0.9,-4.4)--(0.9,-5.2);
 		\draw  (0.7,-4.4)--(0.7,-5.2);

 		\draw [line width=0.8pt] (0,-8)--(2,-8)--(2,-7)--(1,-7)--(1,-6)-- (0,-6)--(0,-7)--(0,-8);
 		\draw [red] (2,-8)--(2,-7); \draw [red] (0,-8)--(0,-7);
 		\draw [blue] (0,-7)--(0,-6); \draw [blue] (1,-7)--(1,-6);
 		\draw (0,-8) node {$\bullet$}; \draw (1,-8) node {$\bullet$}; \draw (2,-8) node {$\bullet$}; \draw (0,-7) node {$\bullet$};\draw (1,-7) node {$\bullet$};\draw (0,-6) node {$\bullet$};\draw (1,-6) node {$\bullet$};
 		\draw (0.5,-0.2-8) node {$b$}; \draw (0.5,2.2-8) node {$a$};
 		\draw (1.5,-0.2-8) node {$a$}; \draw (1.5,1.2-8) node {$b$};
 		\draw (-0.2,0.5-8) node {$\alpha$}; \draw (2.2,0.5-8) node {$\alpha$};
 		\draw (-0.2,1.5-8) node {$\beta$};\draw (1.2,1.5-8) node {$\beta$};
 		\end{tikzpicture}	
 		\caption{Action of $T$ on   $St(3)$} \label{pasVeech}
 	\end{center}
 \end{figure}
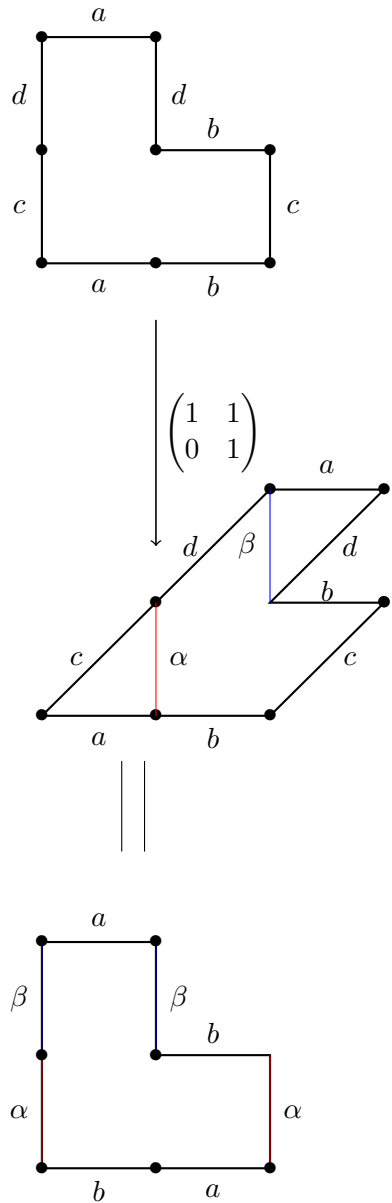
 
 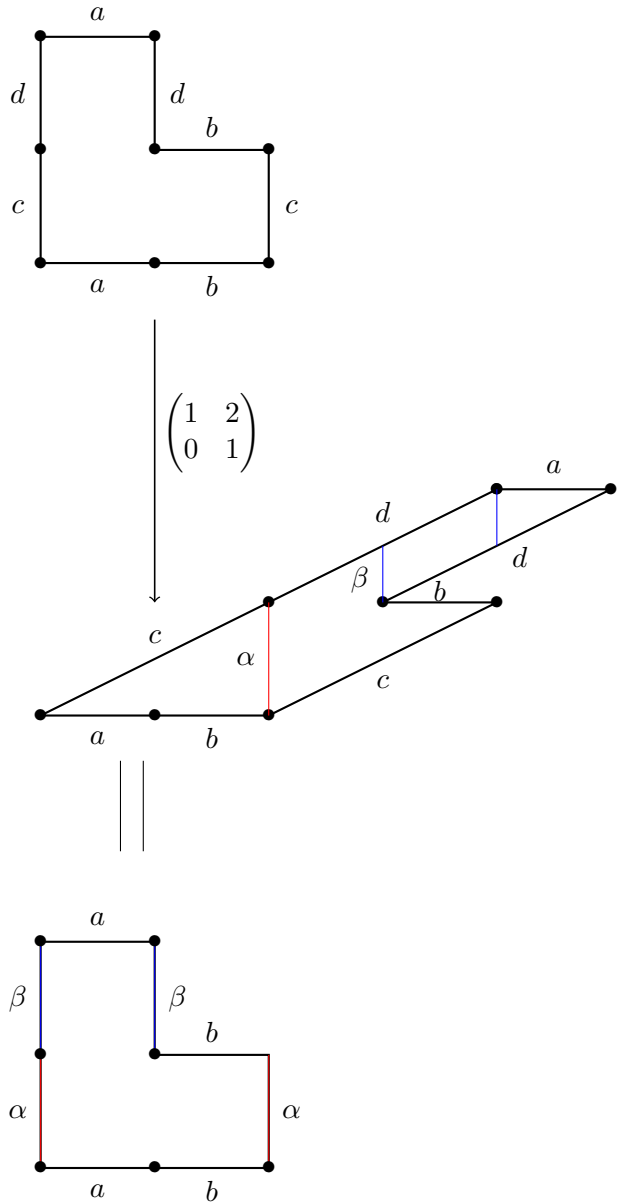
\begin{figure}[b!]
 	\begin{center} 
 		\begin{tikzpicture}[scale=1.5]		
 		\draw [line width=0.8pt] (0,0)--(2,0)--(2,1)--(1,1)--(1,2)-- (0,2)--(0,1)--(0,0);
 		\draw (0.5,-0.2) node {$a$}; \draw (0.5,2.2) node {$a$};
 		\draw (1.5,-0.2) node {$b$}; \draw (1.5,1.2) node {$b$};
 		\draw (-0.2,0.5) node {$c$}; \draw (2.2,0.5) node {$c$};
 		\draw (-0.2,1.5) node {$d$};\draw (1.2,1.5) node {$d$};
 		\draw (0,0) node {$\bullet$};\draw (1,0) node {$\bullet$};\draw (2,0) node {$\bullet$};\draw (1,1) node {$\bullet$};\draw (1,2) node {$\bullet$};\draw (0,1) node {$\bullet$};\draw (0,2) node {$\bullet$};\draw (2,1) node {$\bullet$};
 		\draw [->, line width=0.6pt]  (1,-0.5)--(1,-3);
 		\draw (1.5,-1.5) node {$\begin{pmatrix}
 			1 & 2  \\ 
 			0 & 1  
 			\end{pmatrix}$};
 		
 		\draw [line width=0.8pt] (0,-4)--(2,-4)--(4,-3)--(3,-3)--(5,-2)-- (4,-2)--(2,-3)--(0,-4);
 		\draw (0,-4) node {$\bullet$};\draw (1,-4) node {$\bullet$};\draw (2,-4) node {$\bullet$};\draw (4,-2) node {$\bullet$};\draw (3,-3) node {$\bullet$};\draw (5,-2) node {$\bullet$};\draw (4,-2) node {$\bullet$};\draw (4,-3) node {$\bullet$}; \draw (2,-3) node {$\bullet$};
 		\draw (1.8,-3.5) node {$\alpha$}; \draw (2.8,-2.8) node {$\beta$};
 		
 		\draw [red] (2,-4)--(2,-3);
 		\draw [blue] (3,-3)--(3,-2.5);
 		\draw [blue] (4,-2.5)--(4,-2);
 		\draw (0.5,-0.2-4) node {$a$}; \draw (4.5,2.2-4) node {$a$};
 		\draw (1.5,-0.2-4) node {$b$}; \draw (3.5,1.1-4) node {$b$};
 		\draw (1,-3.3) node {$c$}; \draw (3,-3.7) node {$c$};
 		\draw (3,-2.2) node {$d$}; \draw (4.2,-2.6) node {$d$};

 		\draw  (0.9,-4.4)--(0.9,-5.2);
 		\draw  (0.7,-4.4)--(0.7,-5.2);

 		\draw [line width=0.8pt] (0,-8)--(2,-8)--(2,-7)--(1,-7)--(1,-6)-- (0,-6)--(0,-7)--(0,-8);
 		\draw [red] (2,-8)--(2,-7); \draw [red] (0,-8)--(0,-7);
 		\draw [blue] (0,-7)--(0,-6); \draw [blue] (1,-7)--(1,-6);
 		\draw (0,-8) node {$\bullet$}; \draw (1,-8) node {$\bullet$}; \draw (2,-8) node {$\bullet$}; \draw (0,-7) node {$\bullet$};\draw (1,-7) node {$\bullet$};\draw (0,-6) node {$\bullet$};\draw (1,-6) node {$\bullet$};
 		\draw (0.5,-0.2-8) node {$a$}; \draw (0.5,2.2-8) node {$a$};
 		\draw (1.5,-0.2-8) node {$b$}; \draw (1.5,1.2-8) node {$b$};
 		\draw (-0.2,0.5-8) node {$\alpha$}; \draw (2.2,0.5-8) node {$\alpha$};
 		\draw (-0.2,1.5-8) node {$\beta$};\draw (1.2,1.5-8) node {$\beta$};
 		\end{tikzpicture}	
 		\caption{Action of $T^2$ on  $St(3)$} \label{stabilisateur}
 	\end{center}
 \end{figure}
 %\subsubsection{Parabolic elements in Veech groups}

 \section{Teichm\H{u}ller disks}  
  The orbit of $(X, \omega)$ under $\mathrm{GL}_2^+ (\R )$, or, more properly, its projection to the Teichm\H{u}ller space $\mathcal{T}_g$, is called the \textbf{Teichm\H{u}ller disk\index{Teichm\H{u}ller disk}\index{Teichm\H{u}ller!disk}} 
  of $(X, \omega)$. The concept (under the name ``complex geodesic\index{complex geodesic}", which is still used sometimes) originates in \cite{Teichmueller}, § 121. Here is why it is called a disk. 
  
  Observe that if $\omega$ is a 
 holomorphic differential on a Riemann surface $X$, and  $A \in \mathrm{GL}_2^+ (\R )$
 is a similitude, then $A.\omega$ is still holomorphic with respect to the complex structure of $X$, so if we are looking at the projection to the Teichm\H{u}ller space $\mathcal{T}_g$, the action of $\mathrm{GL}_2^+ (\R )$ factors through 
 the quotient space of $\mathrm{GL}_2^+( \R )$  by similitudes.
  Recall that said quotient space  is the hyperbolic disk $\Hyp ^2$. 
  
  The \textbf{Teichm\H{u}ller curve\index{Teichm\H{u}ller curve}\index{Teichm\H{u}ller!curve}} of $(X, \omega)$ is  the  quotient of the hyperbolic disk $\Hyp ^2$ by  $SL(X, \omega)$, which is a Fuchsian group. The fact that $SL(X, \omega)$ is a Fuchsian group tells us that this quotient is a hyperbolic surface, possibly with some conical singularities at symmetrical surfaces (any symmetry of $X$ is an element of $SL(X, \omega)$ which fixes $X$).

  Since the $\mathrm{GL}_2^+ (\R)$-action contains the Teichm\H{u}ller geodesic flow as a subgroup action, Teichm\H{u}ller curves, which are $\mathrm{GL}_2^+ (\R)$-orbits, are invariant under the geodesic flow. Furthermore, Teichm\H{u}ller curves, as hyperbolic manifolds, are isometrically embedded in $\mathcal{M}_g$ (see \cite{Royden}).
  
  Figure \ref{pasVeech} shows that the two surfaces of Figure \ref{3squared} lie in the same Teichm\H{u}ller curve.
 
\subsection{Example: the Teichm\H{u}ller disk of the torus}\label{moduli space of the torus}
The Teichm\H{u}ller disk of the torus is special because it is also the 
Teichm\H{u}ller \textit{space} of the torus.

A torus, as a complex curve, is $\C /\Lambda$, where $\Lambda = \Z u \oplus \Z v$ is a lattice in $\R^2$, $u$ and $v$ being two non-colinear vectors in $\R^2$. Since $\Z u \oplus \Z v=\Z v \oplus \Z u$, we may assume that $(u,v)$ is a positive basis of $\R^2$. Thus the matrix $M$ whose columns are the coordinates of $u$ and $v$, in that order, lies in $\mathrm{GL}_2^+( \R )$.
The Abelian differential we consider on $\C /\Lambda$ is $dz=dx+idy$.
A matrix $A=\begin{pmatrix}
a & b  \\ 
c & d  
\end{pmatrix} \in \mathrm{GL}_2^+( \R )$
acts on $dz$ by pull-back : 
\[ 
\forall \begin{pmatrix}
x \\
y
\end{pmatrix} \in \R^2, A^* dz \begin{pmatrix}
x \\
y
\end{pmatrix}= 
dz \left( A.\begin{pmatrix}
x \\
y
\end{pmatrix}\right) 
= dz \begin{pmatrix}
ax+by \\
cx+dy
\end{pmatrix}
= ax+by +i(cx+dy)
 \]
 so $A^*dz = a.dx +b.dy +i(c.dx +d.dy)$. The definition of the pull-back means that  the complex 1-form $A^*dz$ on $\C /\Lambda$, which is non-holomorphic unless $A$ happens to be $\C$-linear, becomes the Abelian differential $dz$ on the torus $\C / A.\Lambda$, where $A.\Lambda= \Z A.u \oplus \Z A.v$.
 
 Let $(e_1,e_2 )$ be the canonical basis of $\R^2$. The straight segments 
 $\left\lbrace te_i : 0 \leq t \leq 1 \right\rbrace $, for $i=1,2$, become closed curves in $\C/\Z^2=\C /\Z e_1 \oplus \Z e_2$. We again denote $e_1$ and $e_2$, respectively, the homology classes of those closed curves. For any
 torus $\C /\Lambda$, with $\Lambda = \Z u \oplus \Z v$, 
 $M \in \mathrm{GL}_2^+( \R )$ being the matrix whose columns are the coordinates of $u$ and $v$, in that order, $M$ may be viewed as a diffeomorphism from  $\C/\Z^2$ to $\C/\Lambda$. This allows us  to consider $(e_1, e_2)$ as a basis of $H_1(\C/\Lambda, \R)$: $e_1$ (resp. $e_2$) is the homology class of the closed curve $\left\lbrace tu, \mbox{ resp. }tv  : 0 \leq t \leq 1 \right\rbrace $ in $\C/\Lambda$.  
 
We compute  period coordinates in the basis $(e_1, e_2)$, and since $\omega=dz$,  the period coordinates of $\C /\Lambda$ are $z(u)$ and $z(v)$,  where $z(u)$ is the complex number whose affix is $u$. So, considering periods as vectors in $\R^2$, $\mathrm{GL}_2^+( \R )$ acts on periods linearly on the left.

If $\phi : \R^2 \longrightarrow \R^2$ is a similarity (a non-zero $\C$-linear map), then $\C /\Lambda$ and $\C /\phi (\Lambda )$ are bi-holomorphic. 
So, if we are interested in the projection of the Teichm\H{u}ller disk to the moduli space of complex structures, we might as well consider the quotient of $\mathrm{GL}_2^+( \R )$ by the subgroup $H$ of similarities, acting on the left. Beware that $\mathrm{GL}_2^+( \R )$ must then act on the quotient from the right.

Say two matrices $M_1, M_2 \in GL_2^+( \R )$ are equivalent under $H$ if there exists
$P=\begin{pmatrix}
a & c  \\ 
-c & a  
\end{pmatrix} \in H$, such that $PM_1=M_2$. We  use the following representative of an equivalence class: 
\[
\mbox{if } 
u=\begin{pmatrix}
a  \\ 
c 
\end{pmatrix} , 
v=\begin{pmatrix}
b  \\ 
d 
\end{pmatrix}, 
M=\begin{pmatrix}
a & b  \\ 
c & d  
\end{pmatrix} \in GL_2^+( \R ), \mbox{ then } 
\]
\[ 
\frac{1}{a^2+c^2}\begin{pmatrix}
a & c  \\ 
-c & a  
\end{pmatrix} 
\begin{pmatrix}
a & b  \\ 
c & d  
\end{pmatrix}=
\frac{1}{a^2+c^2}\begin{pmatrix}
1 & ab+cd  \\ 
0 & ad-bc  
\end{pmatrix} .
\]
Observe that the matrix 
$\frac{1}{a^2+c^2}\begin{pmatrix}
a & c  \\ 
-c & a  
\end{pmatrix}$,
viewed as a map from $\C$ to $\C$, is $z \mapsto \frac{1}{z(u)}z$. %The upside of this representation is that it has a clear geometrical interpretation: 
Geometrically, this means that the complex number which represents (the biholomorphsm class of) $\C/\Lambda$, where $\Lambda= \Z u \oplus \Z v$,  is just
$\left[ z(u) : z(v) \right] \in \C P^1$, so the canonical representative is 
$\left[ 1 : \frac{z(v)}{z(u)}\right] \in \C P^1$.

Since $(u,v)$ is a positive basis of $\R^2$, $ \frac{z(v)}{z(u)}$ lies in the upper half-plane $\Hyp^2$, so we have a bijection 
\[ \Psi : H\backslash GL_2^+( \R ) \longrightarrow \Hyp^2.
\]

%The downside is that since we are taking right cosets, we have to act on cosets from the right. 

Take an element of $H\backslash GL_2^+( \R )$, represented
by  a matrix $\begin{pmatrix}
1 & x  \\ 
0 & y  
\end{pmatrix}$, 
and an element $\begin{pmatrix}
a & b  \\ 
c & d  
\end{pmatrix}$
of $GL_2^+( \R )$. Then, recalling that $\mathrm{GL}_2^+( \R )$  acts on the quotient from the right, 
\[ 
M:= \begin{pmatrix}
1 & x  \\ 
0 & y  
\end{pmatrix}
\begin{pmatrix}
a & b  \\ 
c & d  
\end{pmatrix}=
\begin{pmatrix}
a+cx & b+dx  \\ 
cy & dy  
\end{pmatrix}.
\]
Set $z=x+iy$, so the affixes 
of the columns of the matrix $M$ are $a+cz$ and $b+dz$, then the map $\Psi$ takes the equivalence class of the matrix $M$ to the point $\frac{dz+b}{cz+a}$ in $\Hyp^2$. This means that the map $\Psi$ is equivariant, with respect to the  right action of $GL_2^+( \R )$ on  right cosets, and the right action of $GL_2^+( \R )$ on $\Hyp^2$ defined by 
\begin{align}
\Hyp^2 \times GL_2^+( \R ) & \longrightarrow \Hyp^2\\
\left(z, \begin{pmatrix}
a & b  \\ 
c & d  
\end{pmatrix}\right) 
& \longmapsto \frac{dz+b}{cz+a}.
\end{align}

\subsubsection{Hyperbolic metric vs. Teichm\H{u}ller metric}
So far we have identified the  Teichm\H{u}ller disk of the torus with the hyperbolic plane. We already know that the Veech group of the square torus is $\mathrm{SL}_2(\Z)$, so the
Teichm\H{u}ller curve of the torus is the \textit{modular curve} $\Hyp^2/ \mathrm{SL}_2(\Z)$ (more on which can be found in \cite{Arnoux}).

The hyperbolic plane, as a set, is not particularly interesting; what is interesting about it is the hyperbolic metric. We shall now see that Teichm\H{u}ller's metric, on the Teichm\H{u}ller space of the torus viewed as the hyperbolic plane, is precisely the hyperbolic metric, up to a multiplicative factor, which we shall ignore. 

We shall use the fact that the hyperbolic metric, up to a multiplicative factor, is the only Finsler metric on $\Hyp^2$ invariant under the action of $\mathrm{SL}_2(\R)$. This is because $\mathrm{SL}_2(\R)$ acts  transitively on the tangent bundle of $\Hyp^2$, so once we know the length of one tangent vector at some point, we know the length of every tangent vector, at every point.

By the equivariance of the map $\Psi$, and since we are not interested in the multiplicative factor, all we have to check is the invariance of the Teichm\H{u}ller metric under the action of $\mathrm{SL}_2(\R)$.

Now recall the Teichm\H{u}ller norm, on the fiber, over $X$, of the holomorphic quadratic differential bundle, is just the total area. 
Since $\mathrm{SL}_2(\R)$ preserves area, it preserves the 
Teichm\H{u}ller metric. Thus we have proved that the  Teichm\H{u}ller metric
is a multiple of the hyperbolic metric on $\Hyp^2$. This was originally observed in § 5 of \cite{Teichmueller}.

\subsubsection{Hyperbolic geodesics}
Let us investigate a little bit the relationship between hyperbolic geodesics and flat tori. The subgroup 
\[ 
\mathcal{G} := \left\lbrace g_t = \begin{pmatrix}
e^t & 0 \\
0 & e^{-t}
\end{pmatrix} : t \in \R \right\rbrace  
 \]
of $\mathrm{SL}_2(\R)$ is mapped by $\Psi$ to the geodesic 
\begin{align*}
\R & \longrightarrow \Hyp^2 \\
t & \longmapsto e^{-2t}i.
\end{align*}
For any $A=\begin{pmatrix}
	a & b  \\ 
	c & d  
\end{pmatrix} \in \mathrm{SL}_2(\R)$, 
\[ 
\mathcal{G}A=\left\lbrace  \begin{pmatrix}
e^t a & e^t b \\
e^{-t}c & e^{-t}d
\end{pmatrix} : t \in \R \right\rbrace  
 \]
is mapped by $\Psi$ to the geodesic 
\begin{align*}
\R & \longrightarrow \Hyp^2 \\
t & \longmapsto 
\frac{e^{2t}ab +e^{-2t}cd +i(ad-bc)}{e^{2t}a^2 +e^{-2t}c^2}
\end{align*} 
whose endpoints at infinity are $\frac{b}{a}$ and $\frac{d}{c}$. 

Here is a geometric interpretation of the endpoints. Consider the complex differential form 
$\eta := A^* dz= a.dx+b.dy +i(c.dx+d.dy)$.
 Then, for any $t \in \R$, 
\[ (g_t A)^* dz = e^{t}\Re \eta +i e^{-t} \Im \eta,
\]
 that is, $e^{t}\Re \eta +i e^{-t} \Im \eta$ is holomorphic on the torus 
$\C /\Lambda$, where $\Lambda= g_t.A.\Z^2$. So $\eta^2$ is the quadratic differential given by Teichm\H{u}ller's theorem (see \ref{TeichThm}), associated with the geodesic $\Psi g_t. $, whose real part is contracted along the geodesic, while its imaginary part is expanded. The endpoints at infinity of the geodesic are the (reciprocals of the opposites of the) slopes of the respective kernels of  the real and imaginary parts of $\eta$.

\subsubsection{Horocycles}
The subgroup 
\[ 
\mathcal{N} := \left\lbrace n_t = \begin{pmatrix}
1 & t \\
0 & 1
\end{pmatrix} : t \in \R \right\rbrace  
\]
of $\mathrm{SL}_2(\R)$ is mapped by $\Psi$ to the horocycle
\begin{align*}
\R & \longrightarrow \Hyp^2 \\
t & \longmapsto i+t.
\end{align*}
For any $ A=\begin{pmatrix}
	a & b  \\ 
	c & d  
\end{pmatrix} \in \mathrm{SL}_2(\R)$,
\[   
\mathcal{N}A=\left\lbrace  \begin{pmatrix}
 a+tc &  b+td \\
c & d
\end{pmatrix} : t \in \R \right\rbrace  
\]
is mapped by $\Psi$ to the horocycle 
\begin{align*}
\R & \longrightarrow \Hyp^2 \\
t & \longmapsto 
\frac{(a+tc)(b+td)+cd +i(ad-bc)}{(a+tc)^2+c^2}
\end{align*} 
whose endpoint at infinity is $\frac{d}{c}$, and whose apogee is $\frac{d}{c}+i\frac{ad-bc}{c^2}$. 

The pull-back, by $n_t A$, of the real 1-form $dy$, is $c.dx+d.dy$, because
\[ 
(n_t A)^*dy \begin{pmatrix}
x \\
y
\end{pmatrix}
= dy \left( n_tA\begin{pmatrix}
x \\
y
\end{pmatrix}\right) = dy\begin{pmatrix}
(a+tc)x+(b+td)y \\
cx+dy
\end{pmatrix}=cx+dy,
 \]
in particular $(n_t A)^*dy$ is constant along the horocycle 
$\Psi \left(\mathcal{N}A \right) $. The point at infinity of the horocycle is (minus the reciprocal of) the slope of the kernel of the constant 1-form. 

The reason for the annoying recurrence of "minus the reciprocal of..." is that the matrix 
$\begin{pmatrix}
d & b \\
c & a 
\end{pmatrix}$ 
is conjugate, by 
$\begin{pmatrix}
1 & 0 \\
0 & -1 
\end{pmatrix}$,
which acts on the hyperbolic plane by $z \mapsto -z$, to the inverse matrix
$A^{-1}=\begin{pmatrix}
d & -b \\
-c & a 
\end{pmatrix}$. 

If $\frac{d}{c} \in \Q$, then all the closed geodesics whose velocity vectors are tangent to the kernel of $c.dx+d.dy$ are closed, and of  length $c^2+d^2$. So  the torus  $n_t A  $, for $t \in \R$,  may be seen as a   cylinder of girth $c^2+d^2$, with the boundaries identified. When $t$ varies in $\R$, the identification varies but the closed geodesics remain fixed.

\subsubsection{Relationship between the asymptotic behaviour of hyperbolic geodesics, and the dynamic behaviour of  Euclidean geodesics} 
The set of slopes of closed geodesics in $\C/\Z^2$ is $\Q$. The Veech group $SL_2(\Z)$ acts transitively on $\Q$. Given a  hyperbolic geodesic $t \mapsto \gamma(t)$ in the modular surface $\Hyp^2 /SL_2(\Z)$, the following four points are equivalent:
\begin{itemize}
	\item $\gamma(t)$ goes to infinity in $\Hyp^2 /SL_2(\Z)$, when $t \rightarrow \infty$
	\item the  lifts of $\gamma$ to $\Hyp^2$ have a rational endpoint when $t \rightarrow \infty$
	\item all orbits of the vertical flow of the Abelian differential associated with $\gamma$ are closed
	\item $\gamma$ is invariant under the parabolic subgroup of $SL_2(\Z)$ generated by 
	$\begin{pmatrix}
	1 & 0 \\
	1 & 1
	\end{pmatrix}$.
\end{itemize}

\subsection{Examples of higher genus Teichm\H{u}ller disks}
\subsubsection{Three-squared surfaces}
The Veech group of the surface $St(3)$ is an index 3 subgroup  of $SL_2(\Z)$. It has a fundamental domain which is an ideal triangle in $\Hyp^2$, with ideal vertices $-1, 1, \infty$. The parabolic transformation $T^2$ identifies the two vertical boundaries, and the rotation $S$ identifies the two halves of the semi-circular boundary. The point $i$, which corresponds to the identity matrix, that is, to the surface $St(3)$ itself, is a singular point: a conical point with angle $\pi$. This is because it is invariant by $S$, which acts as the involution $z \mapsto -\frac{1}{z}$ on the hyperbolic plane. The other two three-squared translation surfaces, which correspond to the matrices $T$ and $ST$, are both (since the induce the same complex structure) mapped by $\Psi$, as defined in Subsection \ref{moduli space of the torus}, to the point $1+i$ (which is identified with $1-i$ by $T^2$).

The Teichm\H{u}ller curve of $St(3)$ has two cusps, one at $\infty$, and the other at $\pm 1$ (which become the same point in the Teichm\H{u}ller curve). Since $St(3)$ is square-tiled, the set of slopes of closed geodesics is $\Q$. The difference with the torus is that the Veech group does not act transitively on $\Q$, it has two orbits, one for each cusp. The orbit of $\infty$ consists of all fractions $\frac{p}{q}$, with $p \neq q \mod 2$. The orbit of $\pm 1$ consists of all fractions $\frac{p}{q}$, with $p = q=1 \mod 2$. 

The first cusp is called a two-cylinder cusp, while the second cusp is called a one-cylinder cusp, for the following reason. If a geodesic, in the Teichm\H{u}ller curve, escapes to infinity in the two-cylinder cusp, then any lift of this geodesic to $\Hyp^2$ corresponds, by Teichm\H{u}ller's theorem, to a  quadratic differential, the trajectories of whose real part are closed, and decompose $St(3)$ into two cylinders of closed geodesics. For instance, the vertical direction admits two cylinders, one obtained by identifying the green boundaries  in Figure \ref{St3, stage 1} is made of vertical closed geodesics of length 2, and the other,  obtained by identifying the blue boundaries  in Figure \ref{St3, stage 1} is made of vertical closed geodesics of length 1. If a geodesic, in the Teichm\H{u}ller curve, escapes to infinity in the one-cylinder cusp, then any lift of this geodesic to $\Hyp^2$ corresponds, by Teichm\H{u}ller's theorem, to a  quadratic differential, the trajectories of whose real part are closed,  pairwise homotopic, and of equal length. For instance, all geodesics in the 45° direction, except those that hit the singular point,  are closed and of length $3\sqrt{2}$.

See  \cite{CKM} for more on this Teichm\H{u}ller curve. Note that the three-squared Teichm\H{u}ller curve has genus zero, topologically it is a sphere minus two points, one for each cusp. The genus, and the number  of cusps, of Teichm\H{u}ller curves of square-tiled surfaces may be arbitrarily large. See \cite{HL} for the stratum $\mathcal{H}_2$; given a square-tiled surface in any  stratum, the genus and number of cusps of its Teichm\H{u}ller curve  may be computed with \cite{Delecroix}. 

\subsubsection{Regular polygons} 
We have seen that the Veech group of the regular polygons is generated by a parabolic element, and a rotation. The Veech group of the double $2n+1$-gon has a fundamental domain which looks like Figure \ref{disque_DP_octogone}, left. It has only one cusp, at infinity in Figure \ref{disque_DP_octogone}. This is reflected in the fact that the Veech group acts transitively on the set of directions of closed geodesics. It is an $n$-cylinder cusp, meaning that in any direction of closed geodesic, the surface decomposes into $n$ cylinders. 

The  point $(0,1)$ in Figure \ref{disque_DP_octogone}, left, is the image by $\Psi$ of the identity matrix, that is, the  double $2n+1$-gon. In the Teichm\H{u}ller curve it is singular, more precisely it is a conical point with angle $\frac{2\pi}{2n+1}$. This reflect the fact that the double $2n+1$-gon is invariant by a rotation of order $2n+1$.

The corner of coordinates $(\pm \cot \frac{\pi}{2n+1}, y)$,  is the image by $\Psi$, when $n=2$,  of the golden L (see \cite{DL}), and when $n >2$,  of stair-shaped polygons (see \cite{LM}). Those polygons are  invariant by $S$, which acts as the involution $z \mapsto -\frac{1}{z}$ on the hyperbolic plane, this is why the corner becomes a conical point with angle $\pi$ in the Teichm\H{u}ller curve.

The main difference between the Teichm\H{u}ller curve of the double $2n+1$-gon, and Teichm\H{u}ller curve of the regular $4n$-gon, is that the latter has two cusps, at $\infty$ and $(\pm \cot \frac{\pi}{4n}, 0)$ on Figure \ref{disque_DP_octogone}, right. See \cite{SU} for more on the Teichm\H{u}ller curve of the regular octagon, where the fundamental domain is drawn in the hyperbolic disk rather than in the hyperbolic half-plane. 

The Teichm\H{u}ller curve of the octagon, however, is different from the three-square Teichm\H{u}ller curve, in that both cusps are two-cylinders, meaning that if a geodesic, in the Teichm\H{u}ller curve, escapes to infinity in any cusp, then any lift of this geodesic to $\Hyp^2$ corresponds, by Teichm\H{u}ller's theorem, to a  quadratic differential, the trajectories of whose real part are closed, and decompose the octagon into two cylinders of closed geodesics. 

Then, you may ask, how do we distinguish between the cusps ? First, look at the horizontal direction in the octagon (Figure \ref{octagonL}). We leave it to the reader to check that the cylinder made up with the triangles $5$ and $6$ has height $1$ and length $1+\sqrt{2}$, while the cylinder made up with the triangles $1,2,3$ and $4$ has height $\frac{1}{\sqrt{2}}$ and length $2+\sqrt{2}$.

The \textbf{module}\index{module} of a cylinder is the ratio of its height to its length. So we see that for the horizontal direction, the ratio of the modules of the cylinders is $2$. What is important about this ratio of modules, is that it is invariant by $\mathrm{GL}_2^+ (\R)$, simply because linear maps, while they may not preserve length, preserve ratios of length of collinear segments. Now we leave it to the reader to check that in the direction of the short red diagonal, one cylinder is made up of the triangles $1$ and $2$, while the other is made of triangles $3,4,5,6$, and the ratio of their modules is not $2$. This means that no element of the Veech group takes 
the horizontal direction to the direction of the short red diagonal, so the Veech group, acting on the set of directions of closed geodesics, has at least two orbits, therefore there are at least two cusps.

 \begin{figure}[b!]
	%\begin{center}
	\begin{minipage}[b]{0.5\linewidth}
		\definecolor{uuuuuu}{rgb}{0.26666666666666666,0.26666666666666666,0.26666666666666666}
		\definecolor{ududff}{rgb}{0.30196078431372547,0.30196078431372547,1}
	\begin{tikzpicture}[line cap=round,line join=round,>=triangle 45,x=1cm,y=1cm, scale=1.1]
	\clip(-2.6,-1.19) rectangle (3.5,6);
	\draw [shift={(1.3763819204711736,0)},line width=1pt]  plot[domain=1.5707963267948966:2.5132741228718345,variable=\t]({1*1.7013016167040798*cos(\t r)+0*1.7013016167040798*sin(\t r)},{0*1.7013016167040798*cos(\t r)+1*1.7013016167040798*sin(\t r)});
	\draw [line width=1pt] (1.3763819204711736,1.7013016167040798) -- (1.3763819204711736,6.63);
	\draw [shift={(-1.3763819204711736,0)},line width=1pt]  plot[domain=0.6283185307179586:1.5707963267948966,variable=\t]({1*1.70130161670408*cos(\t r)+0*1.70130161670408*sin(\t r)},{0*1.70130161670408*cos(\t r)+1*1.70130161670408*sin(\t r)});
	\draw [line width=1pt] (-1.3763819204711736,1.7013016167040798) -- (-1.3763819204711736,6.63);
	%\begin{scriptsize}
	\draw [fill=ududff] (0,1) circle (2.5pt);
	\draw[color=black] (0.,.42) node {$(0,1)$};
	\draw [fill=uuuuuu] (1.3763819204711736,1.7013016167040798) circle (2pt);
	\draw[color=uuuuuu] (1.54,1.1) node {$(\cot \frac{\pi}{2n+1}, y)$};
	%\draw[color=black] (0.8,1.88) node {$e$};
	%\draw[color=black] (1.44,2.78) node {$f$};
	\draw [fill=uuuuuu] (-1.3763819204711736,1.7013016167040798) circle (2pt);
	\draw[color=uuuuuu] (-1.5,1.1) node {$(-\cot \frac{\pi}{2n+1}, y)$};
	%\draw[color=black] (-0.42,1.88) node {$g$};
	%\draw[color=black] (-1.3,2.78) node {$h$};
	%\end{scriptsize}
	\end{tikzpicture}	
	\end{minipage}
	%\end{center}
	\hfill
	%	\begin{center}
	\begin{minipage}[b]{0.4\linewidth}
		\definecolor{uuuuuu}{rgb}{0.26666666666666666,0.26666666666666666,0.26666666666666666}
		\definecolor{ududff}{rgb}{0.30196078431372547,0.30196078431372547,1}   
		\begin{tikzpicture}[line cap=round,line join=round,>=triangle 45,x=1cm,y=1cm]
		\clip(-2.9,-1.17) rectangle (4.6,6.65);
		\draw [line width=1pt] (2.414213562373095,0) -- (2.414213562373095,12.65);
		\draw [line width=1pt] (-2.414213562373095,0) -- (-2.414213562373095,12.65);
		\draw [shift={(1,0)},line width=1pt]  plot[domain=0:2.356194490192345,variable=\t]({1*1.414213562373095*cos(\t r)+0*1.414213562373095*sin(\t r)},{0*1.414213562373095*cos(\t r)+1*1.414213562373095*sin(\t r)});
		\draw [shift={(-1,0)},line width=1pt]  plot[domain=0:2.356194490192345,variable=\t]({-1*1.414213562373095*cos(\t r)+0*1.414213562373095*sin(\t r)},{0*1.414213562373095*cos(\t r)+1*1.414213562373095*sin(\t r)});
		%\begin{scriptsize}
		\draw [fill=ududff] (0,1) circle (2.5pt);
		\draw[color=black] (0.16,.44) node {$(0,1)$};
		\draw [fill=uuuuuu] (2.414213562373095,0) circle (2pt);
		\draw[color=uuuuuu] (2,-0.4) node {$(\cot \frac{\pi}{4n}, 0)$};
		\draw [fill=uuuuuu] (-2.414213562373095,0) circle (2pt);
		\draw[color=uuuuuu] (-1.9,-0.4) node {$(-\cot \frac{\pi}{4n}, 0)$};
		%\draw[color=black] (2.48,1.92) node {$f$};
		%\draw[color=black] (-2.34,1.92) node {$h$};
		%\draw[color=black] (1.74,1.66) node {$k$};
		%\draw[color=black] (-1.28,1.66) node {$k'$};
		%\end{scriptsize}
		\end{tikzpicture}
	\end{minipage}
	%	\end{center}	

	\caption{Fundamental domains for the Veech groups of the double $2n+1$-gon, on the left, and the $4n$-gon, on the right. The number $y$ is such that the circular boundary meets the vertical boundary perpendicularly. The angle at the point $(0,1)$ is $\frac{2\pi}{2n+1}$ (left), and  $\frac{\pi}{2n}$ (right).}
	
	\label{disque_DP_octogone} 
	
\end{figure}
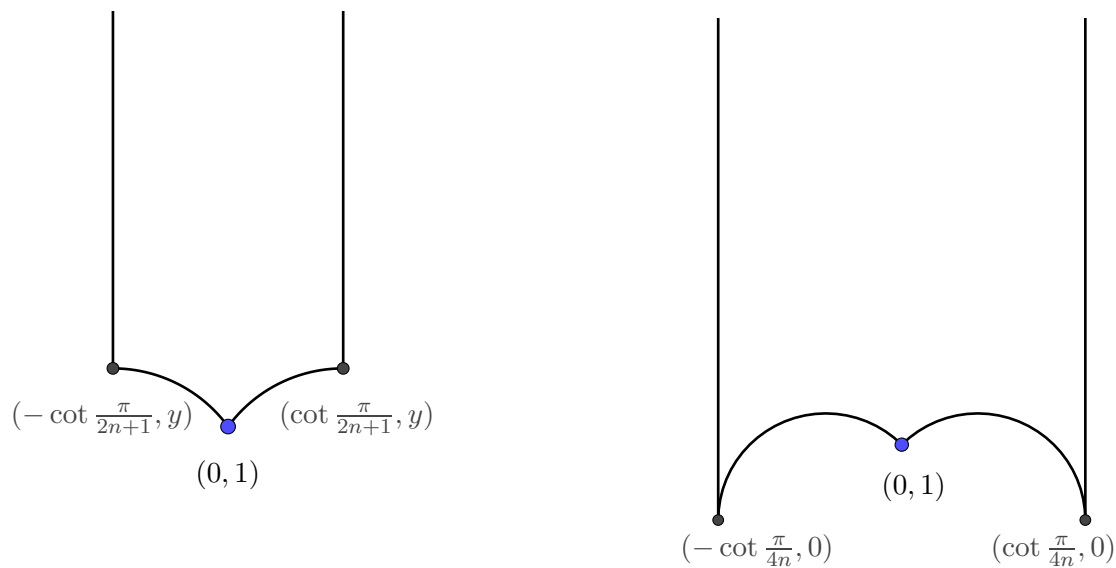

%citer Vorobets quelque part...
%%%%%%%%%%%%%%%%%%%%%%%%%%%%%%%%%%%%%%%%%%%%%%%%%%%%%
\section{Veech surfaces}
%%%%%%%%%%%%%%%%%%%%%%%%%%%%%%%%%%%%%%%%%%%%%%%%%%%
When facing a daunting dynamical system, and looking for closed invariant sets, the first thing to do is to look  for \textit{small} closed invariant sets: if your dynamical system is an $\R$-action (a flow), you are going to look for fixed points, or periodic orbits. In the case of an action by a larger group, you are going to look for closed (topologically speaking) orbits. If the $\mathrm{GL}_2^+ (\R)$-orbit of a holomorphic differential 
$(X, \omega)$ is closed in $\mathcal{H}_g$, we say $(X, \omega)$ is a 
\textbf{Veech surface\index{Veech surface}\index{surface!Veech}}.

%\subsection{Cusps of higher genus Teichm\H{u}ller disks, and parabolic elements in Veech groups}

%See \cite{Mukamel},  Section 2,

\subsection{The Smillie-Weiss theorem}
 What Veech was after was Veech groups with finite co-volume. Later it was proved in \cite{SW} that having a closed Teichm\H{u}ller disk is equivalent to having a finite-covolume Veech group. 

What follows, while very far from a proof, is meant to convey a bit of the idea.

First, let us assume that for some Abelian differential $(X,\omega)$, $\Gamma=SL(X, \omega)$ has finite co-volume in $\Hyp^2$. Then 
$\mathrm{SL}_2 (\R) /\Gamma$
consists of a compact part $K$ and finitely many cusps $C_1, \ldots, C_n$ (see \cite{S.Katok}). 
Therefore $K.(X, \omega)$ is compact, so it is closed in $\mathcal{H}_g$, and 
$C_i.(X, \omega)$ is closed in $\mathcal{H}_g$ unless there exists a sequence $A_n$ of matrices in $C_i$, which converges to infinity in $\mathrm{SL}_2 (\R)$, and such that $A_n.(X, \omega)$ does not go to infinity in $\mathcal{H}_g$. But recall that every cusp of $\Hyp^2 / \Gamma$ is stabilized by a parabolic element of the Veech group, which in turns yields a decomposition of $X$ into cylinders of closed geodesics. When $A_n$ goes to infinity, since $A_n \in C_i$, the length of the closed geodesics must go to zero, which means that $A_n.(X, \omega)$ leaves every compact set of $\mathcal{H}_g$.

Conversely, let let us assume that for some Abelian differential $(X,\omega)$,
the orbit $\mathrm{GL}_2^+ (\R).(X,\omega)$ is closed in $\mathcal{H}_g$. First, let us observe that since $\mathrm{GL}_2^+ (\R).(X,\omega)$ is closed in $\mathcal{H}_g$, the embedding of $\mathrm{SL}_2 (\R) /\Gamma$ into $\mathcal{H}_g$ is proper (the inverse image of any compact set is compact). This can be seen from the Magic Wand Theorem of \cite{EM}: since $\mathrm{GL}_2^+ (\R).(X,\omega)$ is closed, it must be an embedded submanifold, in particular, it is properly embedded. This means that the ends at infinity of $\mathrm{GL}_2^+ (\R).(X,\omega)$ in $\mathcal{H}_g$ are exactly the images in $\mathcal{H}_g$ of the ends at infinity of $\mathrm{SL}_2 (\R) /\Gamma$.

 Now let us ask ourselves, what kind of ends at infinity a closed Teichm\H{u}ller disk may have?

The first two examples that come to mind are a hyperbolic funnel (an end at infinity of a quotient of the hyperbolic disk by a hyperbolic element), and a cusp (the thin end a quotient of the hyperbolic disk by a parabolic element). One difference between the two is that the former has infinite volume, while the latter has not. Another difference is that geodesics that go to infinity in a cusp are all asymptotic to each other, while there exists a point $x$ in $\Hyp^2$, and an open interval $I$ of directions  such that any geodesic going from $x$ with direction in $I$ goes to infinity in the funnel. Now we apply Theorem \ref{thmKMS}: among those directions, there must be a uniquely ergodic one, let us call it $\theta$ (meaning that the vertical flow of $e^{i\theta}\omega$ is uniquely ergodic). But then, by Masur's Criterion, geodesics corresponding to uniquely ergodic directions do not go to infinity. This contradiction shows that the ends at infinity of a  Teichm\H{u}ller disk are cusps. 

Of course, we also have to prove that there are but finitely many cusps. Note that by \cite{HS2}, a Veech group may be infinitely generated, but in that case it is not the Veech group of a Veech surface. Here we draw the Magic Wand again: orbits closures are algebraic subvarieties of $\mathcal{H}_g$, so orbit closures of dimension one cannot have infinitely many cusps.  Now,  since cusps have finite volume, a closed Teichm\H{u}ller disk is the union of a compact part, and finitely many cusps, so it has finite volume.

\subsection{The Veech alternative}
The Veech alternative (see \cite{Veech89}) is another great example of sowing in the parameter space and reaping in the phase space. It says that every direction on a Veech surface is either uniquely ergodic, or completely periodic. This means that if $(X,\omega)$ is a Veech surface, for every $\theta \in \R$, the vertical flow of $(X,e^{i\theta}\omega)$ is either uniquely ergodic, or every orbit which does not hit a singularity is periodic. Beware this does not mean the flow itself is periodic, for the periods of the orbits may not be commensurable.

 Let us give some flavour of the proof. Assume  the vertical flow  on a Veech surface $(X,\omega)$ is not uniquely ergodic, so the orbit of $\omega$ under the Teichm\H{u}ller geodesic flow goes to infinity. Now, since $(X,\omega)$ is a Veech surface, the only ends at infinity of its Teichm\H{u}ller disk are cusps; and a geodesic which goes to infinity in a cusp must be invariant by the parabolic element which generates the cusp.

\subsubsection{Examples of Veech surfaces}

The first examples  found,  by Veech (\cite{Veech89}), are  the surfaces obtained by identifying opposite sides of a regular $2n$-gon.

 Once people got interested in Veech surfaces, it was immediately realized (see \cite{GJ}) that square-tiled surfaces are Veech surfaces, just because their Veech groups are (conjugate to) subgroups of $\mathrm{SL}_2(\Z)$, which has finite co-volume. 
  
In general, finding the Veech surfaces in a given stratum is a hard problem. However, in genus two,  it was solved by McMullen in \cite{Mc, Mc03, Mc05bis, Mc06, Mc06b, Mc07}, see Subsection \ref{homologicalDetour}.

\section{Classification of orbits in genus two}\label{genus2} 
The Magic Wand, magic as it is, does not tell us everything about the orbit closures. For instance, in most strata, all we know is that
\begin{itemize}
	\item by ergodicity, almost every orbit is dense
	\item the union of all orbits of square-tiled surfaces (each of which is closed) is  dense.	
\end{itemize}
When are there orbit closures of intermediate dimension? Are there closed orbits (i.e. Veech surfaces) besides those of square-tiled surfaces? Speaking of square-tiled surfaces, how do they distribute in Teichm\H{u}ller disks?   For instance, could it be that each square-tiled surface is  the only square-tiled surface in its Teichm\H{u}ller disk?

The answers  are known in genus two: we have a complete list of non-square-tiled Veech surfaces (\cite{Calta}, \cite{Mc}), a complete list of 
orbit closures of intermediate dimension (\cite{Mc07}), and, in the stratum $\mathcal{H}(2)$,  a complete list of orbits of \textit{primitive} (more on this later) square-tiled surfaces (\cite{HL}).
\subsection{Square-tiled surfaces}

Observe that a square may be subdivided into little rectangles, and then we may use some matrix in $\mathrm{GL}_2(\R)$ to square the rectangles, so a square-tiled surface may be square-tiled in many ways. 

Following \cite{HL}, we say a square-tiled surface is primitive if it is not obtained from another square-tiled surface by subdivising squares into rectangles and then applying a matrix to square the rectangles. 

Then \cite{HL} says that for $n >4$, there are exactly two Teichm\H{u}ller disks of primitive $n$-square surfaces $\mathcal{H}(2)$. For $n=3,4$, there is only one Teichm\H{u}ller disk of primitive $n$-square surfaces. Each Teichm\H{u}ller disk contains finitely many primitive $n$-square surfaces, the exact number is given by the index of the Veech group in $\mathrm{SL}_2(\Z)$. In genus two it turns out that this index is never $1$, however, there exist square-tiled surfaces of higher genus whose Veech group is $\mathrm{SL}_2(\Z)$ (see ``Eierlegende Wollmilchsau" and ``Ornythorynque" in \cite{FM}). Even in genus two the situation is not fully understood, for instance, we don't know how many Teichm\H{u}ller disks of primitive $n$-square surfaces there are in $\mathcal{H}(1,1)$.

\subsection{A homological detour }\label{homologicalDetour}
\subsubsection{The tautological subspace}
Since the complex-valued 1-form $\omega$ is holomorphic, the real-valued 1-forms $\Re (\omega)$ and $\Im (\omega)$ are closed, by the Cauchy-Riemann relations. We denote $S$ the 2-dimensional subspace of $H^1(X,\R)$ generated by the cohomology classes of $\Re (\omega)$ and $\Im (\omega)$.

Recall that $H^1(X,\R)$ has a symplectic structure: let $\left[ X \right] \in H^2(X,\R)$ be the fundamental class, then 
\[ 
\forall a, b \in H^1(X,\R), \exists c \in \R, a \wedge b = c \left[ X \right].
 \]
 The symplectic 2-form is then defined as $(a,b) \mapsto c$.
 It is Poincaré dual to the intersection form (the bilinear form $\mbox{Int}$ on $H_1(X,\Z)$ such that $\mbox{Int}(h,k)$ is the total intersection, counted with signs, of any two representatives of $h$ and $k$ in transverse position) on $H_1(X,\R)$, that is, if 
 $P : H_1(X,\R) \longrightarrow H^1(X,\R)$ is the Poincaré map,  then 
 \[ 
 \forall h,k \in H_1(X,\R), P(h)\wedge P(k) = \mbox{Int} (h,k) 
 \left[ X \right].
  \]
  The subspace $S$ is symplectic, meaning that the 2-form $.\wedge .$, restricted to $S \times S$, is non-degenerate.
 Indeed, recall that in an appropriate chart, $\omega=dz=dx+idy$, so 
 \[ \Re (\omega)\wedge \Im (\omega) = dx\wedge dy. \]  
  The symplectic orthogonal $S^{\perp}$ of $S$ is the set of cohomology classes
  $c$ such that $c \wedge \Re (\omega)= c \wedge \Im (\omega)=0$. We leave it as an exercise for the reader that in a symplectic vector space, the symplectic orthogonal of a symplectic subspace is also symplectic.

  Now let us take a look at how the Veech group $\Gamma$ acts on $S$. Let $\gamma$ be an element of $\Gamma$, and let 
  $ \begin{pmatrix}
  a & b  \\ 
  c & d  
  \end{pmatrix} 
  \in \mathrm{SL}_2(\R)$ be the  matrix of the derivative of $\gamma$, seen as a diffeomorphism of $X$ minus its singular set. Let $(x,y)$ be a point in $X$, in local coordinates, and let $(u,v)$ be a tangent vector to $X$ at $(x,y)$. Then 
  \begin{align*}
  \gamma^* (dx)_{(x,y)}.(u,v)&=dx(d\gamma_{ (x,y)}.(u,v))= dx(au+bv, cu+dv)
  =au+bv \\
  \gamma^* (dy)_{(x,y)}.(u,v)&=dy(d\gamma_{ (x,y)}.(u,v))= dy(au+bv, cu+dv)
  =cu+dv
  \end{align*}
so the action of  $\Gamma$ by pull-back on $S$, endowed with the basis 
$\left[ dx\right] , \left[ dy\right] $, is given by the same matrix as the linear action of $\gamma$ on $\R^2$. For this reason $S$ is sometimes called the \textbf{tautological subspace\index{tautological subspace}} of $H^1(X,\R)$.

 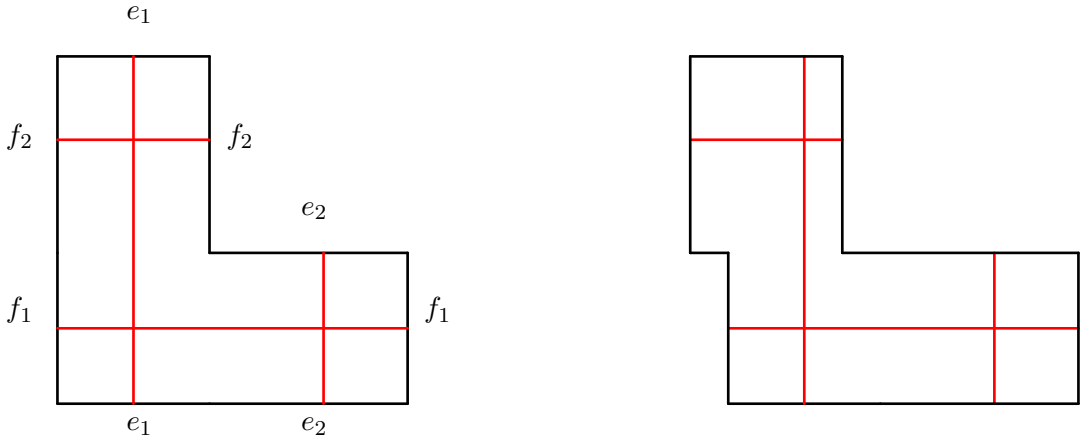
\begin{figure}[b!]
 	%\begin{center}
 		\begin{minipage}[b]{0.5\linewidth}
 			\begin{tikzpicture}[line cap=round,line join=round,>=triangle 45,x=1cm,y=1cm]
 			\clip(-1,-1) rectangle (8.4,6);
 			\draw [line width=1pt] (0,0)-- (2,0);
 			\draw [line width=1pt] (2,0)-- (4.60555127546399,0);
 			\draw [line width=1pt] (4.60555127546399,0)-- (4.60555127546399,2);
 			\draw [line width=1pt] (4.60555127546399,2)-- (2,2);
 			\draw [line width=1pt] (2,2)-- (2,4.60555127546399);
 			\draw [line width=1pt] (2,4.60555127546399)-- (0,4.60555127546399);
 			\draw [line width=1pt] (0,4.60555127546399)-- (0,2);
 			\draw [line width=1pt] (0,2)-- (0,0);
 			%\begin{scriptsize}
 			\draw[color=black] (1.08,-0.3) node {$e_1$};
 			\draw[color=black] (3.38,-0.3) node {$e_2$};
 			\draw[color=black] (5,1.24) node {$f_1$};
 			\draw[color=black] (3.38,2.56) node {$e_2$};
 			\draw[color=black] (2.4,3.54) node {$f_2$};
 			\draw[color=black] (1.08,5.16) node {$e_1$};
 			\draw[color=black] (-0.5,3.54) node {$f_2$};
 			\draw[color=black] (-0.5,1.24) node {$f_1$};
 				\draw [line width=1pt, color=red] (0,1)-- (4.60555127546399,1);
 			\draw [line width=1pt, color=red] (1,0)-- (1,4.60555127546399);
 			\draw [line width=1pt, color=red] (0,3.5)-- (2,3.5);
 			\draw [line width=1pt, color=red] (3.5,0)-- (3.5,2);

 			\end{tikzpicture}
 		\end{minipage}
 	%\end{center}
 \hfill
 %	\begin{center}
 		\begin{minipage}[b]{0.4\linewidth}   
 			\begin{tikzpicture}%[scale=1.3]
 		[line cap=round,line join=round,>=triangle 45,x=1cm,y=1cm]
 		\clip(-1,-1) rectangle (8.4,6);
 		\draw [line width=1pt] (0,0)-- (2,0);
 		
 		\draw [line width=1pt, color=red] (0,1)-- (4.60555127546399,1);
 		\draw [line width=1pt, color=red] (1,0)-- (1,4.60555127546399);
 			\draw [line width=1pt, color=red] (-0.5,3.5)-- (1.5,3.5);
 		\draw [line width=1pt, color=red] (3.5,0)-- (3.5,2);
 		
 			\draw [line width=1pt] (2,0)-- (4.60555127546399,0);
 		\draw [line width=1pt] (4.60555127546399,0)-- (4.60555127546399,2);
 		\draw [line width=1pt] (4.60555127546399,2)-- (1.5,2);
 		\draw [line width=1pt] (1.5,2)-- (1.5,4.60555127546399);
 		\draw [line width=1pt] (1.5,4.60555127546399)-- (-0.5,4.60555127546399);
 		\draw [line width=1pt] (-0.5,4.60555127546399)-- (-0.5,2);
 		\draw [line width=1pt] (-0.5,2)--(0,2);
 		\draw [line width=1pt] (0,2)-- (0,0);
 		%\begin{scriptsize}
 		%\draw[color=black] (1.08,-0.3) node {$e_1$};
 		%\draw[color=black] (3.38,-0.3) node {$e_2$};
 		%\draw[color=black] (5,1.24) node {$f_1$};
 		%\draw[color=black] (3.38,2.56) node {$e_2$};
 		%\draw[color=black] (2.4,3.54) node {$f_2$};
 		%\draw[color=black] (1.08,5.16) node {$e_1$};
 		%\draw[color=black] (-0.5,3.54) node {$f_2$};
 		%\draw[color=black] (-0.5,1.24) node {$f_1$};
 			\end{tikzpicture}
 		\end{minipage}
 %	\end{center}	

	\caption{On the left, the surface $L(a,1)$. The sides $e_1$ and $f_1$ have length $1$, while the sides $e_2$ and $f_2$ have length $a-1$. The surface decomposes into two  cylinders of horizontal closed geodesics, the lower one of width $a$ and height $1$, and the upper one of width $1$ and height $a-1$. 
	On the right, the top rectangle in the surface $L(a,1)$ has been translated to the left. Sides are still identified by vertical or horizontal translations. The resulting surface lies in the stratum $\mathcal{H}(1,1)$, like $St(4)$, but  it has the same absolute periods as $L(a,1)$. The closed curves with respect to which the periods are computed are drawn in red.} \label{La1}

\end{figure}

\subsubsection{The trace field}  
 The \textbf{trace field\index{trace field}} $K(\Gamma)$ of a subgroup $\Gamma$ of $\mathrm{SL}_2 (\R)$ is the subfield of $\R$ generated by the traces of the elements of $\Gamma$. We are going to see that 
 \begin{lemma}
 	When $\Gamma$ is the Veech group of some translation surface of genus two,  $K(\Gamma)=\Q\left[ \sqrt{d}\right] $ for some $d \in \N^*$.
 \end{lemma} 
Of course, for most surfaces, the Veech group is trivial, so the trace field is $\Q$. The trace field is most interesting when the Veech group is large, especially for Veech surfaces.
 
 It is proved in \cite{GJ} that for a Veech surface of any genus, having $\Q$ as its trace field  is equivalent to being in the Teichm\H{u}ller disk of a square-tiled surface.  In the genus two case, this means that if $d$ is a square, then $X$ is parallelogram-tiled (that is, $X$ lies in the Teichm\H{u}ller disk of a square-tiled surface).
 
 Take $\gamma \in \Gamma$, and consider the induced automorphism $\gamma^*$ of  $H^1(X,\R)$. Since $\gamma^*$ is induced by a homeomorphism of $X$, it preserves the integer lattice $H^1(X,\Z)$ of $H^1(X,\R)$, and the same goes for $(\gamma^*)^{-1}$. Thus  the endomorphism $\gamma^* + (\gamma^*)^{-1})$ of $H^1(X,\R)$ may be represented by an integer matrix, in particular its eigenvalues are algebraic numbers. 
 
 But recall that $\gamma^*$, restricted to $S$, is just $\gamma$, so $\gamma^* + (\gamma^*)^{-1}$, restricted to $S$, is just $\gamma + \gamma^{-1}= Tr(\gamma).Id$. Therefore the trace of $\gamma$ is a double eigenvalue of 
 $\gamma^* + (\gamma^*)^{-1}$, so the square of the minimal polynomial $\chi$ of  $Tr(\gamma)$ divides the characteristic polynomial of $\gamma^* + (\gamma^*)^{-1}$, hence $\chi$ has degree at most two. 
 
 We have just proved that for every $\gamma \in \Gamma$, there exists $d \in \N^*$ such that $Tr(\gamma)\in \Q\left[ \sqrt{d}\right] $. The same argument applies to any linear combination of traces of elements of $\Gamma$. We claim that $d$ may be chosen independantly of $\gamma$. If it were not the case, say, $Tr(\gamma) \in \Q\left[ \sqrt{d}\right] $ and $Tr(\gamma') \in \Q\left[ \sqrt{d'}\right] $ for some $\gamma' \neq \gamma$ and $d' \neq d$, then we could find a non-quadratic element in $K(\Gamma)$ (think, for instance, that $\sqrt{2}+\sqrt{3}$ is not quadratic). Therefore, $K(\Gamma)=\Q\left[ \sqrt{d}\right] $ for some $d \in \N^*$.
 
 \textbf{Example.} Consider the surface $L(a,1)$ on the left of  Figure \ref{La1}, where 
 $a=\frac{1+\sqrt{d}}{2}$. For $d=5$, this surface lies in the Teichm\H{u}ller curve of the double pentagon, it is usually called the Golden L (see \cite{DFT, DL}). For $d=2$, it can be shown to lie in the Teichm\H{u}ller curve of the regular octagon (see \citen{SU}).
 
 The matrix $A := \begin{pmatrix}
 1 & 4a  \\ 
 0 & 1  
 \end{pmatrix} $
 acts on the lower cylinder, whose height is $1$ and whose width is $a$,  as the $4$-th power of the horizontal Dehn twist, because 
 \[A.\begin{pmatrix}
 0  \\ 
  1  
 \end{pmatrix} = \begin{pmatrix}
 0  \\ 
 1  
 \end{pmatrix} +4\begin{pmatrix}
 a  \\ 
  0 
 \end{pmatrix},
 \]
 and it acts on the upper cylinder,  whose height is $a-1$ and whose width is $1$,  as the $(d-1)$-th power of the horizontal Dehn twist, because  
 \[
A.\begin{pmatrix}
 0  \\ 
 a-1  
 \end{pmatrix} = \begin{pmatrix}
 4a(a-1)  \\ 
 a-1  
 \end{pmatrix} = \begin{pmatrix}
 0  \\ 
 a-1  
 \end{pmatrix}
 +(d-1)\begin{pmatrix}
 1  \\ 
 0  
 \end{pmatrix}.
  \]

 Thus the matrix $A$ lies in the Veech group of $L(a,1)$. The same arguments apply to the vertical cylinders, so the matrix  $\begin{pmatrix}
 1 & 0  \\ 
 4a & 1  
 \end{pmatrix} $
 also lies in the Veech group, and so does their product 
$ \begin{pmatrix}
 1 +16a^2 & 4a  \\ 
 4a & 1  
 \end{pmatrix}$, 
 whose trace is $2+16a^2=2+4d+8\sqrt{d}$, hence the trace field of $L(a,1)$ is 
 $\Q\left[ \sqrt{d}\right] $.

 \subsubsection{The Jacobian torus}
 For any manifold $X$, the \textbf{Jacobian torus\index{Jacobian torus}\index{Jacobian!torus}} $Jac(X)$ of $X$ is the quotient of $H^1(X,\R)$ by the integer lattice $H^1(X,\Z)$. When $X$ is a Riemann surface, $H^1(X,\R)$ identifies with the space of Abelian differentials, by sending an Abelian differential to the cohomology class of its real part, and $H^1(X,\Z)$ acts by translations, which are holomorphic, so $Jac(X)$ is actually a complex manifold.
 
 The Jacobian torus also has a symplectic structure, given by the wedge pairing, and the symplectic structure is compatible   with the complex structure, in the following way: any complex subspace in $H^1(X,\R)$ is also a symplectic subspace. If $E$ is any complex subspace of $H^1(X,\R)$, then the symplectic orthogonal $E^{\perp}$ is also a complex subspace. 
 For instance, the canonical subspace $S$ and its symplectic orthogonal are both complex subspaces of $H^1(X,\R)$.
 
 An \textbf{endomorphism} of $Jac(X)$ is a $\C$-linear endomorphism of $H^1(X,\R)$, which preserves the integer lattice, so it quotients to a self-map of $Jac(X)$, and is self-adjoint with respect to the symplectic form. Such a thing actually exists, here is an example.  
 
 Assume $X$ is a translation surface of  genus two, and $\gamma$ is an element of the Veech group of $X$, such  that its trace is irrational (hence quadratic since the genus is two). Think of the example we have just seen  on the surface $L(a,1)$.

  Then both $\gamma^*$ and $(\gamma^*)^{-1}$ are $\R$-linear endomorphisms of $H^1(X,\R)$ which preserve the integer lattice, so the same goes for $\gamma^*+(\gamma^*)^{-1}$. Recall that 
 $T:=\gamma^* + (\gamma^*)^{-1}$, restricted to $S$, is just $Tr(\gamma).Id$, so $t := Tr(\gamma)$ is a double eigenvalue of  $T$. Thus the Galois conjugate $\bar{t}$ of 
 $t$ is also a double eigenvalue of  $T$. But both $\gamma^*$ and $(\gamma^*)^{-1}$ are symplectic, so, since they preserve $S$, they must preserve the symplectic orthogonal $S^{\perp}$. Therefore, since the dimension of $S^{\perp}$ is two (recall that the genus of $X$ is two), we get that $T$, restricted to $S^{\perp}$, is $\bar{t} Id$. This proves that $T$ is $\C$-linear, since it may be expressed by the matrix  $ \begin{pmatrix}
 t & 0  \\ 
 0 & \bar{t} 
 \end{pmatrix} $
 in some complex basis. 
 
 This is the part that breaks down in higher genus, and the reason why McMullen's results in genus two have been dubbed miraculous (see \cite{Zorich}). Let us see exactly how it breaks down. Assume $X$ is a translation surface of genus $g>2$, with a trace field of degree $g$ (the maximal possible degree). Let  $\gamma$ be an element of the Veech group of $X$, whose trace is an algebraic integer of degree $g$. Then the $g-1$ Galois conjugates of $Tr(\gamma)$ are double eigenvalues of $T=\gamma^* + (\gamma^*)^{-1}$, so $S^{\perp}$ decomposes as a direct sum of  2-dimensional real subspaces, on each of which $T$ acts by homothety. The trouble is, those 
 2-dimensional real subspaces have no reason to be complex subspaces, even though $S^{\perp}$ is a complex subspace, so we cannot conclude that $T$ is $\C$-linear. In some cases (see \cite{LN}) $T$ happens to be $\C$-linear, so McMullen's methods may be applied.
 
 To prove that $T$ is an endomorphism of $Jac(X)$, we still have to prove that 
$T$ is self-adjoint with respect to the symplectic form. Since $\gamma$ is a homeomorphism of $X$, $\gamma^*$ preserves the symplectic form, so the adjoint of $\gamma^*$ is $(\gamma^*)^{-1}$. Similarly the adjoint of  $(\gamma^*)^{-1}$
is $\gamma^*$, so  $T=\gamma^* + (\gamma^*)^{-1}$ is self-adjoint. 
 
\subsubsection{Real multiplication}
Let $X$ be a translation surface of genus two, with trace field 
$\Q\left[ \sqrt{d}\right] $ for some non-square $d \in \N$. Let $\gamma$ be an element of the Veech group of $X$, with irrational trace, and let 
$T=\gamma^* + (\gamma^*)^{-1}$. 

The  endomorphisms of $Jac(X)$ form a ring, and the subring generated by $Id$ and $T$ is isomorphic to a finite index subring of the ring of integers of 
$\Q\left[ \sqrt{d}\right] $, the isomorphism being the trace of the restriction to $S$. Of course there is another isomorphism, which is the trace of the restriction to $S^{\perp}$. 

For any translation surface, when the ring of endomorphisms of $Jac(X)$ has a subring isomorphic to (some finite index subring of) the integer ring of some number field $K$, we say $Jac(X)$ has 
\textbf{real multiplication\index{real multiplication}} by $K$. So, what we have proved so far is 

\begin{lemma}
The Jacobian of any non-square-tiled Veech surface of genus two has real multiplication by some 
$\Q\left[ \sqrt{d}\right] $, with $d \in \N$ non-square.	
\end{lemma}  

For instance, the double pentagon has real multiplication by $\Q\left[ \sqrt{5}\right] $, while the octagon has real multiplication by $\Q\left[ \sqrt{2}\right] $. More generally, for any non-square $d \in \N$, the surface $L(a,1)$, with $a=\frac{1+\sqrt{d}}{2}$, has real multiplication by $\Q\left[ \sqrt{d}\right] $.

Now, what is the point of all this, you may ask? The point is that there is a way to know which surfaces have Jacobians with real multiplication, and the stroke of genius was to think of looking at the problem this way.

Let us denote $W_d$ the projection to the moduli space $\mathcal{M}_2$ of the set of translation surfaces of genus two whose Jacobian  has real multiplication by 
$\Q\left[ \sqrt{d}\right] $.

\begin{theorem}
	For $d \in \N$ non-square, $W_d$ is a two-dimensional algebraic subvariety of $\mathcal{M}_2$.
\end{theorem}
\begin{proof}
	We will focus on the dimension, and refer the reader to \cite{Mc03} for the algebraicity statement. Let us forget about Jacobians for a moment, and consider the set of all complex tori of dimension two (also known, poetically, as principally polarized Abelian varieties). This set may be viewed as the set of all complex structures on $\R^4/\Z^4$, which are compatible with the symplectic structure on $\Z^4$, that is, denoting $(.,.)$ the canonical scalar product and $. \wedge .$ the symplectic structure, for any $a,b$, $a \wedge b = (a,ib)$. 
	
	Among all tori, let us consider those whose endomorphsim ring contains a copy of (a finite index subring of) the integer ring of $\Q\left[ \sqrt{d}\right] $. Let $\phi$ be a generator of said subring, which then is the $\Z$-module of rank two $\Z \left[ \phi \right] = \left\lbrace a +b \phi : a, b \in \Z \right\rbrace $. The embedding of $\Z \left[ \phi \right]$ into the ring of endomorphisms of $\Z^4$ makes $\Z^4$ a 
	$\Z \left[ \phi \right]$-module, necessarily of rank two. Tensorizing by $\Q$,  we may view $\Q^4$ as a two-dimensional vector space over 
	$\Q\left[ \sqrt{d}\right] =\Q \left[ \phi \right]$. 
	
	Let  $\bar{\phi}$ be the Galois conjugate of $\phi$, so  $P_{\phi}=(X-\phi)(X-\bar{\phi})$ is the minimal polynomial of $\phi$, and let $T_{\phi}$ be the endomorphism of $\Q^4$ associated with $\phi$. Then $P_{\phi}(T_{\phi})=0$. Since $T$ preserves $\Z^4$,  $T_{\phi} \neq \phi Id$, so $T_{\phi}$ has both $\phi$ and  $\bar{\phi}$ as eigenvalues, which entails that, as an endomorphism of a two-dimensional  
	$\Q \left[ \phi \right]$-vector space, $T_{\phi}$ is diagonalizable. Both eigenspaces have dimension one over 
	$\Q \left[ \phi \right]$, which is dimension two over $\Q$. Since $T$ is self-adjoint, the eigenspaces are mutually orthogonal, and they are symplectic, for if they were not, since their dimension is two, then the restriction of the symplectic form to either of them would vanish, and since they are orthogonal, the symplectic form on $\Z^4$ would be zero, a contradiction. 
	
	The complex structure on $\R^4/\Z^4$ is then completely determined by the requirement that the eigenspaces be complex lines, and by the choice of a complex structure for each eigenspace, compatible with the symplectic structure. Each choice is determined by one complex parameter (recall our discussion of the Teichm\H{u}ller space, which is also the Teichm\H{u}ller disk,  of the torus), so the complex structure on $\R^4/\Z^4$ is determined by two complex parameters. 
	
	Next, observe that almost all complex tori are Jacobians: essentially, those who are not Jacobians of regular surfaces,  are Jacobians of degenerate surfaces, for instance a bouquet of two 1-dimensional tori. So the set of Jacobians of regular surfaces has dimension two over $\C$.
	
\end{proof}
McMullen's result actually says much more: the image, in  the Teichm\H{u}ller space $\mathcal{T}_2$, of the set of all translation surfaces of genus two whose Jacobian has real multiplication by $\Q\left[ \sqrt{d}\right] $, for $d \in \N$ non-square, is an holomorphically  embedded copy of $\Hyp^2 \times \Hyp^2$. Its image $W_d$ in the moduli space $\mathcal{M}_2$ is called a 
\textbf{Hilbert modular surface\index{Hilbert modular surface}\index{Hilbert!modular surface}}. 

Now we'd like to understand how  Hilbert modular surfaces behave with respect to the $\mathrm{SL}_2(\R)$-action, and the strata. 
\begin{lemma}\label{invarSL2}
	For any non-square $d \in \N$, $W_d$  is $\mathrm{SL}_2(\R)$-invariant.
\end{lemma}
\begin{proof}
Let $(X,\omega)$ be a translation surface of genus two whose Jacobian  has real multiplication by $\Q\left[ \sqrt{d}\right] $, for $d \in \N$ non-square,
let $T$ be an endomorphism of the Jacobian of $X$, and let $\gamma$ be an element of $\mathrm{SL}_2(\R)$. Then $(\gamma^*)^{-1}$ conjugates $T$ with an endomorphism of the Jacobian of $\gamma.(X,\omega)$, so $\gamma.(X,\omega)$ 
has real multiplication by $\Q\left[ \sqrt{d}\right] $.
\end{proof}
Recall that we  denote by $P$ the canonical projection 
$\mathcal{H}_2 \longrightarrow \mathcal{M}_2$. 
\begin{lemma}\label{interH(2)}
	The intersection  $P(\mathcal{H}(2)) \cap W_d$ is not empty.

\end{lemma}
\begin{proof}
	Consider the surface $L(a,1)$, with $a=\frac{1+\sqrt{d}}{2}$.
\end{proof}
\begin{lemma}\label{interH(1,1)}
	The intersection $P(\mathcal{H}(1,1) )\cap W_d$ is not empty.
\end{lemma}
\begin{proof}
	Consider the surface on the right of Figure \ref{La1},   obtained from $L(a,1)$ by shifting the upper cylinder to the left. This surface has the same absolute periods as $L(a,1)$, hence it has the same Jacobian, since Jacobians ignore relative periods. Hence it lies in $W_d$.
\end{proof}
By Lemmata \ref{invarSL2} and \ref{interH(2)}, $P(\mathcal{H}(2)) \cap W_d$ is a reunion of $\mathrm{SL}_2(\R)$-orbits, so its dimension is at least one. If its dimension were two, since two is the dimension of $W_d$, and the latter is connected   as a quotient of  $\Hyp^2 \times \Hyp^2$, then $\mathcal{H}(2) \cap W_d$ would be the whole of $W_d$. But this is impossible by Lemma \ref{interH(1,1)}. So the dimension of $P(\mathcal{H}(2)) \cap W_d$ is one, and since it is an algebraic subvariety of $W_d$, it must be a finite union of closed orbits. There, we have found Veech surfaces in $\mathcal{H}(2)$, with trace field  $\Q\left[ \sqrt{d}\right] $ for any non-square $d \in \N$, and proved that there are only finitely many of them, for each $d$. We have also proved that the surface $L(a,1)$ is a Veech surface for  $a=\frac{1+\sqrt{d}}{2}$, just by computing its trace field. 

Again, McMullen's results actually say much more: the intersection $P(\mathcal{H}(2)) \cap W_d$ contains exactly one orbit, except when $d=1 \mod 8$, $d \neq 9$, in which case there are exactly two (see \cite{Mc}). The intersection $P(\mathcal{H}(1,1)) \cap W_d$ contains no closed orbit, except when $d=5$, in which case it contains exactly one, the Teichm\H{u}ller curve of the regular decagon (see \cite{Mc05bis, Mc06}). An orbit closure is either the orbit itself, in which case it is  a Teichm\H{u}ller curve, of which we have a complete list,  or $W_d$ for some $d$, or a whole stratum (see \cite{Mc07}). Note that in genus two, both strata project surjectively to the moduli space $\mathcal{M}_2$ (see \cite{FK}, III.7.5, Corollary 1), so an orbit which is dense in its stratum projects to a dense subset of $\mathcal{M}_2$.

\section{What is known in higher genus?}
First, a bit of vocabulary: a Veech surface is said to be \textbf{geometrically primitive\index{geometrically primitive}\index{primitive!geometrically}} if it is not a ramified cover of a Veech surface of lower genus. For instance, square-tiled surfaces, other than tori, are not geometrically primitive since they cover the torus. A Veech surface is said to be \textbf{algebraically primitive\index{algebraically primitive}\index{primitive!algebraically}} if its trace field has maximal possible degree, which is the genus of the surface. Veech's family (the regular polygons) and McMullen's family in genus $2$ are both algebraically and geometrically primitive, as well as Ward's (see \cite{Ward}). 

The general philosophy of Veech surface hunters seems to be that algebraically and geometrically primitive Veech surfaces are scarce, and you should expect at most finitely many of them in a given stratum of genus $>2$. In \cite{MW} this very statement is proved, for the minimal stratum of each genus $>2$.

In \cite{BM} a family of Veech surfaces is found, which generalizes Veech's family.
In \cite{Mc06b}, \cite{LN}, the authors find a way around the fact that McMullen's techniques do not generalize in higher genus, and discover a family of Veech surfaces in genera $3$ and $4$, which generalizes McMullen's family, in the sense that their trace fields are quadratic (so they are not algebraically primitive).
In  \cite{MMW},  \cite{EMMW}, another family of Veech surfaces in genus  $4$ is found,  again with quadratic trace fields.

\bigskip

\noindent \textbf{Adress}: 

\noindent Daniel Massart, Institut Montpelliérain Alexander Grothendieck, Universit\'e de  Montpellier, France 

\noindent email: daniel.massart@umontpellier.fr

\end{document}